\documentclass[a4paper,11pt]{article}
\usepackage[latin1, utf8]{inputenc}
\usepackage[english]{babel}
\usepackage{amsmath,amsthm}
\usepackage{amsfonts}
\usepackage{amssymb}
\usepackage{geometry}
\usepackage{ulem}
\usepackage{color, xcolor}
\usepackage{hyperref}
\usepackage{pdfpages}
\usepackage{bm}
\usepackage{xargs}
\usepackage{xcolor}
\usepackage[colorinlistoftodos,prependcaption,textsize=tiny]{todonotes}
\usepackage{lipsum}
\usepackage{comment}

\allowdisplaybreaks

\newcommand{\eps}{\varepsilon}

\newcommand{\Blam}{{\bm{\lambda}}}
\newcommand{\BGam}{{\bm{\Gamma}}}
\newcommand{\BPsi}{{\bm{\Psi}}}

\newcommand{\Beta}{{\bm{\eta}}}
\newcommand{\BA}{{\bm A}}
\newcommand{\BU}{{\bm U}}
\newcommand{\BV}{{\bm V}}
\newcommand{\BB}{{\bm B}}
\newcommand{\BD}{{\bm D}}
\newcommand{\BJ}{{\bm J}}
\newcommand{\BI}{{\bm I}}
\newcommand{\Bu}{{\bm u}}
\newcommand{\Bp}{{\bm p}}
\newcommand{\But}{{\bm{ \tilde{u} } }}
\newcommand{\Bx}{{\bm x}}
\newcommand{\Bz}{{\bm z}}
\newcommand{\Bv}{{\bm v}}
\newcommand{\Be}{{\bm e}}
\newcommand{\Bw}{{\bm w}}
\newcommand{\BK}{{\bm K}}
\newcommand{\Bf}{{\bm f}}

\newcommand{\lambdaMin}{\underline{\lambda}}
\newcommand{\lambdaMax}{\overline{\lambda}}
\newcommand{\lamIn}{\lambda^{\text{in}}}    
\newcommand{\lamOut}{\lambda^{\text{out}}}
\newcommand{\corr}{K}
\newcommand{\Dsf}{\mathsf{D}}
\newcommand{\Bell}{{\bm B_\ell}}           
\newcommand{\Bellti}{{\bm{\tilde{B}}_\ell}}

\newcommand{\Bellh}{{\bm B_{\hat \ell}}}

\theoremstyle{definition}
\newtheorem{theorem}{Theorem}[section]
\newtheorem{corollary}[theorem]{Corollary}
\newtheorem{proposition}[theorem]{Proposition}
\newtheorem{lemma}[theorem]{Lemma}
\newtheorem{definition}[theorem]{Definition}
\newtheorem{remark}[theorem]{Remark}

\title{Efficient and accurate separable models for discrete material optimization: A continuous perspective
}
\author{P. Gangl$^{1,2}$, N. Nees$^2$, M. Stingl$^2$ \vspace{5mm}\\
$^1$ Johann Radon Institute for Computational and Applied Mathematics (RICAM), \\Altenberger Straße 69, 4040 Linz\\
$^2$ Friedrich-Alexander-Universität Erlangen-Nürnberg, \\Cauerstraße 11, 91058 Erlangen}

\begin{document}
\maketitle

\begin{abstract}
   Multi-material design optimization problems can, after discretization, be solved by the iterative solution of simpler sub-problems which approximate the original problem at an expansion point to first order. In particular, models constructed from convex separable first order approximations have a long and successful tradition in the design optimization community and have led to powerful optimization tools like the prominently used method of moving asymptotes (MMA). In this paper, we introduce several new separable approximations to a model problem and examine them in terms of accuracy and fast evaluation. The models can, in general, be nonconvex and are based on the Sherman-Morrison-Woodbury matrix identity on the one hand, and on the mathematical concept of topological derivatives on the other hand. We show a surprising relation between two models originating from these two -- at a first sight -- very different concepts.
   Numerical experiments show a high level of accuracy for two of our proposed models while also their evaluation can be performed efficiently once enough data has been precomputed in an offline phase. Additionally it is demonstrated that suboptimal decisions can be avoided using our most accurate models. 
\end{abstract}

\textbf{Highlights}
\begin{itemize}
    \item An efficient and easy-to-use separable model for discrete material optimization is discussed 
    \item A highly accurate model based on the topological derivative for triangular inclusion shapes is introduced
    \item A highly accurate model based on the Sherman-Morrison-Woodbury matrix identity is introduced
    \item A surprising connection between these two accurate models is established
    \item Efficient numerical evaluation of all models as well as their accuracy are discussed
\end{itemize}

\tableofcontents

\section{Introduction}
The goal of computational design optimization is to find an optimal arrangement of possibly multiple materials inside a design region of a computational domain. Here, optimality is considered with respect to a given cost function, which most often depends on the solution of a constraining partial differential equation (PDE). Thus, a general PDE-constrained design optimization problem reads
\begin{align} \label{eq_intro_optiDesign}
    \Omega^* = \underset{\Omega}{\mbox{arg min }} J(\Omega, u) \quad \mbox{ subject to } \quad e(\Omega; u) = 0
\end{align}
where $e(\cdot; \cdot)$ represents the PDE constraint and $\Omega$ can also be understood as a set of subdomains corresponding to different materials. There exist different classes of approaches to solving problems of this type. While shape optimization methods \cite{DZ2} can only modify existing boundaries or interfaces between subdomains in a smooth way, topology optimization approaches \cite{SigmundMauteReview2013} can also alter the topology of a design and can thus admit more general solutions. In topology optimization, the design is most often represented by either a level set function \cite{AllaireJouveToader2004, AmstutzAndrae2006} or by means of a density function \cite{BendsoeSigmund2003} that interpolates between different material properties. Note that both kinds of approaches can be extended to the case of multiple materials, see, e.g., \cite{AllaireDapognyMichailidis2014, Gangl2020} or \cite{Cherriere2022}. Typically, the constraining PDE is solved by a numerical method, most often by the finite element method. When approaching a design optimization problem of the type \eqref{eq_intro_optiDesign} by a gradient-based approach, one first has to decide whether the differentiation with respect to the design variable should be carried out before or after discretization of the problem. It should be noted that, depending on the chosen discretization and parametrization of the design, these two ways need not yield the same discrete sensitivities \cite{GanglGfrerer2022}.

In this paper we are interested in a (multi-material) topology optimization problem which we parametrize by a density function. Here, we focus on the approach where we first discretize the problem and then aim to solve the discretized, finite-dimensional problem. Given a computational domain $\Dsf \subset \mathbb R^2$ which is discretized using a fixed structured mesh consisting of $m$ triangular elements with $n$ vertices,  we aim at finding the optimal material distribution $\Blam^* \in \mathbb R^m$ minimizing the heat compliance,
\begin{align} \label{eq_optiProb_intro}
    \Blam^* = \underset{\Blam \in \mathbb [\lambdaMin, \lambdaMax]^m}{\mbox{arg min}} \; \mathcal J(\Blam)
\end{align}
with $ \mathcal J(\Blam) := \Bf^\top \BK(\Blam)^{-1} \Bf$.
Here, $\BK(\Blam) \in \mathbb R^{n \times n}$ and $\Bf \in \mathbb R^n$ represent, respectively, the (invertible) stiffness matrix and the load vector corresponding to a discretization by means of piecewise linear, globally continuous finite elements where the material coefficient in the $\ell$-th element $T_\ell$ is given by $\Blam_\ell$, $\ell \in \{1, \dots, m\}$. The problem may be subject to additional constraints, e.g., on the volume of given materials, or enriched by terms that penalize the appearance of intermediate materials or that regularize the obtained designs by filtering \cite{Bourdin2001}. 

The idea of \textit{sequential global programming} (SGP) \cite{Semmler2018} is the following: Instead of solving an optimization problem like \eqref{eq_optiProb_intro} over $\mathbb R^m$ directly, one solves a sequence of simpler optimization problems with cost function $\hat{\mathcal J}[\Blam^{(k)}](\Blam)$ which approximate the original problem with cost function $\mathcal J(\Blam)$ at an expansion point $\Blam^{(k)}$ to first order. The solution of the simpler optimization problem at iteration $k$ is subsequently used as an expansion point $\Blam^{(k+1)}$ in iteration $k+1$, i.e.,
\begin{align*}
    \Blam^{(k+1)} = \underset{\Blam  \in \mathbb [\lambdaMin, \lambdaMax]^m}{\mbox{arg min}} \, \hat{\mathcal J}[\Blam^{(k)}](\Blam). 
\end{align*}
A class of approximate models that is of particular interest is the class of \textit{separable} models.
The use of convex, separable approximations has a long tradition in design optimization, see, e.g., \cite{Fleury1979,Bruyneel2002,Svanberg2002} and have lead to powerful software realizations like CONLIN \cite{Fleury1989} or the method of moving asymptotes \cite{Svanberg1987}.
Such models allow to solve the approximate optimization problem that is posed over $\mathbb R^m$ by solving merely $m$ one-dimensional optimization problems. These one-dimensional sub-problems can often be solved to global optimality. This observation holds true trivially for the more traditional convex separable approximations, used, e.g., in MMA. However, for separable approximations, convexity is not  a strict requirement. It is clear that the convergence speed of an SGP algorithm strongly depends on the quality of the approximating model $\hat{\mathcal J}[\Blam^{(k)}]$. In this paper, we propose different first order separable models $\hat{\mathcal J}[\Blam^{(k)}]$ and numerically examine them in terms of accuracy and efficiency of evaluation. But it is not only the efficiency, which is of interest. This becomes particularly evident, when topology optimization or discrete material optimization problems are studied. The usual way to deal with such problems is to use a combined relaxation and penalization scheme, see, e.g., \cite{BendsoeSigmund2003} for an introduction to that topic. While such approaches are very successful in practice, in extreme cases it can happen that any feasible integer solution satisfies first order optimality conditions for the continuous relaxations. Thus, there is a certain risk that rather poor local minimizers are obtained. While in literature so-called continuation strategies provide a viable concept to cope with that situation, in this paper we demonstrate that it is in particular the approximation quality in the sub-problem, which can help to avoid 'wrong' decisions taken in the course of the iterations.

We will investigate models $\hat{\mathcal J}$ that exploit the Sherman-Morrison-Woodbury matrix identity on the discrete level and are thus purely algebraic. On the other hand, we will consider the mathematical concept of topological derivatives \cite{NovotnySokolowski2013} which is a notion defined on a purely continuous setting. We will draw some, at a first glance, surprising connections between these two types of approaches and present some models that are at the same time accurate approximations of the original problem and cheap to evaluate.

The rest of this paper is organized as follows: In Section \ref{sec_2_prelim}, we introduce the continuous model problem and its finite element discretization and recall the notions of topological derivatives, separable approximations of optimization problems and also the Sherman-Morrison-Woodbury formula. Next, we introduce a first efficient model based on this formula that is based on a diagonal approximation of the stiffness matrix in Section \ref{sec_SMWdiag}. Subsequently, we introduce a different model that is motivated by the concept of topological derivatives in Section \ref{sec_modelTD}. In Section \ref{sec_SMW_W}, inspired by the procedure of the previous section, we introduce another accurate and efficient to evaluate approximation to the exact Sherman-Morrison-Woodbury model. A relation between these latter two models is established in Section \ref{sec_relations}. Finally, we examine all introduced models numerically in Section \ref{sec_numExp}.

\section{Preliminaries} \label{sec_2_prelim}
In this section, we will introduce the model problem and collect some mathematical preliminaries, which we will make use of in later sections. In particular, we introduce the mathematical concept of topological derivatives, the concept of separable first order approximations of a continuously differentiable function $f : \mathbb R^N \rightarrow \mathbb R$ and recall the Sherman-Morrison-Woodbury matrix identity.

\textbf{Notation.} Vector quantities and matrices will be denoted by bold-face symbols and we will denote the $j$-th component of a vector $\Bv \in \mathbb R^N$ by a sub-index $\Bv_j$. We will denote the $i$-th Cartesian unit vector in $\mathbb R^N$ by $\Be^{(i)}$. The identity matrix of dimension $N$ will be denoted by the symbol $\BI_N$ and for a square matrix $\BA \in \mathbb R^{N\times N}$ we denote by $\mbox{diag}\BA\in \mathbb R^{N\times N}$ the diagonal matrix corresponding to $\BA$, i.e., $(\mbox{diag}\BA)_{i,i} = \BA_{i,i}$ and $(\mbox{diag}\BA)_{i,j} = 0$ for $i\neq j$, $i, j \in \{1,\dots N\}$. We denote by $B_r(x)$ the ball of radius $r$ centered at the point $x$. Moreover, given a set $A \subset \mathbb R^d$, we denote the characteristic function of the set $A$ by $\chi_A(x)$, i.e., $\chi_A(x) = 1$ if $x\in A$ and $\chi_A(x) =0$ otherwise. The space of square integrable functions over a domain $\Dsf$ is denoted by $L^2(\Dsf)$ and the subspace of $L^2(\Dsf)$ functions whose weak gradient is also a $L^2(\Dsf)$ function is denoted by $H^1(\Dsf)$. Finally, given a function $f$ defined on a domain $\Dsf$ and a subdomain $\Omega \subset \Dsf$, we will denote the restriction of $f$ to $\Omega$ by $f|_\Omega$.

\subsection{Model problem} \label{sec_modelProb}
As a model problem, we consider a stationary heat equation on a bounded Lipschitz domain $\Dsf \subset \mathbb R^d$. We are interested in finding the optimal material distribution within $\Dsf$ such that the heat compliance is minimized. We first state the model problem in its continuous version before introducing the discretized problem, which we are actually interested in. In this paper, we restrict ourselves to space dimensions $d=1$ and $d=2$, but remark that most concepts treated here can be extended (with some effort) also to three space dimensions.

\subsubsection{Continuous model problem for two materials} \label{sec_modelProb_cont}
For the spatial dimension $d \in \{1,2\}$, we consider the bounded Lipschitz domain $\Dsf =(0,1)^d \subset \mathbb R^d$ and a given heat source $f \in L^2(\Dsf)$. Given a polygonal set $\Omega \subset \Dsf$, let the piecewise constant heat conductivity $\lambda$ be defined by
\begin{align}
\lambda_\Omega(x) = 
    \begin{cases}
        \lamIn, & x \in \Omega, \\
        \lamOut, & x \in \Dsf \setminus \overline \Omega,
    \end{cases}
\end{align}
for two positive constants $\lamIn, \lamOut > 0$. We assume the boundary of the computational domain to be composed of a Dirichlet and a Neumann boundary, $\partial \Dsf = \Gamma_D \cup \Gamma_N$ with $\Gamma_D \cap \Gamma_N = \emptyset$, where inhomogeneous Dirichlet data $g_D$ and Neumann data $g_N$ are prescribed, respectively. We are interested in minimizing the heat compliance,
\begin{align} \label{eq_J}
    J(u) := \int_\Dsf f(x) \, u(x) \; \mbox dx + \int_{\Gamma_N} g_N(x) \, u(x) \; \mbox ds_x
\end{align}
subject to a stationary heat equation. The weak formulation of the problem at hand reads
\begin{subequations} \label{eq_prob}
    \begin{align}
    &\underset{\Omega}{\mbox{inf }} J(u) \label{eq_probJ}\\
        \mbox{s.t. } u \in V_g: \; \int_\Dsf \lambda_\Omega(x) \nabla u(x) \cdot \nabla v(x) \; \mbox dx &= \int_\Dsf f(x) v(x) \; \mbox dx + \int_{\Gamma_N} g_N(x) v(x) \; \mbox ds_x \quad \forall v \in V_0 \label{eq_probpde}
    \end{align}
\end{subequations}
with the ansatz and test spaces
\begin{align*}   
    V_g := \{v \in H^1(\Dsf): v|_{\Gamma_D} = g_D\}, \qquad V_0 :=  H^1_{\Gamma_D}(\Dsf) := \{v \in H^1(\Dsf): v|_{\Gamma_D} = 0\}.
\end{align*}
For a given subdomain $\Omega$ and assuming that $|\Gamma_D|>0$, due to the Lemma of Lax-Milgram, problem \eqref{eq_probpde} admits a unique solution which we denote by $u_\Omega$. Thus, we introduce the reduced cost function $\mathcal J(\Omega) := J(u_\Omega)$. We assume that the solution $u_\Omega$ is sufficiently regular such that a point evaluation of its gradient $\nabla u_\Omega(z)$ is well-defined for all points $z \in D \setminus \partial \Omega$. When the set $\Omega$ is clear from the context, we will drop the index $\Omega$ and just write $u$ instead of $u_\Omega$.
For simplicity and without loss of generality, we assume $g_D=0$ and thus have $V_g= V_0$. The general case can be obtained by homogenization of the Dirichlet data.
Note that, in the case $d=1$, the boundary of $\Dsf$ consists of two points, $\partial \Dsf = \{0,1\}$. Thus, integrals over $\Gamma_N \subset \partial \Dsf$ have to be understood as point evaluations.

The adjoint state corresponding to optimization problem \eqref{eq_prob} is the unique solution $p \in V_0$ of
\begin{align} \label{eq_defAdj}
    \int_\Dsf \lambda_\Omega(x) \nabla v(x) \cdot \nabla p(x) \; \mbox dx = - \int_\Dsf f(x) v(x) \;  \mbox dx - \int_{\Gamma_N}g_N(x) v(x) \; \mbox ds_x \qquad \forall v \in V_0.
\end{align}
Thus, it can be seen that $p  = -u$.

\subsubsection{Discrete model problem} \label{sec_discmodelprob}
Next, we introduce the discretization of \eqref{eq_prob} by means of piecewise linear, globally continuous finite elements. For that purpose, let $\mathcal T = \{ T_1, \dots, T_m \}$ denote a set of open simplicial elements (i.e., intervals in 1D or triangles in 2D) which form a subdivision of the computational domain $\Dsf$, i.e., 
\begin{align*}
    \overline \Dsf = \bigcup_{\ell=1}^m \overline T_\ell, \qquad T_i \cap T_j = \emptyset \mbox{ for } i \neq j.
\end{align*}
Moreover, we assume that the subdomain $\Omega$ is resolved by the mesh $\mathcal T$, i.e., $\partial \Omega \cap T_\ell = \emptyset$ for all $\ell \in \{1, \dots, m\}$. Let $n  \in \mathbb N$ denote the number of vertices in the mesh, $\{ \varphi_1, \dots, \varphi_n\}$ the nodal basis and $V_h := \text{span}\{\varphi_1, \dots, \varphi_n\} \cap H^1_{\Gamma_D}(\Dsf)$.

Let a vector of conductivity values $\Blam \in \mathbb R^m$ be given. Note that we will sometimes identify a vector $\Blam \in \mathbb R^m$ of material values with a piecewise constant material function $\lambda(x) := \sum_{\ell=1}^m \chi_{T_\ell}(x) \Blam_\ell$.
For given $\Blam \in \mathbb R^m$, the discrete version of the boundary value problem \eqref{eq_probpde} reads
\begin{align}
    \BK(\Blam) \Bu = \Bf
\end{align}
where the stiffness matrix $\BK(\Blam) \in \mathbb R^{n\times n}$ and the load vector $\Bf \in \mathbb R^n$ can be written as
\begin{align}
    \BK(\Blam) = \sum_{\ell=1}^m   \Blam_\ell \Bellti K^{(\ell)}_{loc} \Bellti^\top, \qquad \Bf = \sum_{\ell=1}^m \Bellti f^{(\ell)}_{loc}
\end{align}
with the local stiffness matrix $K^{(\ell)}_{loc} \in \mathbb R^{(d+1)\times (d+1)}$ and local load vector $f^{(\ell)}_{loc} \in \mathbb R^{d+1}$,
\begin{align*}
    \left(K^{(\ell)}_{loc} \right)_{i,j}  =& \int_{T_\ell} \nabla  \varphi_{\ell,j} \cdot \nabla  \varphi_{\ell,i} \; \mbox dx, \quad\quad\quad \; \quad\quad\quad i, j \in \{1, \dots, d+1\}, \\
    \left(f^{(\ell)}_{loc} \right)_{i}  =& \int_{T_\ell} f   \varphi_{\ell,i} \; \mbox dx + \int_{\Gamma_N \cap \bar T_\ell} g_N  \varphi_{\ell,i} \; \mbox ds_x, \quad i \in \{1, \dots, d+1\},
\end{align*}
and the local-to-global operator $\Bellti \in \mathbb R^{n\times (d+1)}$ satisfying $(\Bellti)_{i,j} = 1$ if $i$ is the global index of the $j$-th vertex of element $T_\ell$, and $(\Bellti)_{i,j} = 0$ else. Here, $\varphi_{\ell,i}$, $i=1,\dots, d+1$, denotes the $i$-th basis functions that has non-zero support on $T_\ell$.
Since we are dealing with piecewise linear and globally continuous finite elements, the local stiffness matrix can be written as
\begin{align}
    K^{(\ell)}_{loc} = \BD_\ell  \BD_\ell^\top
\end{align}
with some constant matrices $\BD_\ell \in \mathbb R^{(d+1)\times d}$ depending solely on the coordinates of the vertices of element $T_\ell$. Thus, defining $\Bell := \Bellti \BD_\ell \in \mathbb R^{n \times d} $, the stiffness matrix can also be written as
\begin{align}
    \BK(\Blam) = \sum_{\ell=1}^m \Blam_\ell \Bell \Bell^\top.
\end{align}

\begin{remark}
    In dimension $d=1$, the matrix $\BD_\ell \in \mathbb R^{2\times 1}$ corresponding to an element $T_\ell = (x_{\ell-1}, x_{\ell})$ is given by
    \begin{align*}
        \BD_\ell = \frac{1}{\sqrt{x_{\ell} - x_{\ell-1}}} \begin{pmatrix} -1 \\ 1 \end{pmatrix}.
    \end{align*}
    For $d=2$ and a triangular element $T_\ell$ with vertices $\Bx_{\ell,1}, \Bx_{\ell,2}, \Bx_{\ell,3}$ in counter-clockwise enumeration, the matrix $\BD_\ell \in \mathbb R^{3 \times 2}$ reads
    \begin{align*}
        \BD_\ell = \sqrt{ \frac12 \mbox{det}\BJ_\ell } \begin{pmatrix} -1 & -1 \\ 1 & 0 \\ 0 & 1 \end{pmatrix} \BJ_\ell^{-1}
    \end{align*}
    with $\BJ_\ell = \begin{pmatrix} \Bx_{\ell,2} - \Bx_{\ell,1} & \Bx_{\ell,3} - \Bx_{\ell,1} \end{pmatrix} \in \mathbb R^{2 \times 2}$.
\end{remark}

Finally, the Dirichlet boundary conditions on nodes $\Bv^{(i)}$ on $\overline{\Gamma}_D$ are incorporated by setting $(\BK(\Blam))_{i,i} = 1$ and $(\BK(\Blam))_{i,j} = (\BK(\Blam))_{j,i} = 0$ for $i \neq j$ and $\Bf_i = g_D(\Bv_i)$.
Note that, for $\Blam \in [\lambdaMin, \lambdaMax]^m$ with $\lambdaMin>0$, the stiffness matrix after incorporation of the Dirichlet boundary conditions is invertible. Thus, we can define the solution vector
\begin{align}
    \Bu(\Blam) := \BK(\Blam)^{-1} \Bf
\end{align}
and the corresponding discrete solution $u_h(x) := \sum_{i=1}^n \Bu_i \varphi_i(x)$,
and we introduce the discrete compliance function $\mathcal J: \mathbb [\lambdaMin, \lambdaMax]^m \rightarrow \mathbb R$,
\begin{align} \label{eq_defDiscrCompl}
    \mathcal J(\Blam) := \Bf^\top  \BK(\Blam)^{-1} \Bf.
 \end{align}

\begin{remark} \label{rem_gray_reg}
    In order to obtain practically interesting multi-material designs, the discretized problem should additionally include a mechanism to penalize intermediate material. This can be done by adding a term of the form $J_{\text{gray}}(\Blam) = \sum_\ell (\Blam_\ell - \lamIn)(\lamOut - \Blam_\ell)$ (or an extension of this to multiple materials) to the cost function.
    On the other hand, it is well-known that topology optimization problems of the type \eqref{eq_optiProb_intro} often lack a solution which can be observed numerically in the form of mesh-dependent optimized designs. In order to obtain a well-defined problem, one typically introduces a length scale in the form of a filter radius. This can be realized by adding a term of the form $J_{\text{reg}}(\Blam) = \| \mathbf F_R \Blam - \Blam \|_2^2$ to the cost function. Here, $\mathbf F_R$ is a filtering operator with a given length scale $R$. Thus, one typically is interested in minimizing an enriched cost function $\tilde{\mathcal J}(\Blam) = \mathcal J(\Blam) + \gamma_1 J_{\text{gray}}(\Blam) + \gamma_2 J_{\text{reg}}(\Blam)$. Since the functionals $J_{\text{gray}}(\Blam)$ and $J_{\text{reg}}(\Blam)$ are often separable and can be evaluated efficiently by default, we will ignore these terms for the rest of this paper. For a more detailed discussion on this aspect for multi-material topology optimization, see \cite[Sec. 2.2]{NeesEtAl2022}.
\end{remark}
 
 Later on, we will make use of the following relation.
 \begin{lemma} \label{lem_Bltu}
    Let $u_h$ a piecewise linear and globally continuous finite element function on a given simplicial mesh in $\mathbb R^d$, $d \in \{1,2\}$ with vector of basis coefficients $\Bu \in \mathbb R^n$ and let $\Bell = \Bellti \BD_\ell$ defined as above. Then it holds for any $\ell \in \{1,\dots, m\}$
    \begin{align} \label{eq_Bltu}
        \Bell^\top \Bu = \sqrt{|T_\ell|} \, \nabla u_h|_{T_\ell}.
    \end{align}
\end{lemma}
\begin{proof}
    First note that $\Bu^\top \Bell = \Bu^\top \Bellti \BD_\ell$ and that $\Bu^\top \Bellti$ is the vector of local degrees of freedom on element $T_\ell$. Thus, for $d=1$, we get
        $\Bu^\top \Bell = [\Bu_{\ell,1}, \Bu_{\ell,2}] \BD_\ell = \sqrt{|T_\ell|}(\Bu_{\ell,2} - \Bu_{\ell,1})/|T_\ell|$. The assertion follows by recalling that $|T_\ell|$ is the length of the interval $T_\ell$ and that $u_h$ is linear on $T_\ell$.

        In order to see the relation in the case $d=2$, let $\Phi_\ell$ denote the affine transformation with Jacobian matrix $\BJ_\ell$ that maps the reference triangle with vertices $(0,0)^\top$, $(1,0)^\top$, $(0,1)^\top$ to the given physical triangle $T_\ell$. Recall that, by the chain rule $(\nabla u_h)\circ \Phi_\ell = \BJ_\ell^{-\top} \nabla\hat u $ with $\hat u =u_h\circ \Phi_\ell$ and thus $\nabla\hat u^\top \BJ_\ell^{-1} = ((\nabla u_h)\circ \Phi_\ell)^\top$. Now we have
        \begin{align*}
            \Bu^\top \Bell = [\Bu_{\ell,1}, \Bu_{\ell,2}, \Bu_{\ell,3}] \BD_\ell = \sqrt{ \frac12 \mbox{det}\BJ_\ell } [\Bu_{\ell,2}-\Bu_{\ell,1},\Bu_{\ell,3}-\Bu_{\ell,1} ] \BJ_\ell^{-1},
        \end{align*}
        and noting that $[\Bu_{\ell,2}-\Bu_{\ell,1},\Bu_{\ell,3}-\Bu_{\ell,1} ] = \nabla \hat u^\top$ and $\mbox{det}\BJ_\ell = 2 |T_\ell|$ finishes the proof.
\end{proof}

\paragraph{Chosen meshes.}
Given a refinement level $n_{\text{ref}} \in \{4,5,6\}$, we use a structured mesh with $2^{n_{\text{ref}}}+1$ many uniformly distributed points per dimension. For $d=1$ this corresponds to a uniform grid with $n=2^{n_{\text{ref}}}+1$ points and $m= 2^{n_{\text{ref}}}$ elements. For $d=2$, we have $n = (2^{n_{\text{ref}}}+1)^2$ and $m= 2^{2n_{\text{ref}}+1}$ triangular elements. The triangles are obtained by dividing each square in the Cartesian grid created by the vertices into two triangles with a diagonal connecting the bottom left and top right vertex of a square, see Figure \ref{fig_mesh_elTypes}. As a result, our triangular mesh contains only two types of triangles (both being isosceles right triangles): Element type 1 having the right angle on the bottom right, and element type 2 having the right angle on the top left, see Figure~\ref{fig_mesh_elTypes}.

\begin{figure}
    \begin{center}
        \includegraphics[width=.35\textwidth]{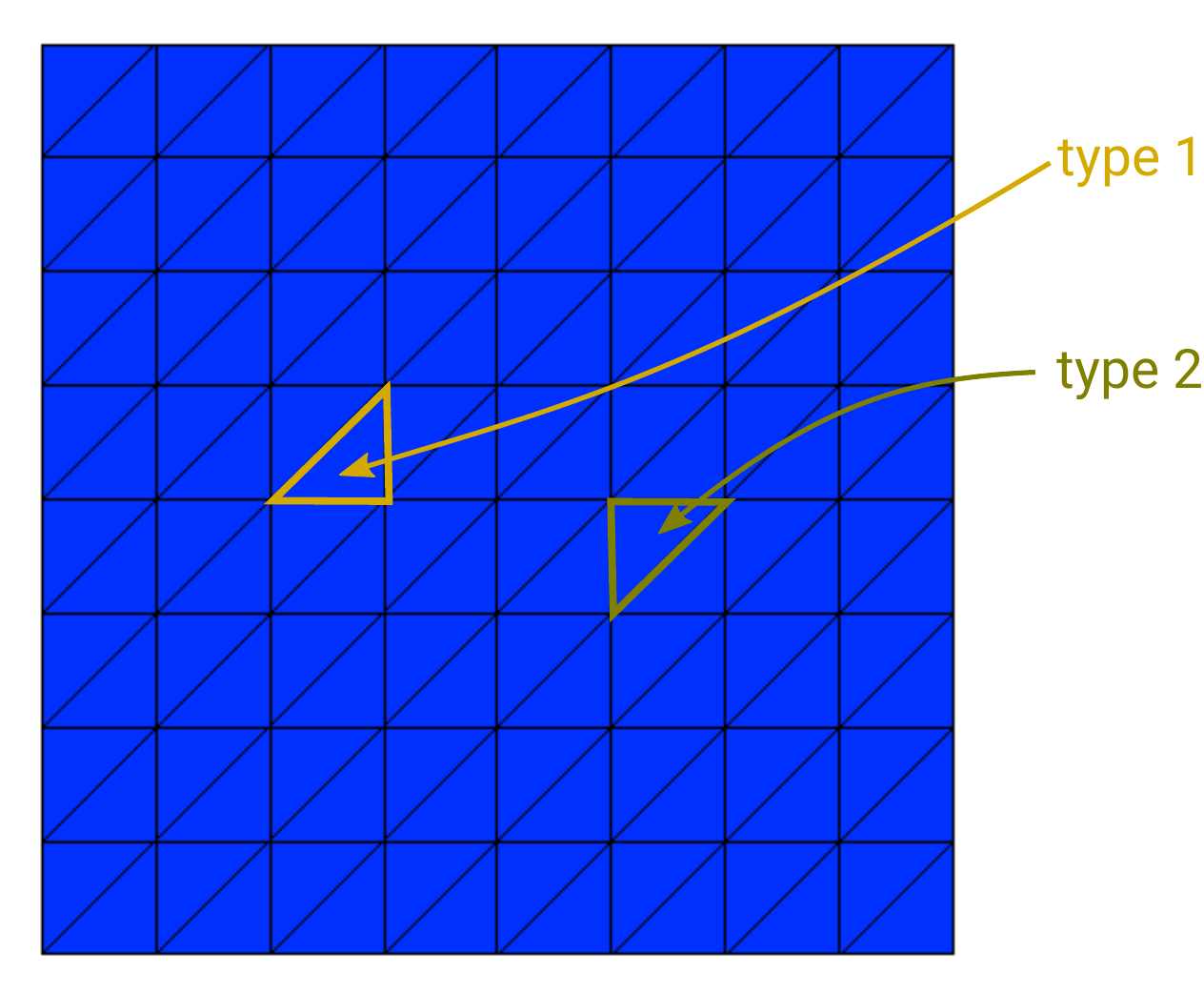}
    \end{center}
    \caption{Example of structured mesh as used in this paper consisting of only two different types of elements.}
    \label{fig_mesh_elTypes}

\end{figure}

\subsection{Topological derivative}
\label{sec_topder}
The topological derivative of a shape function $\mathcal J = \mathcal J(\Omega)$ represents the sensitivity of $\mathcal J$ with respect to a local topological perturbation of the domain $\Omega$ around an inner point $z$. Consider the setting introduced in Section~\ref{sec_modelProb_cont}  where $\Omega$ denotes a subdomain of a domain $\Dsf$. Let $\omega \subset \mathbb R^d$ with $\mathbf 0 \in \omega$ represent the shape of the considered perturbation, e.g., $\omega = B_1(\mathbf 0)$ the unit ball for circular inclusion shapes, and let $z \in \Omega \cup (\Dsf \setminus \overline \Omega)$. For $\eps > 0$, we define the perturbation of shape $\omega$ and size $\eps$ as $\omega_\eps(z) := z + \eps \omega$.
\begin{definition}[topological derivative] \label{def_TD}
    The topological derivative of a shape function $\mathcal J$ at the point $z\in \Omega \cup (\Dsf \setminus \overline \Omega)$ with respect to the inclusion shape $\omega$ is defined by
    \begin{align} \label{eq_defTD}
        d \mathcal J(\Omega)(z, \omega) := \begin{cases}
        \underset{\eps \searrow 0}{\mbox{lim }} \frac{1}{|\omega_\eps|} (\mathcal J(\Omega \setminus \overline \omega_\eps) - \mathcal J(\Omega) ), & z \in \Omega, \\
        \underset{\eps \searrow 0}{\mbox{lim }} \frac{1}{|\omega_\eps|} (\mathcal J(\Omega \cup \omega_\eps) - \mathcal J(\Omega) ), & z \in \Dsf \setminus \overline \Omega.
                                           \end{cases}
    \end{align}
\end{definition}

\begin{remark}
    Note that Definition \ref{def_TD} is equivalent to stating that $d\mathcal J(\Omega)(z,\omega)$ is the first term in a topological asymptotic expansion of the form (here for $z \in \Dsf \setminus \overline \Omega$)
    \begin{align} \label{eq_topAsympExp}
        \mathcal J(\Omega \cup \omega_\eps) = \mathcal J(\Omega) + |\omega_\eps| d \mathcal J(\Omega)(z, \omega) + o(|\omega_\eps|).
    \end{align}

\end{remark}

In general, the topological derivative of PDE-constrained topology optimization problems with elliptic PDE constraints where the principal part of the PDE operator is perturbed involves the solution to an exterior corrector equation, which we define in the following. We refer the reader to \cite{GanglSturm_TDauto} for a comprehensive introduction to the numerical computation of topological derivatives for arbitrary inclusion shapes.
\begin{definition} \label{def_corrK}
    Let $\omega \subset \mathbb R^d$ open with $0 \in \omega$ and let
    \begin{align}
        \lambda_\omega^{i \rightarrow j}(x) = \chi_{\omega}(x) \lambda^j + \chi_{\mathbb R^d \setminus \overline \omega}(x) \lambda^i
    \end{align}
    for $i, j \in \{\text{in},\text{out}\}$, $i\neq j$. Furthermore, for $z \in \Omega \cup (\Dsf \setminus \overline \Omega)$, let $\nabla u(z)$ denote the point evaluation of the gradient of the solution $u$ to \eqref{eq_probpde} at $z$. For any given $\zeta \in \mathbb R^d$, we define the corrector function $\corr_\omega[\lambda^i, \lambda^j;\zeta] \in \dot{BL}(\mathbb R^d)$ for switching from material $\lambda^i$ to $\lambda^j$, $i, j \in \{\text{in},\text{out}\}$, $i\neq j$, as the unique solution to
    \begin{align} \label{eq_defCorr}
        \int_{\mathbb R^d} \lambda_\omega^{i \rightarrow j} \nabla \corr_\omega[\lambda^i, \lambda^j;\zeta](x) \cdot \nabla v(x) \;  \mbox dx = -(\lambda^j - \lambda^i) \int_\omega \zeta \cdot \nabla v(x) \; \mbox dx
    \end{align}
    for all $v \in BL(\mathbb R^d)$.
\end{definition}

Here, $BL(\mathbb R^d) := \{v \in H^1_{\text{loc}}(\mathbb R^d): \nabla v \in L^2(\mathbb R^d)^d \}$ denotes the so-called Beppo-Levi space of locally square integrable functions whose gradient is square integrable over the whole unbounded domain and $\dot{BL}(\mathbb R^d) := BL(\mathbb R^d)/\mathbb R$ is the space of equivalence classes where the constants are factored out, see also \cite{a_DELI_1955a, a_GAST_2020a}.

\begin{remark} \label{rem_Klinear}
Note that, for $i, j \in \{\text{in},\text{out}\}$, $i\neq j$, the mapping $\zeta \mapsto \corr_\omega[\lambda^i, \lambda^j;\zeta]$ is linear and we have for $d=2$
    \begin{align*}
        \corr_\omega[\lambda^i, \lambda^j;\nabla u(z)] = \frac{\partial u}{\partial x_1}(z) \corr_\omega[\lambda^i, \lambda^j;\Be^{(1)}] + \frac{\partial u}{\partial x_2}(z) \corr_\omega[\lambda^i, \lambda^j;\Be^{(2)}].
    \end{align*}

\end{remark}

\begin{proposition} \label{prop_TDomega}
    Let $\omega \in \mathbb R^d$ open with $0 \in \omega$. Let further $p$ be the adjoint state defined in \eqref{eq_defAdj} and $\corr_\omega[\lamIn, \lamOut;\cdot]$, $\corr_\omega[\lamOut, \lamIn;\cdot]$ according to Definition  \ref{def_corrK}. The topological derivative of problem \eqref{eq_prob} with respect to $\omega$ for $z \in \Dsf \setminus \overline \Omega$ reads
\begin{align} \label{eq_TD_Corr}
    d \mathcal J[\lamOut, \lamIn](\Omega)(z, \omega) = (\lamIn- \lamOut) \frac{1}{|\omega|} \int_\omega ( \nabla u(z) + \nabla \corr_\omega[\lamOut, \lamIn;\nabla u(z)](x)) \cdot \nabla p(z) \; \mbox dx.
\end{align}
Likewise, for $z \in \Omega$, the topological derivative is given by
\begin{align} \label{eq_TD_Corr2}
    d \mathcal J[\lamIn, \lamOut](\Omega)(z, \omega) = (\lamOut- \lamIn) \frac{1}{|\omega|} \int_\omega ( \nabla u(z) + \nabla \corr_\omega[\lamIn, \lamOut;\nabla u(z)](x)) \cdot \nabla p(z) \; \mbox dx.
\end{align}
\end{proposition}

\begin{proof}
    For a detailed proof, see, e.g., \cite{a_GAST_2020a}. Moreover, the idea of the proof is outlined in Section \ref{sec_derivTDtrig} where the focus is on triangular inclusion shapes.
\end{proof}

\begin{definition}[weak polarization matrix]
    For $\omega \in \mathbb R^d$ open with $0 \in \omega$, $\zeta \in \mathbb R^d$ and $i,j\in \{\text{in}, \text{out}\}, i\neq j$, let $\corr_\omega[\lambda^i, \lambda^j; \zeta]$ be as defined in Definition \ref{def_corrK}. We introduce the weak polarization matrix
    \begin{align} \label{eq_defPolMat}
        \mathcal P_\omega[\lambda^i, \lambda^j] = \left[ \frac{1}{|\omega|} \int_\omega \nabla \corr_\omega[\lambda^i, \lambda^j; \Be^{(1)}]\mbox dx \quad \frac{1}{|\omega|} \int_\omega  \nabla \corr_\omega[\lambda^i, \lambda^j; \Be^{(2)}]\mbox dx \right] \in \mathbb R^{d\times d}.
    \end{align}
\end{definition}
Using \eqref{eq_defPolMat} and Remark \ref{rem_Klinear}, in the case $z \in \Dsf \setminus \overline \Omega$, we can also write \eqref{eq_TD_Corr} as
\begin{align} \label{eq_TD_IpP}
    d \mathcal J[\lamOut, \lamIn](\Omega)(z, \omega) = (\lamIn - \lamOut) \nabla u(z)^\top (\BI_2 + \mathcal P_\omega[\lamOut, \lamIn]) \nabla p(z),
\end{align}
and an analogous formula is obtained for the case $z \in \Omega$.

It can be seen that the evaluation of the topological derivative at a single point $z$ involves the solution of problem \eqref{eq_defCorr} with $\zeta = \nabla u(z)$. 
In the special cases where $\omega$ is an interval ($d=1$), a disk or ellipse ($d=2$) or a ball or ellipsoid ($d=3$), the solution to problem \eqref{eq_defCorr} can be written explicitly in a closed form. For $d \in \{2,3\}$ and $\omega = B_1(0)$ we have
$\nabla \corr_\omega[\lambda^i, \lambda^j; \nabla u(z)]|_\omega = -(\lambda^j - \lambda^i)/(\lambda^j + (d-1)\lambda^i) \nabla u(z)$ and thus $\mathcal P_\omega[\lambda^i, \lambda^j] = -(\lambda^j - \lambda^i)/(\lambda^j + (d-1)\lambda^i) \BI_d$ and \cite{Amstutz2006}
\begin{align} \label{eq_TDformula_analytic}
    d \mathcal J[\lambda^i, \lambda^j](\Omega)(z, B_1(0)) = d \lambda^i \frac{\lambda^j - \lambda^i}{\lambda^j + (d-1)\lambda^i} \nabla u(z) \cdot \nabla p (z).
\end{align}
For $d=1$ and $\omega = (-1,1)$ we have $\mathcal P_\omega[\lambda^i, \lambda^j] = -(\lambda^j-\lambda^i)/\lambda^j \in \mathbb R$ and thus \cite{Amstutz2021}
\begin{align*}
    d \mathcal J[\lambda^i, \lambda^j](\Omega)(z, B_1(0)) = \frac{\lambda^i}{\lambda^j}(\lambda^j - \lambda^i) u'(z) p'(z).
\end{align*}

\begin{remark}
    For a general inclusion shape $\omega$, \eqref{eq_defCorr} cannot be solved analytically. The same holds true for the case of quasilinear PDE constraints where the problem corresponding to \eqref{eq_defCorr} is quasilinear as well. In these cases, however, it is still possible to get a good approximation to the topological derivative values \eqref{eq_TD_Corr}, \eqref{eq_TD_Corr2} by computing a numerical approximation of the solution of \eqref{eq_defCorr} on a comparably large, but bounded domain $B_R(0) \supset \omega$ (e.g. $R=30$) with homogeneous Dirichlet boundary conditions at $\partial B_R(0)$. This procedure is motivated by the fact that the solution $\corr^{i\rightarrow j}_\omega$ to \eqref{eq_defCorr} often can be shown to decay as $|x| \rightarrow \infty$. We refer the reader to \cite{GanglSturm_TDauto} for a detailed discussion of this aspect. We will also follow this approach in Section \ref{sec_modelTD} for triangular shaped inclusions $\omega$.

    Note that, when no closed form solution is available, in the case of linear PDE constraints with only two different materials $\lamIn$, $\lamOut$ it is sufficient to have access to (an approximation of) $\corr_\omega[\lamIn, \lamOut; \Be^{(k)}]$ and $\corr_\omega[\lamOut, \lamIn; \Be^{(k)}]$ for $k=1, \dots, d$ in order to (approximately) evaluate the topological derivative in the full domain in an efficient way. Thus, problem \eqref{eq_defCorr} has to be solved numerically only $2d$ many times in a pre-computation phase. For the case of quasilinear PDE constraints this precomputation phase is more involved, see \cite{AmstutzGangl2019, GanglSturm2021Hcurl}.
\end{remark}

\subsection{Separable approximations}
Problem \eqref{eq_optiProb_intro} can be solved efficiently by the idea of sequential global programming (SGP) \cite{Semmler2018} where the original optimization problem is replaced by a sequence of simpler sub-problems. For these sub-problems it is beneficial to have approximations of the original objective functions, which are separable, since then solving the sub-problem reduces to the solution of several univariate optimization problems. The following definitions can also be found in \cite{NeesEtAl2022}.
\begin{definition}[separable function]
    Let $N \in \mathbb N$.
    A function $g : \mathbb R^N \rightarrow \mathbb R$ is called \textit{separable} if there exist functions $g_1, \dots, g_N : \mathbb R \rightarrow \mathbb R$ and a constant $g_0 \in \mathbb R$ such that for all $\Bx \in \mathbb R^N$
    \begin{align*}
        g(\Bx) = g_0 + \sum_{i=1}^N g_i(\Bx_i).
    \end{align*}
\end{definition}
We call a model $g$ exact if it still coincides with the original function $f$ when only one component is perturbed.
\begin{definition}[separable exact model]
    Let $N \in \mathbb N$, $\overline \Bx \in \mathbb R^N$ and $f: \mathbb R^N \rightarrow \mathbb R$ be given. A separable function $g: \mathbb R^N \rightarrow \mathbb R$ is called a \textit{separable exact model} of $f$ at $\overline \Bx$ if
    \begin{align} \label{eq_sepExact}
        g(\overline \Bx + \delta x \, \Be^{(i)}) = f(\overline \Bx + \delta x \, \Be^{(i)})
    \end{align}
    for all $i \in \{1, \dots N\}$ and all $\delta x \in \mathbb R$.
\end{definition}

\begin{definition}[separable first order approximation]
    Let $N \in \mathbb N$, $\mathcal I \subset \mathbb R^N$, $\overline \Bx \in \mathcal I$ and $f \in C^1(\mathcal I, \mathbb R)$ be given. A function $g \in C^1(\mathcal I, \mathbb R)$ is called a \textit{separable first order approximation} of $f$ at $\overline \Bx$ if $g$ is separable and
    \begin{align}
        f(\overline \Bx) = g(\overline \Bx) \qquad \mbox{ and } \qquad \nabla f(\overline \Bx) = \nabla g(\overline \Bx).
    \end{align}
\end{definition}

Note that if a function $g$ is a separable exact model of a function $f$ it is also a separable first order approximation. In the following lemma we show how, for any given function $f$, a separable exact model can be constructed.

\begin{lemma} \label{lem_sepExFirst}
    Let $N \in \mathbb N$, $\overline \Bx \in \mathbb R^N$ and $f: \mathbb R^N \rightarrow \mathbb R$ be given. For $j \in \{1, \dots, N\}$ define 
    \begin{align}
        g_j(\overline \Bx; \cdot ) : \mathbb R \rightarrow \mathbb R, \qquad g_j(\overline \Bx;s) = f(\overline \Bx + (s - \overline \Bx_j) \Be^{(j)}) - f(\overline \Bx).
    \end{align}
    Then, $g$ defined by
    \begin{align} \label{eq_lem1}
        g : \mathbb R^N \rightarrow \mathbb R, \qquad g(\Bx) = f(\overline \Bx) + \sum_{j=1}^N g_j(\overline \Bx;\Bx_j)
    \end{align}
    is a separable exact model of $f$ at $\overline \Bx$.
\end{lemma}
\begin{proof}
    It is obvious from \eqref{eq_lem1} that $g$ is separable. In order to see \eqref{eq_sepExact}, let $i \in \{1, \dots, N\}$ be fixed. Noting that
    \begin{align*}
        (\overline \Bx + \delta x \, \Be^{(i)})_j =\begin{cases} \overline \Bx_i + \delta x, & j=i, \\ \overline \Bx_j, & j \in \{1,\dots,N\} \setminus \{i\} , \end{cases}
    \end{align*}
    we have
    \begin{align*}
        g(\overline \Bx + \delta x \,\Be^{(i)}) =& f(\overline \Bx) + \sum_{j=1}^N g_j(\overline \Bx; (\overline \Bx + \delta x \,\Be^{(i)})_j) \\
        =& f(\overline \Bx) + g_i(\overline \Bx; \overline \Bx_i + \delta x) + \sum_{j \neq i} g_j(\overline \Bx; \overline \Bx_j) \\
        =& f(\overline \Bx) + f(\overline \Bx + \delta x \,\Be^{(i)}) - f(\overline \Bx)  = f(\overline \Bx + \delta x \,\Be^{(i)}),
    \end{align*}
    where we used $g_j(\overline \Bx; \overline \Bx_j) = 0$. This finishes the proof.
\end{proof}

While Lemma \ref{lem_sepExFirst} defines a separable exact model $g$ for any function $f$, it should be noted that the evaluation of $g$ involves $N+1$ function evaluations of $f$, which can in practice be prohibitively expensive (in particular when a function evaluation involves the solution of a PDE). Thus, considering the problem at hand \eqref{eq_optiProb_intro}, our goal in this paper is to find close approximations of the separable exact model defined by \eqref{eq_lem1} which are cheap to evaluate.

\subsection{The Sherman-Morrison-Woodbury formula in a finite element context}

 The following linear algebra result will prove useful for defining exact models in the context of discrete PDE-constrained material optimization. It gives a formula for the inverse of a perturbed matrix only in terms of the inverse of the unperturbed matrix.
\begin{lemma}[Sherman-Morrison-Woodbury formula \cite{golub2013matrix}] \label{lem_SMW}
    Let $N, k \in \mathbb N$ and matrices $\BA \in \mathbb R^{N \times N}$ invertible, $\BU \in \mathbb R^{N \times k}$, $\BV \in \mathbb R^{k \times N}$ be given. It holds
    \begin{align} \label{eq_SMW}
        (\BA+\BU \BV)^{-1} = \BA^{-1} - \BA^{-1} \BU (\BI_k + \BV \BA^{-1} \BU )^{-1} \BV \BA^{-1}.
    \end{align}
\end{lemma}
 
Lemma \ref{lem_SMW} gives rise to a separable exact model of the discrete compliance functional $\mathcal J = \mathcal J(\Blam)$ defined in \eqref{eq_defDiscrCompl}.
For that purpose, let the material distribution $\Blam \in [\lambdaMin, \lambdaMax]^m \in \mathbb R^m$ be given and fix an element $T_\ell$, $\ell \in \{1,\dots,m\}$. We consider a perturbation of $\Blam$ in this element and define the perturbed vector $\Beta := \Blam + (\eta - \Blam_\ell) \Be^{(\ell)}$ for some $\eta \in [\lambdaMin, \lambdaMax]$. Note that $\Beta$ coincides with $\Blam$ in all components except for component $\ell$ where it has value $\eta$ rather than $\Blam_\ell$. Recall the definition of the matrices $\Bell$, $\ell \in \{1,\dots, m\}$ from Section \ref{sec_discmodelprob}.

\begin{proposition} \label{prop_complSMW}
    Let $\Blam \in \mathbb R^m$, $\ell \in \{1,\dots, m\}$ fixed and $\Beta := \Blam + (\eta - \Blam_{\ell}) \Be^{(\ell)}$. Then
  \begin{align}
    \mathcal J(\Beta) = \mathcal J(\Blam) - |T_\ell| (\eta - \Blam_{\ell}) (\nabla u_h|_{T_\ell})^\top( \BI_d + (\eta - \Blam_{\ell})\Bell^\top \BK(\Blam)^{-1} \Bell)^{-1} \nabla u_h|_{T_\ell}.
  \end{align}
\end{proposition}

\begin{proof}
The stiffness matrices according to $\Blam$ and $\Beta$ read
\begin{align}   
    \BK(\Blam) = \sum_{k = 1}^m \Blam_k {B}_k {B}_k^\top \quad \mbox{ and } \quad \BK(\Beta) = \sum_{k = 1}^m \Beta_k {B}_k {B}_k^\top = \BK(\Blam) + (\eta-\Blam_{\ell}) \Bell \Bell^\top.
\end{align}
Thus, the inverse of $\BK(\Beta)$ can be obtained by means of Lemma \ref{lem_SMW} by setting $\BA = \BK(\Blam) \in \mathbb R^{n\times n}$, $\BU = (\eta-\Blam_{\ell})\Bell \in \mathbb R^{n\times d}$ and $\BV = \Bell^\top \in \mathbb R^{d\times n}$ as
\begin{align}
    \BK(\Beta)^{-1} = \BK(\Blam)^{-1} - (\eta - \Blam_{\ell})\BK(\Blam)^{-1} \Bell ( \BI_d + (\eta - \Blam_{\ell})\Bell^\top \BK(\Blam)^{-1} \Bell)^{-1} \Bell^\top \BK(\Blam)^{-1} .
\end{align}
Thus, defining $\But = \BK(\Beta)^{-1}\Bf$ and using that $\Bu = \BK(\Blam)^{-1}\Bf$ and $\BK(\Blam) = \BK(\Blam)^\top$, it holds
\begin{align}
   \Bf^\top \But = \Bf^\top \Bu - (\eta - \Blam_{\ell}) \Bu^\top \Bell ( \BI_d + (\eta - \Blam_{\ell})\Bell^\top \BK(\Blam)^{-1} \Bell)^{-1} \Bell^\top \Bu.
\end{align}
Noting that $\mathcal J(\Beta) = \Bf^\top \But$, $\mathcal J(\Blam) = \Bf^\top \Bu$ and, by Lemma \ref{lem_Bltu}, $\Bell^\top \Bu  = \sqrt{|T_\ell|} \nabla u_h|_{T_\ell}$ finishes the proof.
\end{proof}

\begin{remark} \label{rem_cost_nonselfadj}
    We remark that this procedure can also be followed for non-selfadjoint problems. Using the adjoint state $p_h \in V_h$ whose basis vector $\Bp$ is obtained as the solution of ${\BK(\Blam)^{\top}\Bp = - J'(\Bu)}$, in the case of a linear cost function $J$ we obtain
    \begin{align*}
       \mathcal J(\Beta) =  J(\tilde \Bu) =& J(\Bu) + J'(\Bu)(\tilde \Bu - \Bu) \\
        =& J(\Bu) + J'(\Bu)(\BK(\Beta)^{-1} - \BK(\Blam)^{-1}) \Bf \\
        =& J(\Bu) - J' (\Bu)(\eta - \Blam_{\ell})\BK(\Blam)^{-1} \Bell ( \BI_d + (\eta - \Blam_{\ell})\Bell^\top \BK(\Blam)^{-1} \Bell)^{-1} \Bell^\top \BK(\Blam)^{-1} \Bf \\
        =&\mathcal J(\Blam) +  |T_\ell|(\eta - \Blam_{\ell}) (\nabla p_h|_{T_\ell} )^\top( \BI_d + (\eta - \Blam_{\ell})\Bell^\top \BK(\Blam)^{-1} \Bell)^{-1}  \nabla u_h|_{T_\ell}
    \end{align*}
    If $J$ is not linear the above identity only holds up to a remainder of second order.
\end{remark}

From Lemma \ref{lem_sepExFirst}, we get the following result:
\begin{proposition}
Let $\Blam, \Beta \in \mathbb R^m$. The function $\hat{\mathcal J}_{\text{SMW}}$ defined by
    \begin{align} \label{eq_SMWexact}
        \hat{\mathcal J}_{\text{SMW}}(\Beta) := \mathcal J(\Blam) -\sum_{\ell=1}^m |T_\ell|   (\Beta_\ell -  \Blam_{\ell}) (\nabla u_h|_{T_\ell})^\top ( \BI_d + (\Beta_\ell -  \Blam_{\ell})\Bell^\top \BK(\Blam)^{-1} \Bell)^{-1} \nabla u_h|_{T_\ell}
    \end{align}
    is a separable exact model of $\mathcal J$ at $\Blam$.
\end{proposition}
\begin{proof}
    From Lemma \ref{lem_sepExFirst} we know that
    \begin{align}
      \hat{\mathcal J}_{\text{SMW}}(\Beta) = \mathcal J(\Blam)  + \sum_{\ell=1}^m  \left( \mathcal J(\Blam + (\Beta_\ell - \Blam_\ell) \Be^{(\ell)}) - \mathcal J(\Blam) \right)
     \end{align}
    is a separable exact model of $\mathcal J$ at $\Blam$. Plugging in the result of Proposition \ref{prop_complSMW} yields the assertion.
\end{proof}

Compared to the general separable exact model according to Lemma \ref{lem_sepExFirst}, which can be defined for any function, using the Sherman-Morrison-Woodbury formula we have found a closed form for a separable exact model for our given cost function in \eqref{eq_SMWexact}.
Of course, it can be seen that model \eqref{eq_SMWexact} still involves the inverse of the stiffness matrix for the material distribution given by $\Blam$ which one typically does not have access to. Even when the stiffness matrix has been factorized for computing the state $\Bu$, the evaluation of \eqref{eq_SMWexact} involves $m$ many forward/backward substitutions which amounts to a total effort in the order of $\mathcal O(m n^2)$ and is thus prohibitive for many real-world applications.
In the subsequent section, we will introduce an approximation of \eqref{eq_SMWexact} which can be evaluated efficiently.

\section{An efficient separable model based on the Sherman-Morrison-Woodbury formula} \label{sec_SMWdiag}
Recall the separable exact first order model \eqref{eq_SMWexact} which reads
\begin{align} \label{eq_SMWexactGam}
    \hat{\mathcal J}_{\text{SMW}}(\Beta) = \mathcal J(\Blam) -\sum_{\ell=1}^m |T_\ell|  (\Beta_\ell -  \Blam_{\ell}) (\nabla u_h|_{T_\ell})^\top \left( \BI_d - (\Beta_\ell -  \Blam_{\ell}) \BGam^{(\ell)}\right)^{-1} \nabla u_h|_{T_\ell}
\end{align}
with the definition $\BGam^{(\ell)} := - \Bell^\top \BK(\Blam)^{-1} \Bell \in \mathbb R^{d\times d}$ for all $\ell \in \{1, \dots, m\}$. Since the evaluation of the model involves the solution of a linear system with the system matrix $\BK(\Blam)$ for every element index $\ell$, we introduce an approximation which can be evaluated more efficiently. For that purpose, we simply approximate the inverse of the stiffness matrix by the inverse of the diagonal approximation of the stiffness matrix to obtain the model
\begin{align} \label{eq_SMWdiag_model}
    \hat{\mathcal J}_{\text{SMWdiag}}(\Beta) := \mathcal J(\Blam) -\sum_{\ell=1}^m   |T_\ell| (\Beta_\ell -  \Blam_{\ell}) (\nabla u_h|_{T_\ell})^\top \left( \BI_d - (\Beta_\ell -  \Blam_{\ell}) \BGam^{(\ell)}_\text{diag}\right)^{-1} \nabla u_h|_{T_\ell}
\end{align}
with
\begin{align}
    \BGam^{(\ell)}_\text{diag} := - \Bell^\top (\mbox{diag}\BK(\Blam))^{-1} \Bell \in \mathbb R^{d\times d}.
\end{align}
This model is a separable first order model, but it is no longer exact. Note that this idea was already proposed in \cite{NeesEtAl2022} in the context of a discrete dipole approximation method in an application from optics and is transferred to a finite element setting here.

In the following, we investigate model \eqref{eq_SMWdiag_model} in the one-dimensional case, where connections to the mathematical concepts of topological and shape derivatives can be established.
In spatial dimension $d=1$, we can compute the matrix $\BGam^{(\ell)}_\text{diag}$ explicitly.
\begin{lemma} \label{lem_SMWdiag_1d}
    Let $d=1$ and let a uniform mesh $\{T_1, \dots, T_m\}$ of the computational domain be given. Assume that element $T_\ell$ is occupied by material $\lamOut$ and also its two neighbors $T_{\ell-1}$, $T_{\ell+1}$ are occupied by the same material, i.e., $\Blam_{\ell-1} = \Blam_{\ell} = \Blam_{\ell+1} = \lamOut$. Then it holds
    \begin{align}
        \BGam^{(\ell)}_\text{diag} = -\frac{1}{\lamOut}
    \end{align}
    and, for $\Beta = \Blam + (\lamIn - \lamOut)\Be^{(\ell)}$,
    \begin{align} \label{eq_SMWdiag_1d}
        \hat{\mathcal J}_{\text{SMWdiag}}(\Beta) = \mathcal J(\Blam) - |T_\ell|\frac{\lamOut}{\lamIn} (\lamIn - \lamOut) (u_h'|_{T_\ell})^2.
    \end{align}

\end{lemma}
\begin{proof}
    Using the definition of $\Bell$ from Section \ref{sec_discmodelprob}, we have
    \begin{align*}
        \Bell^\top (\mbox{diag} \BK(\Blam))^{-1} \Bell = h^{-1}  \left( \begin{array}{c} -1 \\ 1 \end{array} \right)^\top (\Bellti)^\top  (\mbox{diag} \BK(\Blam))^{-1} \Bellti   \left( \begin{array}{c} -1 \\ 1 \end{array} \right)
    \end{align*}
    where $h$ denotes the uniform mesh size.
    Using that $(\mbox{diag} \BK(\Blam) )^{-1}_{ii} =\frac{1}{a(\varphi_i, \varphi_i)}$ with $a(\varphi_i, \varphi_i ) = \int_\Dsf \lambda(x) |\nabla \varphi_i|^2 \; \mbox dx$ and
\begin{align}
    (\Bellti)^\top  (\mbox{diag} \BK(\Blam))^{-1} \Bellti = \left( \begin{array}{cc}
        \frac{1}{a(\varphi_{\ell,1}, \varphi_{\ell,1})} & 0\\
        0& \frac{1}{a(\varphi_{\ell,2}, \varphi_{\ell,2})}
                                                        \end{array}\right)
\end{align}
where $\varphi_{\ell,1}$ and $\varphi_{\ell,2}$ are the basis functions corresponding to the left and right node of element $T_\ell$, respectively, we get
\begin{align}
    \BGam^{(\ell)}_\text{diag} = - \Bell^\top (\mbox{diag} \BK(\Blam))^{-1} \Bell = -h^{-1} \left( \frac{1}{a(\varphi_{\ell,1}, \varphi_{\ell,1})} + \frac{1}{a(\varphi_{\ell,2}, \varphi_{\ell,2})} \right).
\end{align}
The first result follows by noting that $a(\varphi_{\ell,1}, \varphi_{\ell,1}) = a(\varphi_{\ell,2}, \varphi_{\ell,2}) = \frac{2 \lamOut}{h}$. The second identity follows by plugging in and noting that $\Beta_\ell = \lamIn$, $\Blam_\ell = \lamOut$.
\end{proof}

\begin{remark}[Relation to topological derivative in 1D]
    As pointed out in Section \ref{sec_topder}, the topological derivative of our optimization problem \eqref{eq_prob} at a point $z \in \Dsf \setminus \overline \Omega$ (i.e., where $\lambda_\Omega(z) = \lamOut$) with respect to $\omega = (-1,1)$ the one-dimensional unit ball reads
    \begin{align} \label{eq_dJoutin}
        d \mathcal J[\lamOut, \lamIn](\Omega)(z, \omega) = -\frac{\lamOut}{\lamIn}(\lamIn - \lamOut) (u'(z))^2,
    \end{align}
    where we used that $p = -u$ for our particular problem at hand. This means that, in one space dimension, the model $\hat{\mathcal J}_{\text{SMWdiag}}$ introduced in \eqref{eq_SMWdiag_model} actually coincides with the finite element discretization of the model that is naturally defined by the definition of the topological derivative
    \begin{align} \label{eq_TDmodel1d}
        \mathcal J(\Omega \cup \omega_\eps) \approx \mathcal J(\Omega) + |\omega_\eps| d\mathcal J(\Omega)(z, \omega),
    \end{align}
    see also \eqref{eq_topAsympExp}.

    We mention that this direct correspondence of the discrete model $\hat{\mathcal J}_{\text{SMWdiag}}$ and the closed-form formula of the topological derivative only works since the elements $T_\ell$ in a 1D mesh are scaled versions of the 1D unit ball $\omega = B_1(0) = (-1,1)$. This is no longer the case in two or three dimensions, where elements are polygonal or polyhedral.
\end{remark}

We remark that we also observed numerically that the finite element discretization of \eqref{eq_TDmodel1d} and the model \eqref{eq_SMWdiag_model} coincide in elements $T_\ell$ in homogeneous regions (i.e., where $\Blam_{\ell-1}=\Blam_{\ell}=\Blam_{\ell+1}$). In elements $T_\ell$ that are adjacent to the material interface $\partial \Omega$, however, Lemma \ref{lem_SMWdiag_1d} and thus formula \eqref{eq_SMWdiag_1d} are no longer valid. We observed that the topological derivative model \eqref{eq_TDmodel1d}, however, still yielded very good results. This can be explained by the following discussion on the relation between the topological and shape derivative in 1D.

\begin{remark}[Relation to shape derivative in 1D]
The shape derivative for moving the interface $\Gamma := \overline \Omega \cap \overline{ (\Dsf \setminus \Omega)}$ in the direction given by a vector field $V \in C^1( \mathbb R^d, \mathbb R^d)$ can, by the structure theorem of Hadamard-Zolesio \cite{DZ2}, always (under suitable smoothness assumptions) be written in the form
\begin{align*}
    d \mathcal J(\Omega; V) = \int_\Gamma L (V \cdot n) \; ds_x,
\end{align*}
for some scalar function $L$. Here $\Omega \subset \Dsf$ is the domain where $\lambda = \lamIn$ and $\Dsf \setminus \overline \Omega$ is where $\lambda = \lamOut$, and $n$ denotes the outer unit normal vector to $\Omega$. For our problem \eqref{eq_prob}, $L$ is given by the formula
\begin{align*}
   L = (\lamIn - \lamOut) (\nabla u\cdot \tau)(\nabla p\cdot \tau) -\left( \frac{1}{\lamIn} - \frac{1}{\lamOut} \right) (\lambda_\Omega \nabla u \cdot n) (\lambda_\Omega \nabla p \cdot n)
\end{align*}
with the tangential vector $\tau$ (see, e.g., \cite{AmstutzDapognyFerrer2018}). In the case where $d=1$ the tangential derivative of $u$ and thus the first term of $L$ vanishes. Using that $n$ is a scalar with $n^2=1$ and again that $p = -u$, we get
\begin{align*}
    L =  \left( \frac{1}{\lamIn}   - \frac{1}{\lamOut} \right) (\lamOut u_{\text{out}}'|_\Gamma)^2 = -\frac{\lamOut}{\lamIn}(\lamIn - \lamOut) (u_{\text{out}}'|_\Gamma)^2,
\end{align*}
where $u_{\text{out}}'|_\Gamma$ denotes the limit at $\Gamma$ of the discontinuous quantity $u'$ when coming from $\Dsf \setminus \overline \Omega$.

Note that this formula for the 1D shape derivative resembles the topological derivative formula \eqref{eq_dJoutin}, which explains why \eqref{eq_TDmodel1d} is a very good model in 1D.
\end{remark}


\section{A separable model based on the topological derivative} \label{sec_modelTD}
In this section, we propose a separable model that is based on the notion of the topological derivative. We fix the space dimension $d=2$. The topological derivative of a shape function $\mathcal J = \mathcal J(\Omega)$ with respect to a perturbation of shape $\omega$ around a spatial point $z$ was introduced in Section \ref{sec_topder}. We emphasize that, while closed-form formulas for the topological derivative only exist in the case of circular or elliptic inclusion shapes, a numerical approximation of the weak polarization matrix \eqref{eq_defPolMat} and thus of the topological derivative formulas \eqref{eq_TD_Corr} and \eqref{eq_TD_Corr2} is possible for arbitrary inclusion shapes $\omega$ with $0 \in \omega$, see also \cite{GanglSturm_TDauto}. We will follow this idea for the case of triangular inclusion shapes.

\begin{figure}
    \begin{center}
    \begin{tabular}{ccc}
        \includegraphics[width=.3\textwidth]{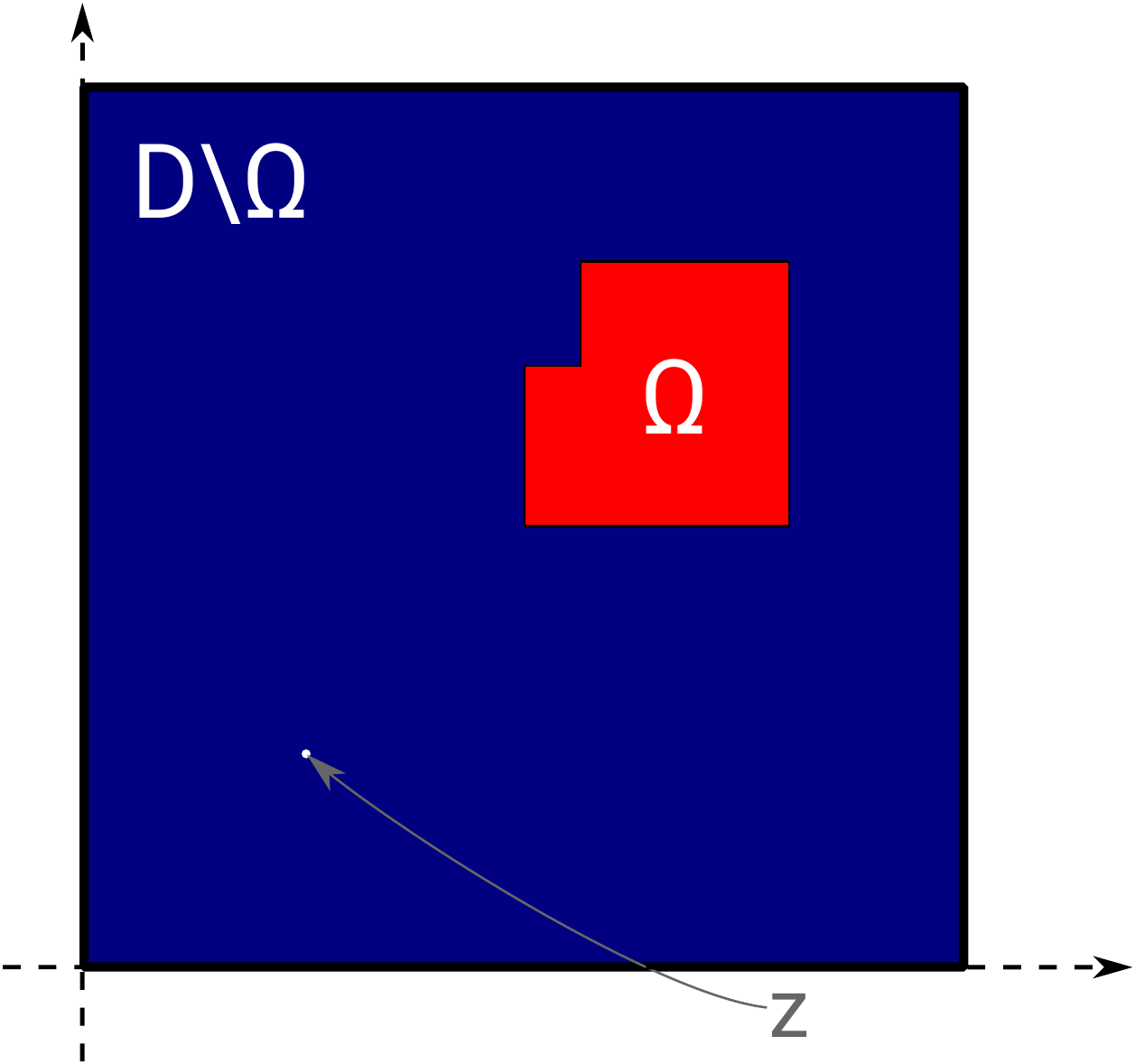}
        &&
        \includegraphics[width=.3\textwidth]{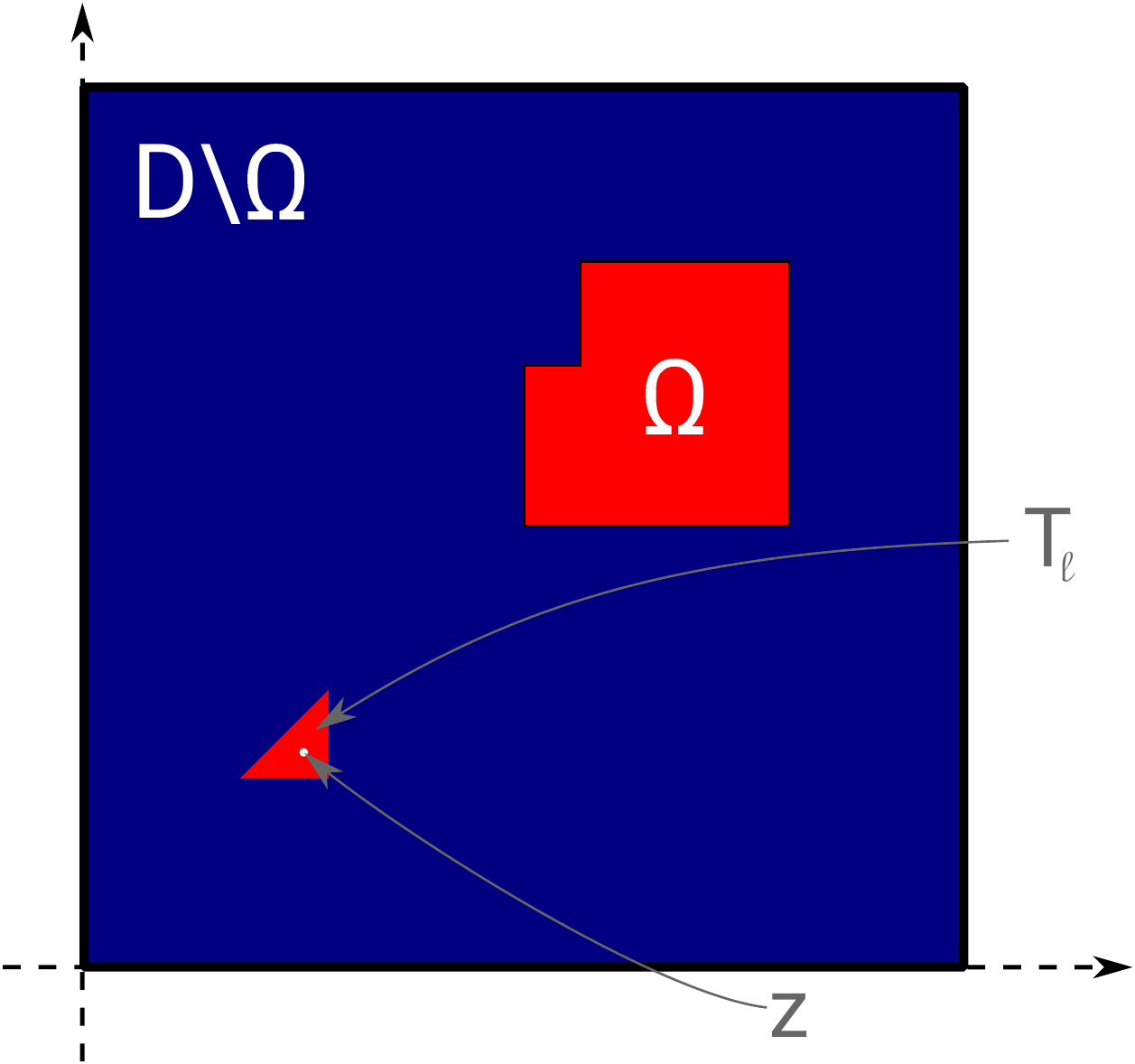} \\
        (a) && (b)
    \end{tabular}
    \end{center}
    \caption{(a) Unperturbed configuration. (b) Perturbed configuration where the domain is perturbed in triangle $T_\ell$ whose centroid is the point $z$.}
    \label{fig_unpert_pert}
\end{figure}

Let $\Omega \subset \Dsf$ be given and consider an element $T_\ell \in \mathcal T$ of type 1 (cf. Fig. \ref{fig_mesh_elTypes}) with $T_\ell \subset \Dsf \setminus \overline \Omega$ with vertices $\Bx_{\ell,1}, \Bx_{\ell,2}, \Bx_{\ell,3}$ (in counter-clockwise enumeration) and centroid $\Bz_\ell := (\Bx_{\ell,1}+\Bx_{\ell,2}+\Bx_{\ell,3})/3$, see Figure \ref{fig_unpert_pert}.
Let $\hat T^{(1)}$ denote the reference triangle defined by its three vertices $\frac{1}{3}(-2,-1)^\top$, $\frac{1}{3}(1,-1)^\top$, $\frac{1}{3}(1,2)^\top$ and $\hat T^{(2)}$ the reference triangle with vertices $\frac{1}{3}(2,1)^\top$, $\frac{1}{3}(-1,-2)^\top$, $\frac{1}{3}(-1,1)^\top$.
For the rest of this section we restrict ourselves to elements of type 1. We set $\hat T := \hat T^{(1)}$ and define $\Phi_{h,\ell} : \hat T \rightarrow T_\ell, \Bx \mapsto \Bz_\ell + h \Bx$ as the affine transformation satisfying $\Phi_{h,\ell}(\hat T) = T_\ell$. We remark that the procedure is completely analogous for elements of type 2 using reference triangle $\hat T^{(2)}$.

\subsection{Derivation of topological derivatives for triangular inclusion shapes} \label{sec_derivTDtrig}
Given a domain $\Omega$, recall the notation $\lambda_\Omega(x) = \chi_\Omega(x) \lamIn + \chi_{\Dsf \setminus \overline \Omega}(x) \lamOut$. For the fixed triangular domain perturbation $T_\ell$, let the perturbed solution $u^{(\ell)} \in V_g$ be defined as the unique solution satisfying
\begin{align} \label{eq_def_upert}
    \int_\Dsf \lambda_{\Omega \cup T_\ell}(x) \nabla u^{(\ell)} \cdot \nabla v \; \mbox dx = \int_\Dsf f v \; \mbox dx  + \int_{\Gamma_N} g_N v \; \mbox ds_x 
\end{align}
for all $v \in V_0$.
We rewrite the difference of the perturbed and unperturbed cost function by adding the equations \eqref{eq_def_upert} and \eqref{eq_probpde} defining $u^{(\ell)}$ and $u$, respectively. Using the adjoint state $p$ defined by \eqref{eq_defAdj} as test function, we obtain
\begin{align*}
    \mathcal J(\Omega \cup T_\ell) -& \mathcal J(\Omega) = J(u^{(\ell)}) - J(u) \\
    =& J(u^{(\ell)}) + \int_\Dsf \lambda_{\Omega \cup T_\ell}(x) \nabla u^{(\ell)} \cdot \nabla p \; \mbox dx - \int_\Dsf f p \; \mbox dx  - \int_{\Gamma_N} g_N p \; \mbox ds_x \\
    &- J(u) -  \int_\Dsf \lambda_{\Omega}(x) \nabla u \cdot \nabla p \; \mbox dx + \int_\Dsf f p \; \mbox dx  + \int_{\Gamma_N} g_N p \; \mbox ds_x  \\
    =& \int_\Dsf f (u^{(\ell)} - u) \; \mbox dx + \int_{\Gamma_N} g_N (u^{(\ell)} - u) \; \mbox ds_x \\
    &+ \int_\Dsf \lambda_{\Omega}(x) \nabla (u^{(\ell)}-u) \cdot \nabla p \; \mbox dx     + (\lamIn - \lamOut)  \int_{T_\ell} \nabla u^{(\ell)} \cdot \nabla p \; \mbox dx  \\
    =& (\lamIn - \lamOut)  \int_{T_\ell} \nabla (u^{(\ell)}-u) \cdot \nabla p \; \mbox dx  + (\lamIn - \lamOut)  \int_{T_\ell} \nabla u \cdot \nabla p \; \mbox dx
\end{align*}
where we used the adjoint equation \eqref{eq_defAdj} in the last step.
Making a change of variables $x \mapsto \Phi_{h,\ell}(x)$ and defining $\corr_{h,\ell, \hat T}[\lamOut, \lamIn] := \frac1h (u^{(\ell)} - u)\circ \Phi_{h,\ell}$, we get
\begin{align}
    \begin{aligned}
    \mathcal J(\Omega \cup T_\ell) = \mathcal J(\Omega) &+ h^2 (\lamIn - \lamOut)  \int_{\hat T}  \nabla  \corr_{h,\ell, \hat T}[\lamOut, \lamIn] \cdot (\nabla p)\circ \Phi_{h,\ell} \; \mbox dx \label{eq_derivTD1}\\
    &+ h^2(\lamIn - \lamOut)  \int_{\hat T} (\nabla u)\circ \Phi_{h,\ell} \cdot (\nabla p)\circ \Phi_{h,\ell} \; \mbox dx,
    \end{aligned}
\end{align}
where we used that, according to the chain rule, $(\nabla v)\circ \Phi_{h,\ell} = \frac1h \nabla (v\circ \Phi_{h,\ell})$ and $\text{det}(\partial \Phi_{h,\ell}) = h^2$.

Subtracting the unperturbed state equation \eqref{eq_probpde} from the perturbed equation \eqref{eq_def_upert}, we see that $u^{(\ell)} - u  \in V_0$ satisfies
\begin{align*}
    \int_\Dsf \lambda_{\Omega \cup T_\ell} \nabla (u^{(\ell)} - u ) \cdot \nabla v \; \mbox dx = -(\lamIn - \lamOut) \int_{T_\ell} \nabla u \cdot \nabla v \; \mbox dx
\end{align*}
for all $v \in V_0$. Making the same change of variables, this amounts to
\begin{align} \label{eq_Khell_exact}
    \int_{\Phi_{h,\ell}^{-1}(\Dsf)}\lambda_{\hat T \cup \Phi_{h,\ell}^{-1}(\Omega)} \nabla \corr_{h,\ell, \hat T}[\lamOut, \lamIn] \cdot \nabla \psi \; \mbox dx  = -(\lamIn - \lamOut) \int_{\hat T} (\nabla u)\circ \Phi_{h,\ell} \cdot \nabla \psi \; \mbox dx
\end{align}
for all $\psi \in H^1_0(\Phi_{h,\ell}^{-1}(\Dsf))$. The domain and material distribution of problem \eqref{eq_Khell_exact} is depicted in Figure \ref{fig_pert_rescaled}.
\begin{figure}
    \begin{center}
        \includegraphics[width=\textwidth]{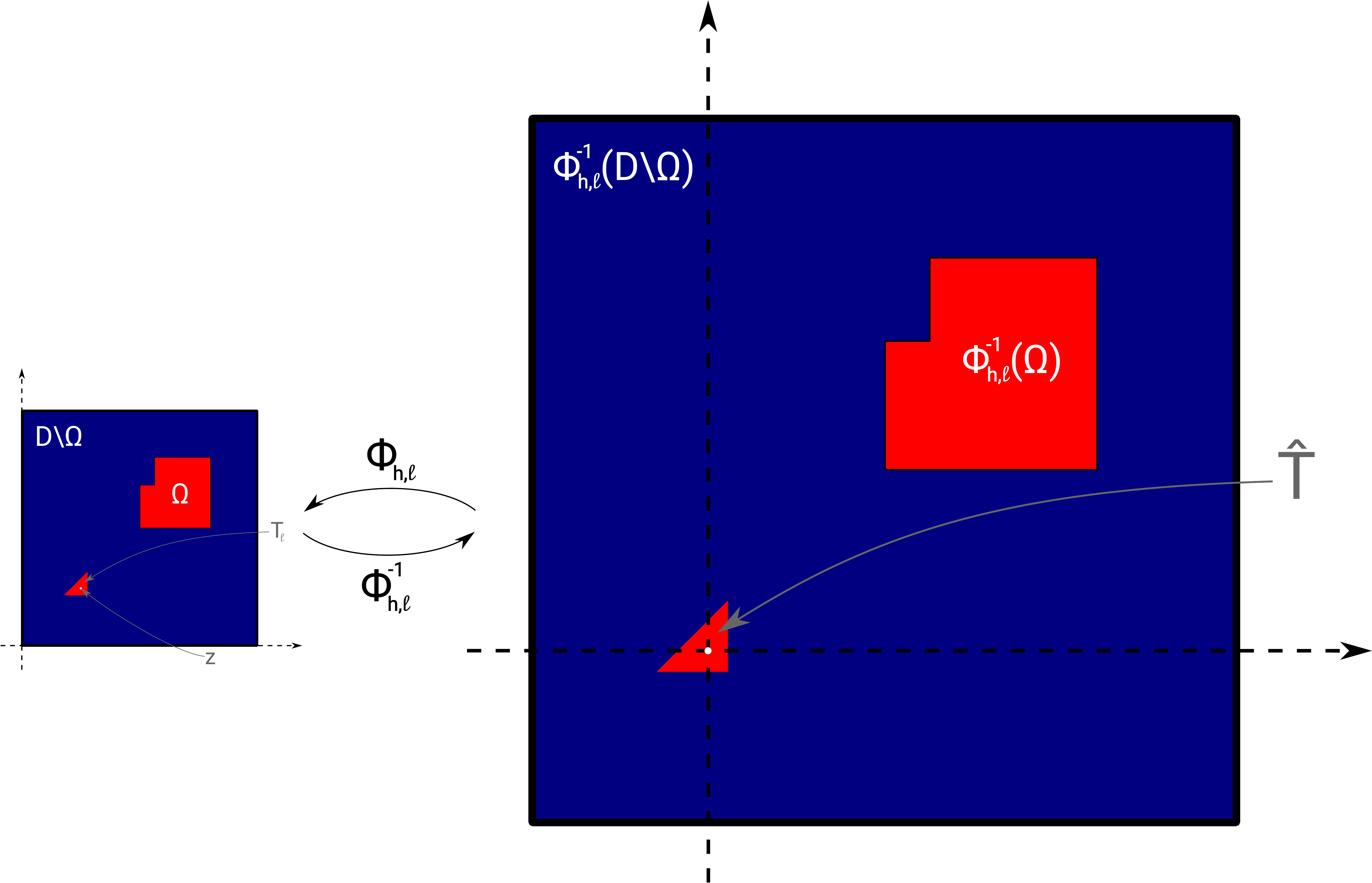}
    \end{center}
    \caption{Rescaled perturbed domain.}
    \label{fig_pert_rescaled}
\end{figure}
\begin{figure}
    \begin{center}
        \begin{tabular}{ccc}
        \includegraphics[width=.5\textwidth]{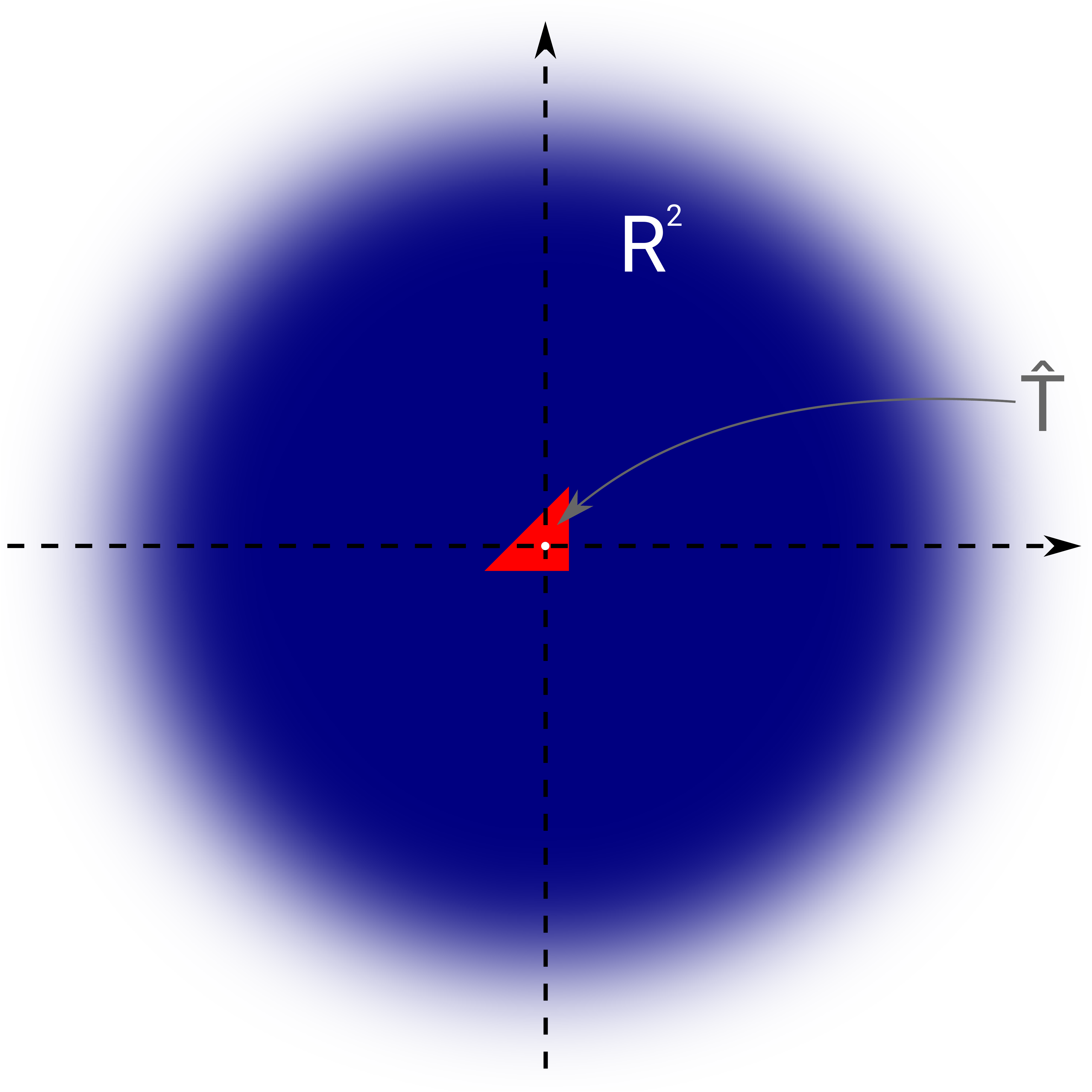} &&
        \includegraphics[width=.5\textwidth]{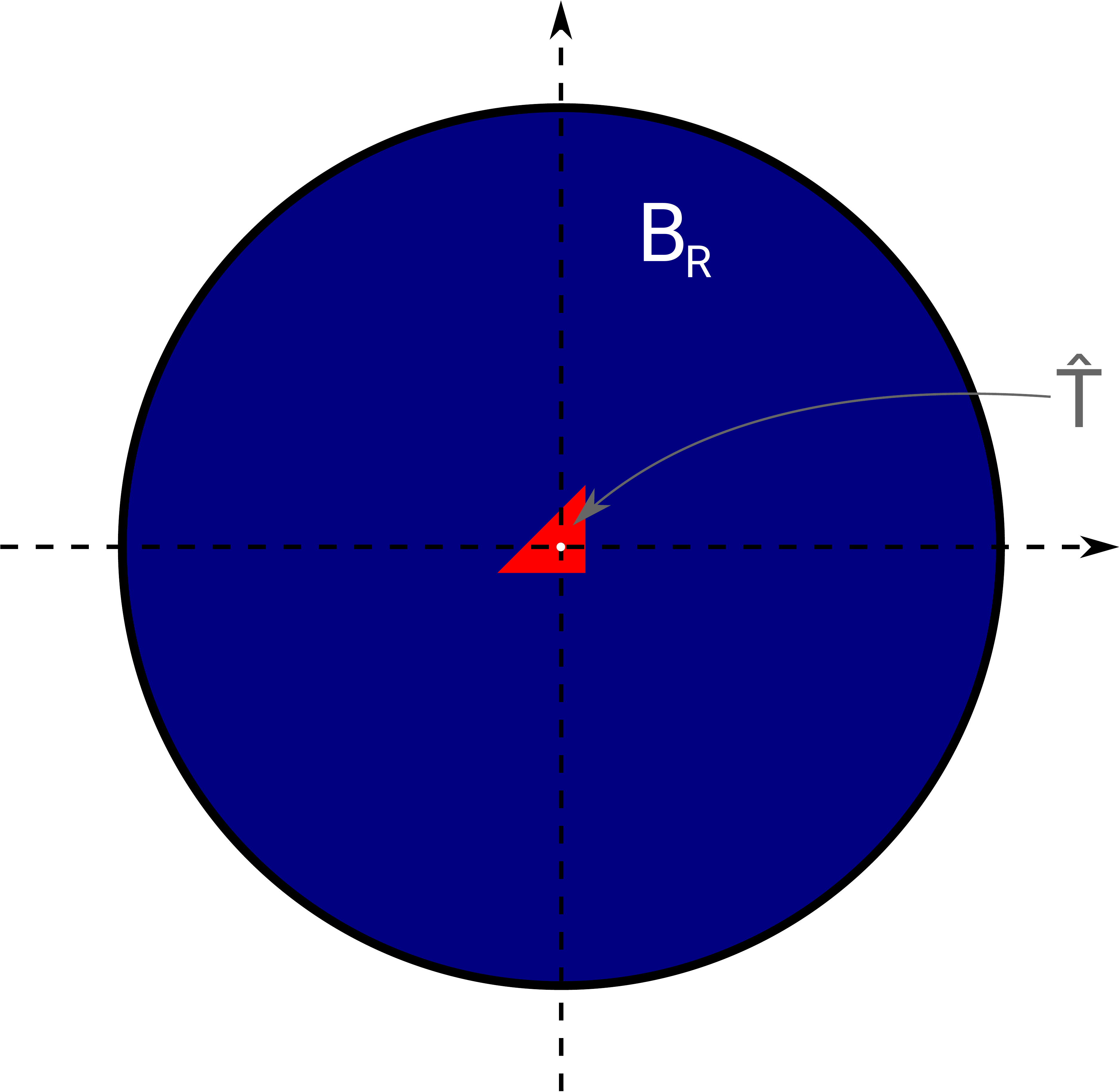}\\
        (a) && (b)
        \end{tabular}
    \end{center}
    \caption{(a) Rescaled perturbed domain after limit $h \rightarrow 0$. (b) Truncation of (a) at radius $R$.}
    \label{fig_pert_rescaled_h0__trunc}
\end{figure}

Obviously, if $\corr_{h,\ell, \hat T}[\lamOut, \lamIn]$ was known exactly then \eqref{eq_derivTD1} would give rise to an exact model for $\mathcal J(\Omega \cup T_\ell)$. However, of course, this would require the solution of a linear problem for every element $\ell$ and is therefore computationally not tractable. Instead, we now aim at obtaining an approximation of the quantity $\corr_{h,\ell, \hat T}[\lamOut, \lamIn]$ that is independent of the particular element index $\ell$. As it is often used in the derivation of topological derivatives, we consider the limit of problem \eqref{eq_Khell_exact} as $h \rightarrow 0$. This leads to the problem to find $ \corr_{\hat T}[\lamOut, \lamIn; \nabla u(z_\ell)] \in X$ satisfying \cite{a_GAST_2020a}
\begin{align} \label{eq_K_That}
    \int_{\mathbb R^2}\lambda_{\hat T } \nabla \corr_{\hat T}[\lamOut, \lamIn;\nabla u(z_\ell)] \cdot \nabla \psi \; \mbox dx = -(\lamIn - \lamOut) \int_{\hat T} \nabla u(z_\ell) \cdot \nabla \psi \; \mbox dx
\end{align}
for all test functions $\psi \in X$ where $X$ is a Beppo-Levi space, see Section \ref{sec_topder}. For an illustration of the corresponding material distribution, see Figure \ref{fig_pert_rescaled_h0__trunc}(a). Note that, if we had access to the exact solution $\corr_{\hat T}[\lamOut, \lamIn;\nabla u(z_\ell)]$ of \eqref{eq_K_That}, the topological derivative at the centroid $z_\ell$ of triangle $T_\ell \in \Dsf \setminus \overline \Omega$ with respect to $\hat T$-shaped inclusion shapes would follow from \eqref{eq_derivTD1} as
\begin{align}
    d \mathcal J[\lamOut, \lamIn](\Omega)(z_\ell, \hat T) &= \underset{h \rightarrow 0}{\mbox{lim}} \frac{\mathcal J(\Omega \cup T_\ell) - \mathcal J(\Omega) }{|T_\ell|} \nonumber \\
    &= (\lamIn - \lamOut) \frac{1}{|\hat T|} \int_{\hat T} (\nabla u(z_\ell) +  \nabla \corr_{\hat T}[\lamOut, \lamIn;\nabla u(z_\ell)](x)) \cdot \nabla p(z_\ell) \; \mbox dx \label{eq_TDtrig}
\end{align}
which coincides with the statement of Proposition \ref{prop_TDomega}. Here we used $|T_\ell| = h^2|\hat T| $.

\subsection{Our proposed topological derivative model}
Unlike in the case of circular or elliptic inclusions, no closed-form solution to the exterior problem \eqref{eq_K_That} for triangular inclusion shapes $\omega = \hat T$ is known in the literature. Thus, formula \eqref{eq_TDtrig} cannot be evaluated exactly. However, as it was shown in \cite{GanglSturm_TDauto}, it is feasible to numerically approximate the exterior problem \eqref{eq_K_That} by truncating the domain at a moderately large radius $R$ (e.g., $R=30$) and using a finite element discretization with homogeneous Dirichlet boundary conditions on the boundary of the truncated domain (see Fig. \ref{fig_pert_rescaled_h0__trunc}(b)). We remark that this truncation is justified, since it is known that the solution to \eqref{eq_K_That} exhibits a certain decay as $|x| \rightarrow \infty$ \cite{GanglSturm_TDauto}.

\subsubsection{Topological derivative model in homogeneous regions} \label{sec_TD_homo}
We restrict ourselves to elements $T_\ell$ in the interior of $\Dsf \setminus \overline \Omega$ such that also all neighboring elements of $T_\ell$ (i.e., elements that share at least one vertex with $T_\ell$) are in $\Dsf \setminus \overline \Omega$.
Of course, all results and statements follow analogously for elements $T_\ell$ in the interior of $\Omega$.
For this setting, we propose the model that is based on the following  procedure:

\begin{enumerate}
    \item Compute a finite element approximation of \eqref{eq_K_That} using a finite element discretization of a truncated domain. More precisely, given a truncation radius $R$ and a mesh $\{\tau_1, \dots, \tau_M\}$ of the truncated domain $B_R(0)$ that resolves the inclusion $\hat T$, we aim to find $\corr_{\hat T,h}[\lamOut, \lamIn;\Be^{(k)}] \in V_{h}^R:= \{v \in C^0(B_R(0)): v|_{\partial B_R(0)}=0, v|_{\tau_i} \in P^1, i=1,\dots, M\}$ such that
    \begin{align} \label{eq_K_That_FEM}
        \int_{B_R(0)}\lambda_{\hat T } \nabla \corr_{\hat T,h}[\lamOut, \lamIn;\Be^{(k)}] \cdot \nabla \psi_h \; \mbox dx = -(\lamIn - \lamOut) \int_{\hat T} \Be^{(k)} \cdot \nabla \psi_h \; \mbox dx
    \end{align}
    for all $\psi_h \in V_{h}^R$ for $k =1$ and $k=2$. Here, recall that $\lambda_{\hat T}(x) = \chi_{\hat T}(x) \lamIn + \chi_{\mathbb R^2 \setminus \hat T}(x) \lamOut$.
    \item Compute the approximate weak polarization matrix
        \begin{align} \label{eq_P_That_FEM}
            \mathcal P_{\hat T, h}[\lamOut, \lamIn] = \left[ \frac{1}{|\hat T|} \int_{\hat T}\nabla \corr_{\hat T,h}[\lamOut, \lamIn;\Be^{(1)}]\mbox dx \quad \frac{1}{|\hat T|} \int_{\hat T}  \nabla \corr_{\hat T,h}[\lamOut, \lamIn;\Be^{(2)}] \mbox dx \right] \in \mathbb R^{2\times 2}.
        \end{align}
    \item Evaluate $d \mathcal J_h[\lamOut, \lamIn](z_\ell, \hat T) = -(\lamIn - \lamOut) (\nabla u_h|_{T_\ell})^\top \left(\BI_2 + \mathcal P_{\hat T, h}[\lamOut, \lamIn] \right) \nabla u_h|_{T_\ell}$.
\end{enumerate}

With this, for a given material distribution $\Blam \in \mathbb R^m$ and $\Beta = \Blam + (\eta - \Blam_\ell) \Be^{(\ell)}$ where $T_\ell$ is in the interior of $\Dsf \setminus \overline \Omega$, we obtain the approximation
\begin{align*}
    \mathcal J(\Beta) \approx & \mathcal J(\Blam) + |T_\ell| d \mathcal J_h[\Blam_\ell,\eta](z_\ell, \hat T) \\
    =& \mathcal J(\Blam) - |T_\ell|(\eta - \Blam_\ell) (\nabla u_h|_{T_\ell})^\top \left(\BI_2 + \mathcal P_{\hat T, h}[\Blam_\ell, \eta] \right) \nabla u_h|_{T_\ell}.
\end{align*}

\begin{remark} \label{rem_oneTrig}
    Concerning the numerical solution of \eqref{eq_K_That_FEM}, we make one important remark. It is essential that the mesh $\{\tau_1, \dots, \tau_M\}$ is chosen in such a way that the triangle $\hat T$ is discretized by exactly one element $\tau_j$ and that, thus, the solution is linear inside the whole of $\hat T$. While a finer discretization of $\hat T$ would yield a better approximation to the true solution of limit problem \eqref{eq_K_That}, the term $\corr_{\hat T,h}[\lamOut, \lamIn; \Be^{(k)}]$ should actually make up for the error $u_h^{(\ell)} - u_h$ inside element $T_\ell$ which is a linear function inside $T_\ell$ due to the chosen discretization. Figure \ref{fig_comparison_K_ref} shows the solution to \eqref{eq_K_That_FEM} when $\hat T$ is resolved by exactly one element and when the whole mesh is twice uniformly refined.
\end{remark}
\begin{figure}
    \centering
    \begin{tabular}{cc}
    \includegraphics[width=.5\textwidth, trim = 50 0 50 0, clip]{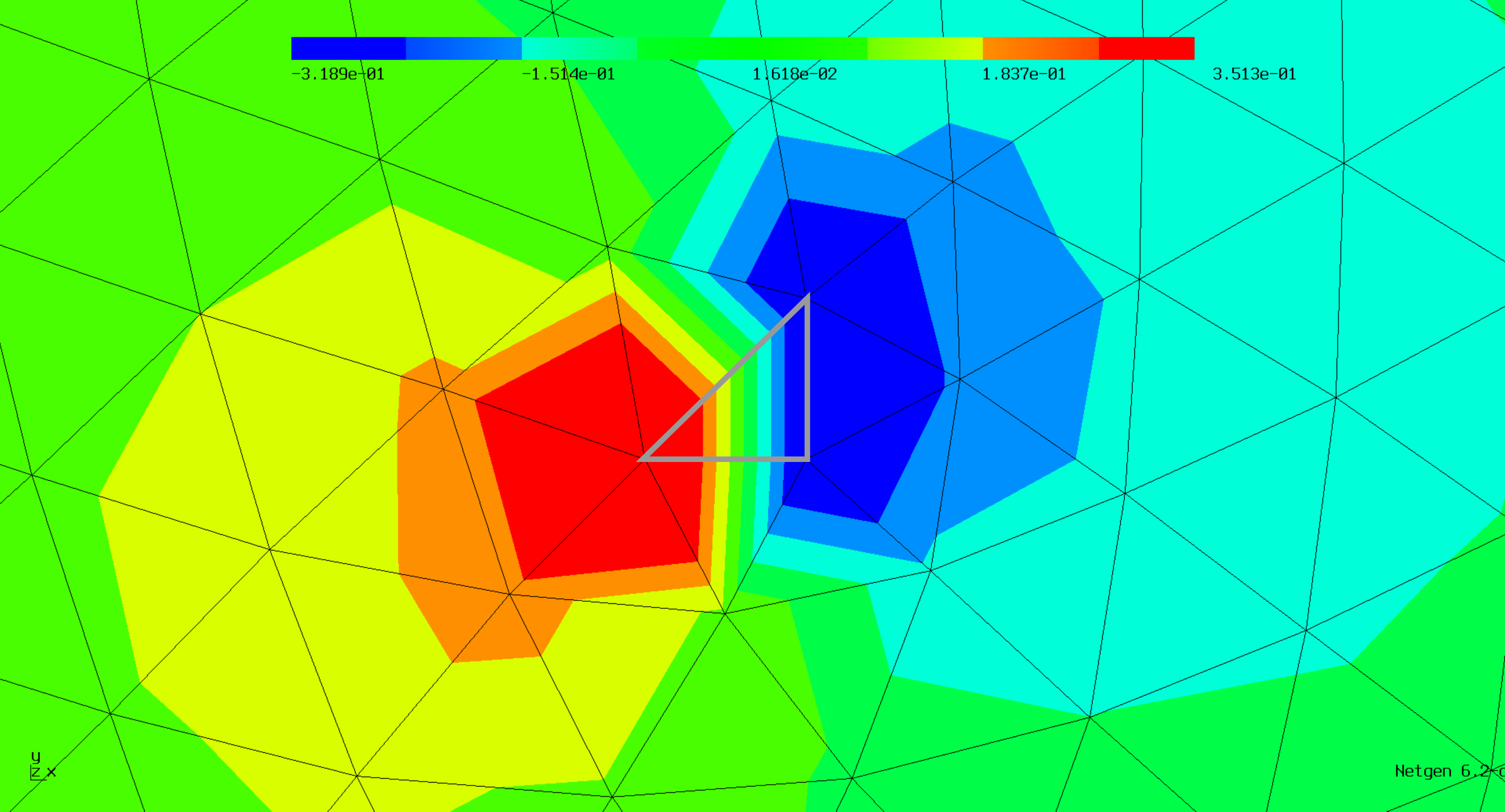} &
    \includegraphics[width=.5\textwidth, trim = 50 0 50 0, clip]{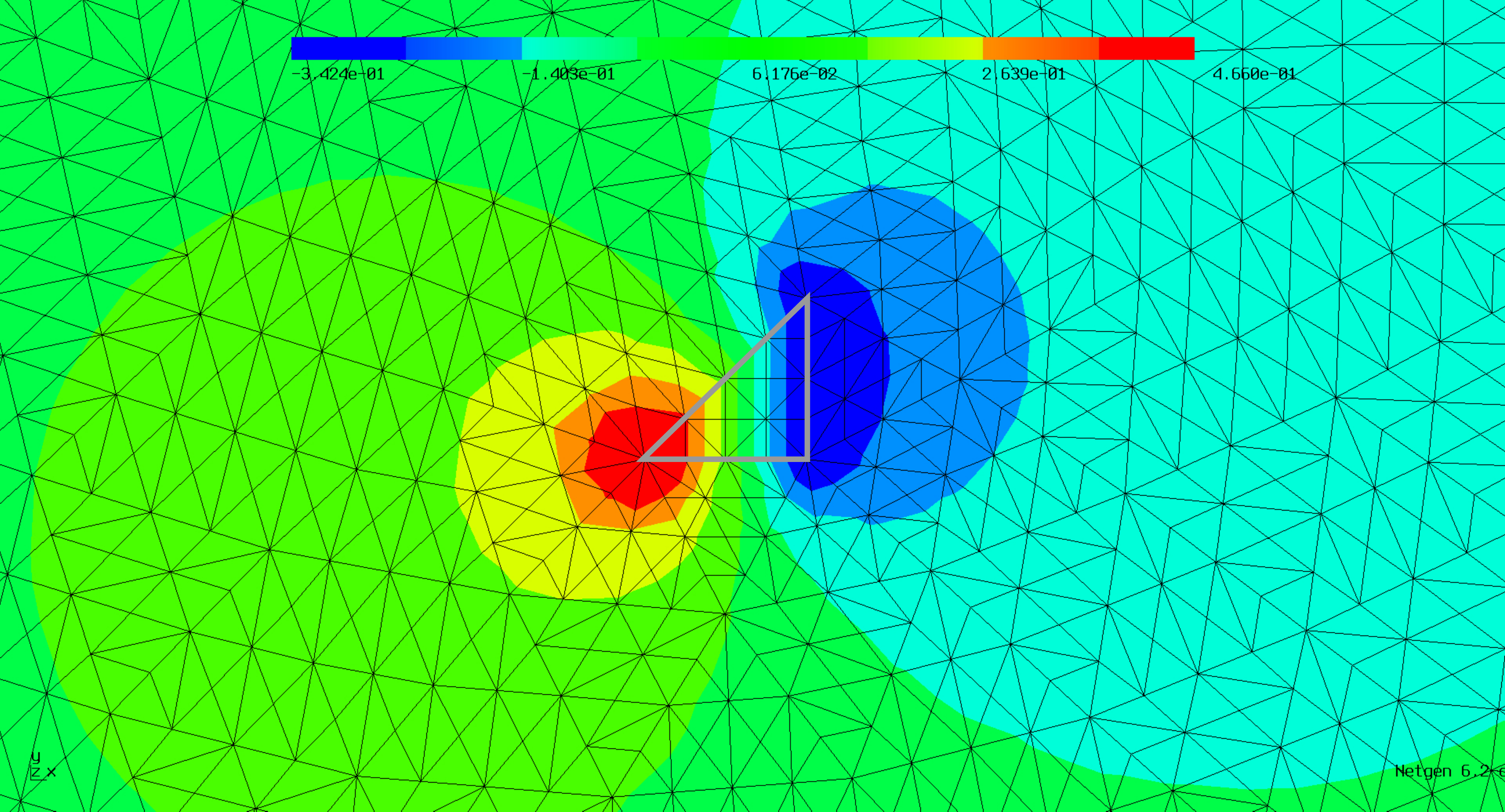} \\
    (a) & (b)
    \end{tabular}
    \caption{Comparison of numerical solution $K_{\hat T, h}[1, 1000;\Be^{(1)}]$ to \eqref{eq_K_That_FEM} on a mesh where $\hat T$ is resolved by exactly one element (a) and on a twice uniformly refined mesh. The mesh in (a) should be used for solving \eqref{eq_K_That_FEM}.} 
    \label{fig_comparison_K_ref}
\end{figure}

\subsubsection{Topological derivative model in inhomogeneous regions} \label{sec_TD_inhomo}
Even if one is interested in binary designs without intermediate materials, in the course of a density-based topology or material optimization procedure, one will of course always encounter regions of intermediate materials. For elements $T_\ell$ in these regions, the assumptions taken at the beginning of Section \ref{sec_TD_homo} are not satisfied and the corresponding proposed model will not be very accurate in these regions. In order to improve the quality of the approximation also in these regions, we recall the idea behind the topological derivative model:  The quantity $\corr_{\hat T}$ should approximate the local variation of the (discretized) state with respect to a material perturbation in some element $T_\ell$, i.e., it should approximate $u_h^{(\ell)}-u_h$. In other words, problem \eqref{eq_K_That} can be interpreted as considering $u_h^{(\ell)}-u_h$ and zooming in around the fixed triangle $T_\ell$ and neglecting everything that is beyond a certain (small) distance from that triangle, see also the illustrations in Fig. \ref{fig_pert_rescaled} and Fig. \ref{fig_pert_rescaled_h0__trunc}.

We follow this idea also in the case of inhomogeneous material around a fixed triangle, i.e., we want to approximate the local material distribution in a truncated rescaled domain $B_R(0)$ similar to Fig. \ref{fig_pert_rescaled_h0__trunc}(b). For obtaining an approximation of the inhomogeneous material distribution within the computational domain $\Dsf$, we divide the domain $B_R(0)$ into three sectors. The sectors are separated by three lines which are chosen such as to halve the three interior angles of the triangle $\hat T$. Thus, we end up with a domain $B_R(0)$ similar to the one depicted in Fig. \ref{fig_pert_rescaled_h0__trunc}(b) which is occupied by four different materials (one inside the triangle $\hat T$ and one in each of the three sectors), see Fig. \ref{fig_illustrateAvg}(b). For the computation of average values within one of the sectors, we take a weighted H\"older average of the values in the neighboring elements with parameter $\alpha$, i.e., the averaged value in Sector $j$, $j=1,2,3$, is chosen as
\begin{align} \label{eq_lambdaSj}
    \lambda^{S_j}_{T_\ell} = \left( \sum_{T \in \mathcal N(T_\ell)} w_{S_j, T} (\Blam_T)^\alpha \right)^{\frac1\alpha}.
\end{align}
Here, $\mathcal N(T_\ell)$ denotes the set of triangles that have at least one common vertex with triangle $T_\ell$, see Figure \ref{fig_illustrateAvg}(a), and $w_{S_j, T} = |T \cap S_j|/|T| \in [0,1]$ is the volume fraction of triangle $T$ in Sector $S_j$. Moreover, $\Blam_T$ denotes the entry of the vector $\Blam$ corresponding to the triangle $T$. 
In our experiments, we chose the H\"older parameter as $\alpha=-0.5$. This choice will be motivated later in Remark \ref{rem_hoelderPara} of Section \ref{sec_numExp}.

Our proposed model in the case of inhomogeneous material around an element $T_\ell$ with material coefficient $\Blam_\ell$ and averaged sector values $\lambda^{S_1}_{T_\ell}, \lambda^{S_2}_{T_\ell},\lambda^{S_3}_{T_\ell}$ according to \eqref{eq_lambdaSj} follows the same three steps outlined for the homogeneous setting in Section \ref{sec_TD_homo}: We numerically compute the corresponding correctors $\corr_{\hat T, h}[(\lamOut, \lambda^{S_1}_{T_\ell}, \lambda^{S_2}_{T_\ell},\lambda^{S_3}_{T_\ell}), \lamIn; \Be^{(k)}]$ for $k=1,2$ as the finite element solutions to \eqref{eq_K_That_FEM} where $\lambda_{\hat T}$ is replaced by the new three-sector material distribution
\begin{align} \label{eq_lambdaThat_secs}
    \lambda_{\hat T}(x) = \chi_{\hat T}(x) \lamIn + \sum_{j=1}^3\chi_{S_j}(x) \lambda^{S_j}_{T_\ell}.
\end{align} Subsequently, the corresponding weak polarization matrix $\mathcal P_{\hat T, h}[(\lamOut, \lambda^{S_1}_{T_\ell}, \lambda^{S_2}_{T_\ell},\lambda^{S_3}_{T_\ell}), \lamIn]$ and the quantity $d\mathcal J_h[(\lamOut, \lambda^{S_1}_{T_\ell}, \lambda^{S_2}_{T_\ell},\lambda^{S_3}_{T_\ell}), \lamIn](z_\ell, \hat T)$ can be computed according to steps 2 and 3 of Section \ref{sec_TD_homo}.

\begin{figure}
    \begin{tabular}{cc}
        \includegraphics[width=.5\textwidth]{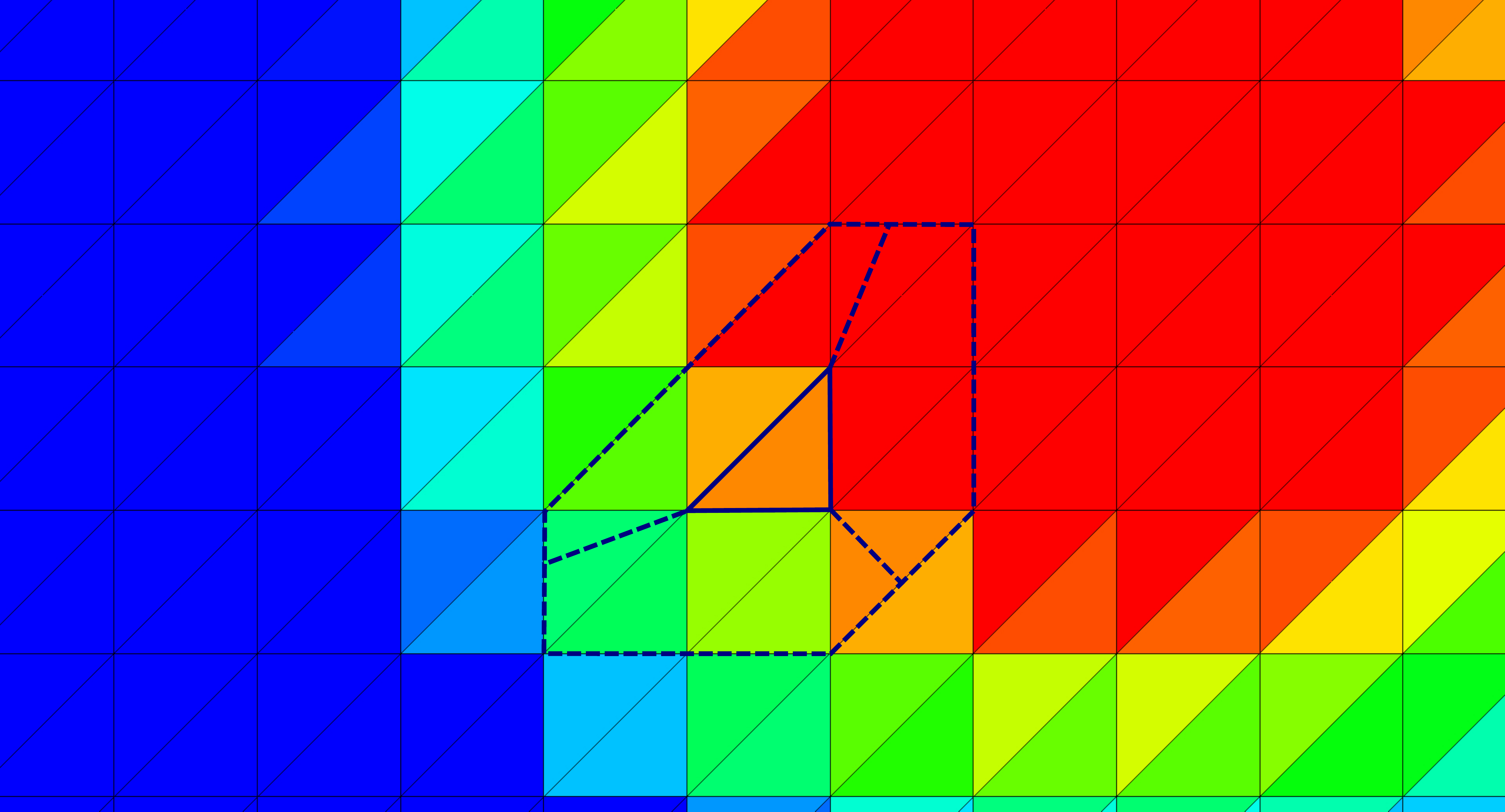}&
        \includegraphics[width=.5\textwidth]{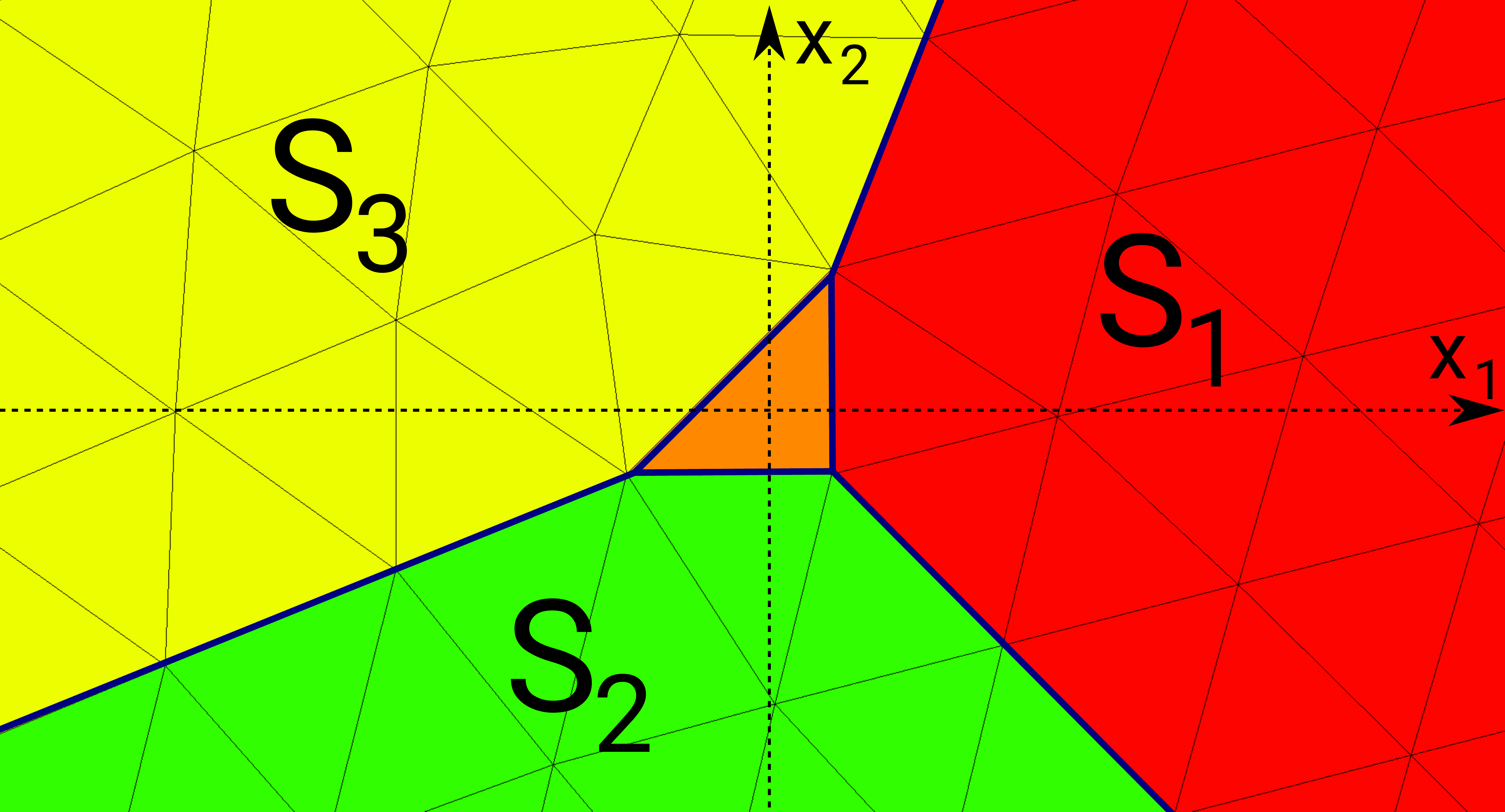} \\
        (a) & (b)
    \end{tabular}
    \caption{(a) Element inside inhomogeneous material distribution in computational domain $\Dsf$. Neighborhood for averaging into sector values is marked. (b) Averaged material distribution in three sectors of truncated unbounded domain $B_R(0)$. The average values per sector are obtained by a weighted H\"older mean of the material values in the neighboring elements.}
    \label{fig_illustrateAvg}
\end{figure}

Summarizing, we define the model
\begin{align} \label{eq_defTDnum_model}
    \hat{\mathcal J}_{\text{TDnum}}(\Beta) := \mathcal J(\Blam) - \sum_{\ell = 1}^m |T_\ell|(\Beta_\ell - \Blam_\ell) (\nabla u_h|_{T_\ell})^\top \left(\BI_2 + \mathcal P_{\hat T, h}[(\Blam_\ell, \lambda^{S_1}_{T_\ell}, \lambda^{S_2}_{T_\ell},\lambda^{S_3}_{T_\ell}), \Beta_\ell] \right) \nabla u_h|_{T_\ell}.
\end{align}
Note that the homogeneous setting of Section \ref{sec_TD_homo} is actually also covered by the more general inhomogeneous setting treated here.

We next make an important remark concerning the efficient evaluation of the model $\hat{\mathcal J}_{\text{TDnum}}$.
\begin{remark} \label{rem_precomputation}
    In the general inhomogeneous setting, the procedure described in this section states that, in order to evaluate model \eqref{eq_defTDnum_model}, a problem of type \eqref{eq_K_That_FEM} has to be solved for each element $T_\ell$. Of course, this is computationally expensive and therefore not recommended by the authors. Instead, the approach followed here is to divide the numerical computations into an offline and an online phase. In the offline phase, which has to be performed only once for the particular type of PDE operator, we compute the quantities
    \begin{align*}
        \corr_{\hat T, h}[(\lamOut, \lambda^{S_1}, \lambda^{S_2},\lambda^{S_3}), \lamIn; \Be^{(k)}]
    \end{align*}
    for $k=1,2$ and for a large number of combinations of relevant values $(\lamOut, \lambda^{S_1}_{T_\ell}, \lambda^{S_2}_{T_\ell},\lambda^{S_3}_{T_\ell}, \lamIn)$ and store the corresponding weak polarization matrices. This, initially, results in a five-dimensional array of $2\times2$ matrices. Moreover, these precomputations should be done for each type of reference triangle, i.e., in our case for $\hat T = \hat T^{(1)}$ and $\hat T = \hat T^{(2)}$, see Fig. \ref{fig_mesh_elTypes}.

    In the online phase, for each element $T_\ell$ the corresponding average sector values are computed according to \eqref{eq_lambdaSj} and the weak polarization matrix $\mathcal P_{\hat T, h}[(\Blam_\ell, \lambda^{S_1}_{T_\ell}, \lambda^{S_2}_{T_\ell},\lambda^{S_3}_{T_\ell}), \Beta_\ell] $ is approximately obtained by piecewise linear interpolation of the precomputed values.
\end{remark}

We finally remark that the precomputation can be reduced from five to four dimensions by exploiting that problem \eqref{eq_K_That_FEM} with $\lambda_{\hat T}$ according to \eqref{eq_lambdaThat_secs} depends on the parameter $\lamOut$ only via the scaling of the right hand side by $(\lamIn - \lamOut)$.

\section{An improved separable model based on the Sherman-Morrison-Woodbury formula} \label{sec_SMW_W}

In this section, we revisit the separable exact model defined in \eqref{eq_SMWexact} and have a closer look at the matrix
\begin{align*}
    \BGam^{(\ell)} = -\Bell^\top \BK(\Blam)^{-1} \Bell \in \mathbb R^{2\times 2}.
\end{align*}
Recall that, in order to employ model \eqref{eq_SMWexact}, this matrix would have to be evaluated for each element index $\ell$, which amounts to solving $m$ many systems of linear equations and is thus computationally prohibitive. Motivated by the procedure of Section \ref{sec_modelTD}, our goal here is to find a good approximation of $\BGam^{(\ell)}$ that is independent of the element index $\ell$ and can thus be precomputed in an offline phase.

We begin by making the following observation. We assume the finite element setting introduced in Section \ref{sec_discmodelprob} with the mesh $\mathcal T$ and the finite element space $V_h\subset H^1_{\Gamma_D}(\Dsf)$ of piecewise linear and globally continuous functions.
\begin{lemma} \label{lem_Gamma_w}
    Let $T_\ell \in \mathcal T$ and, for $k=1,2$, define $w_{k,h} \in V_h$ the unique numerical solution to the variational problem
    \begin{align} \label{eq_Uek}
        \int_D \lambda(x) \nabla w_{k,h} \cdot\nabla v_h \; \mbox dx = - \int_{T_\ell} \Be^{(k)} \cdot \nabla v_h \; \mbox dx
    \end{align}
    for all $v_h \in V_h$. Then it holds
    \begin{align}
        \BGam^{(\ell)} = [ \nabla w_{1,h}|_{T_\ell}\; \nabla w_{2,h}|_{T_\ell} ] .
    \end{align}
\end{lemma}

\begin{proof}
    We use the notation and symbols introduced in Section \ref{sec_discmodelprob}. The discretization of \eqref{eq_Uek} reads
    \begin{align*}
        \BK(\Blam) \Bw^{(k)} = \Bf^{(\ell,k)}
    \end{align*}
    where $\BK(\Blam)$ is the invertible stiffness matrix, $\Bw^{(k)}$ denotes the coefficient vector of the finite element function $w_{k,h} = \sum_{i=1}^n \Bw^{(k)}_i \varphi_i$, $k\in \{1,2\}$, and $\Bf^{(\ell,k)} \in \mathbb R^n$ with
    \begin{align*}
       (\Bf^{(\ell,k)})_i =& -\int_{T_\ell} \Be^{(k)} \cdot \nabla \varphi_i \; \mbox dx, \; i=1,\dots n.
    \end{align*}
    Since the global load vector $\Bf^{(\ell, k)}$ has contributions only from one element, it holds $\Bf^{(\ell, k)} = \Bellti \Bf^{\ell, k}_{\text{loc}}$ with the element load vector
    \begin{align*}
         \Bf^{\ell, k}_{\text{loc}} = - |T_\ell| \begin{pmatrix} -1 & -1 \\ 1 & 0 \\ 0 & 1 \end{pmatrix} \BJ_\ell^{-1} \Be^{(k)} = - \sqrt{|T_\ell|} \BD_\ell \Be^{(k)}.
    \end{align*}
    Thus, it holds
    \begin{align} \label{eq_proof_wk_1}
        \Bell^\top \Bw^{(k)} = \Bell^\top \BK(\Blam)^{-1} \Bf^{(\ell,k)} = \Bell^\top \BK(\Blam)^{-1} \Bellti \Bf^{\ell, k}_{\text{loc}} = - \sqrt{|T_\ell|} \Bell^\top \BK(\Blam)^{-1} \Bell \Be^{(k)}.
    \end{align}
    On the other hand, we know from \eqref{eq_Bltu} that
    \begin{align} \label{eq_proof_wk_2}
        \Bell^\top \Bw^{(k)} = \sqrt{|T_\ell|} \, \nabla w_{k,h}|_{T_\ell}.
    \end{align}
    Comparing \eqref{eq_proof_wk_1} and \eqref{eq_proof_wk_2} for $k=1$ and $k=2$ yields the result.
\end{proof}

Lemma \ref{lem_Gamma_w} gives an interpretation of the matrix $\BGam^{(\ell)}$, which appears in the Sherman-Morrison-Woodbury model \eqref{eq_SMWexact} and is costly to evaluate, in terms of a boundary value problem. In order to find an approximation of $\BGam^{(\ell)}$ that is independent of the element index $\ell$, we proceed similarly to Section \ref{sec_modelTD}. In boundary value problem \eqref{eq_Uek}, we zoom in around the element $T_\ell$, i.e., we apply the transformation $\Phi_{h,\ell}^{-1}$ that transforms $T_\ell$ to the reference element $\hat T$, see Fig. \ref{fig_smwapprox_pert_rescaled} for an illustration in a homogeneous setting. Note that, as opposed to the procedure in Section \ref{sec_modelTD}, here an unperturbed material distribution is transformed.
\begin{figure}
    \begin{center}
        \includegraphics[width=\textwidth]{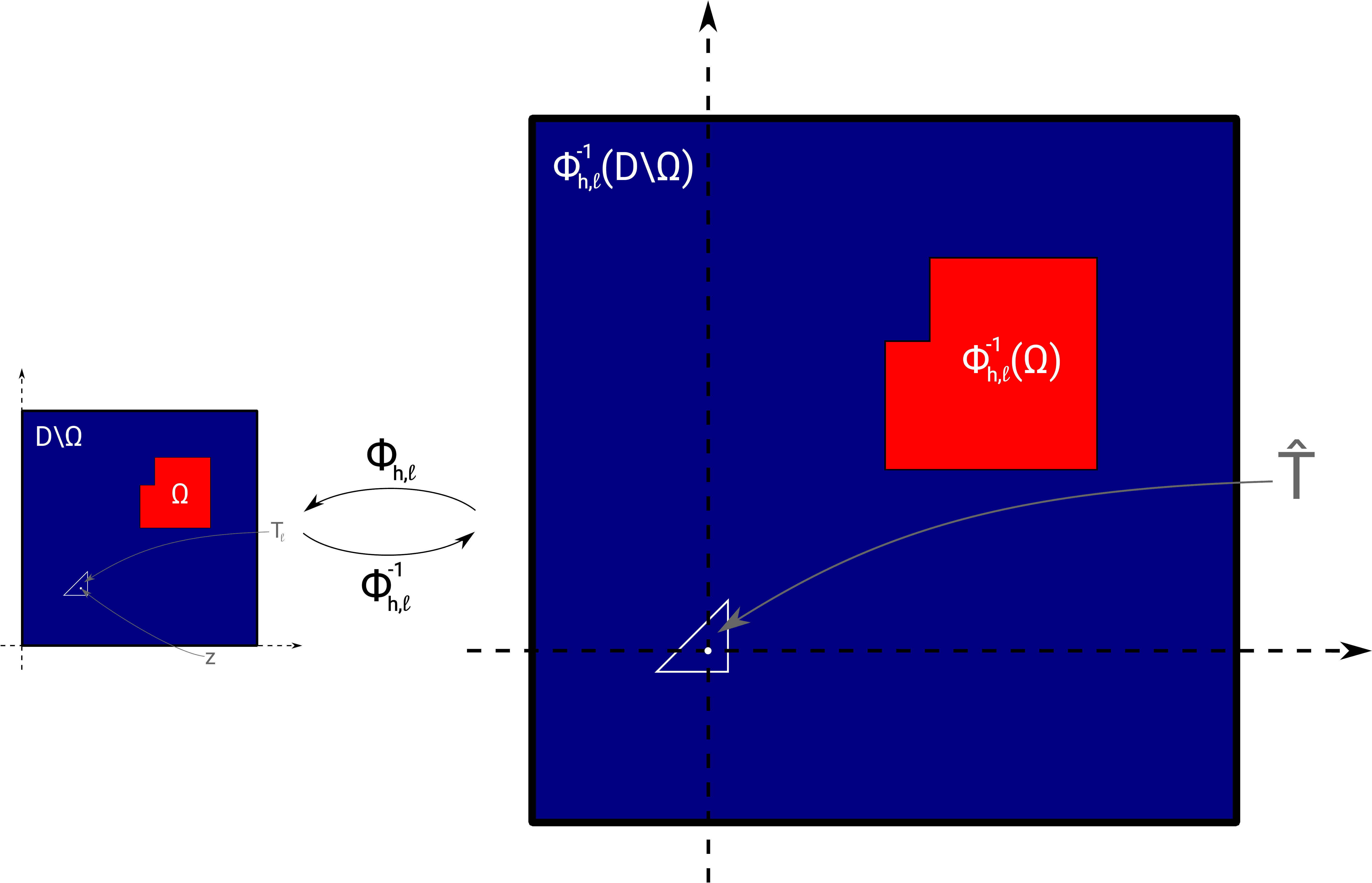}
    \end{center}
    \caption{Rescaled unperturbed domain.}
    \label{fig_smwapprox_pert_rescaled}
\end{figure}

Passing to the limit $h \rightarrow 0$ yields an exterior problem on the unbounded domain, see Fig. \ref{fig_smwapprox_pert_rescaled_h0__trunc}(a) and again truncating this domain leads to the boundary value problem on the truncated domain $B_R(0)$ to find
$W_{\hat T}[\lamOut;\Be^{(k)}] \in H^1_0(B_R(0))$, $k=1,2$, as the unique solution to
\begin{align} \label{eq_W_That_homo}
    \int_{B_R(0)} \lamOut \nabla W_{\hat T}[\lamOut;\Be^{(k)}] \cdot \nabla \psi \; \mbox dx = - \int_{\hat T} \Be^{(k)} \cdot \nabla \psi \; \mbox dx
\end{align}
for all $\psi \in H^1_0(B_R(0))$, see Fig. \ref{fig_smwapprox_pert_rescaled_h0__trunc}(b). Note that this problem differs from problem \eqref{eq_K_That_FEM} only by a different scaling factor on the right hand side and by a homogeneous material distribution $\lamOut$.

In the case when $T_\ell$ is in an inhomogeneous region of the computational domain $\Dsf$ (i.e., not all neighbors of $T_\ell$ have the same material coefficient), we can follow the same averaging procedure with three sectors as in Section \ref{sec_TD_inhomo} and obtain the problem to find $W_{\hat T}[(\Blam_\ell, \lambda_{T_\ell}^{S_1}, \lambda_{T_\ell}^{S_2}, \lambda_{T_\ell}^{S_3});\Be^{(k)}] \in H^1_0(B_R(0))$, $k=1,2$, as the unique solution to
\begin{align} \label{eq_W_That_inhomo}
    \int_{B_R(0)} \lambda_{\hat T} \nabla W_{\hat T}[(\Blam_\ell, \lambda_{T_\ell}^{S_1}, \lambda_{T_\ell}^{S_2}, \lambda_{T_\ell}^{S_3});\Be^{(k)}] \cdot \nabla \psi \; \mbox dx = - \int_{\hat T} \Be^{(k)} \cdot \nabla \psi \; \mbox dx
\end{align}
where $\lambda_{\hat T}(x) = \chi_{\hat T}(x) \Blam_\ell + \sum_{j=1}^3 \chi_{S_j}(x) \lambda_{T_\ell}^{S_j}$, cf. also the material distribution in Fig. \ref{fig_illustrateAvg}. Remark \ref{rem_oneTrig} concerning the numerical approximation of \eqref{eq_W_That_homo} and \eqref{eq_W_That_inhomo} with a mesh where the subdomain $\hat T$ of $B_R(0)$ is discretized by exactly one element remains valid.
We define the $2\times 2$ matrix
\begin{align} \label{eq_Gamma_approx}
\begin{aligned}
    \BGam_{\hat T,\ell} :=& \BGam_{\hat T}[(\Blam_\ell, \lambda_{T_\ell}^{S_1}, \lambda_{T_\ell}^{S_2}, \lambda_{T_\ell}^{S_3})] \\
    :=& \left[ \frac{1}{\hat T} \int_{\hat T}\nabla W_{\hat T}[(\Blam_\ell, \lambda_{T_\ell}^{S_1}, \lambda_{T_\ell}^{S_2}, \lambda_{T_\ell}^{S_3});\Be^{(1)}] \mbox dx \quad  \frac{1}{\hat T} \int_{\hat T}\nabla W_{\hat T}[(\Blam_\ell, \lambda_{T_\ell}^{S_1}, \lambda_{T_\ell}^{S_2}, \lambda_{T_\ell}^{S_3});\Be^{(2)}]  \mbox dx\right],
\end{aligned}
\end{align}
and remark that, in the same way as pointed out in Remark \ref{rem_precomputation}, we can also precompute the matrices $\BGam_{\hat T,\ell}$ for a four-dimensional array of values in an offline phase and interpolate them efficiently in the online phase.
This way, we get the separable model
\begin{align} \label{eq_SMW_W}
    \hat{\mathcal J}_{\text{SMWapprox}}(\Beta) := \mathcal J(\Blam) -\sum_{\ell=1}^m   |T_\ell|(\Beta_\ell -  \Blam_{\ell}) (\nabla u_h|_{T_\ell})^\top \left( \BI_2 - (\Beta_\ell -  \Blam_{\ell}) \BGam_{\hat T, \ell}\right)^{-1} \nabla u_h|_{T_\ell}
\end{align}
as an approximation to the separable exact model \eqref{eq_SMWexact}.

\begin{figure}
    \begin{center}
        \begin{tabular}{ccc}
        \includegraphics[width=.5\textwidth]{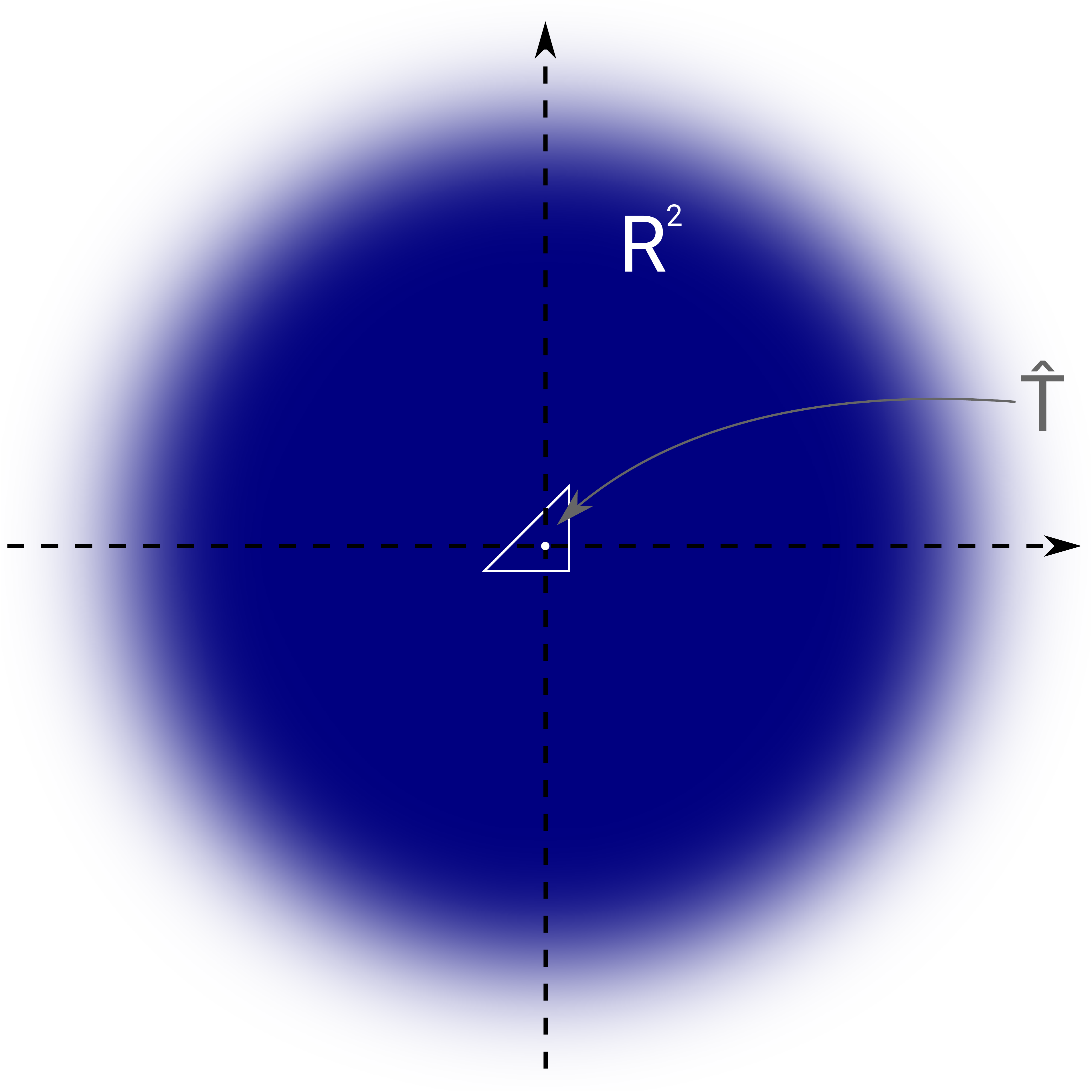} &&
        \includegraphics[width=.5\textwidth]{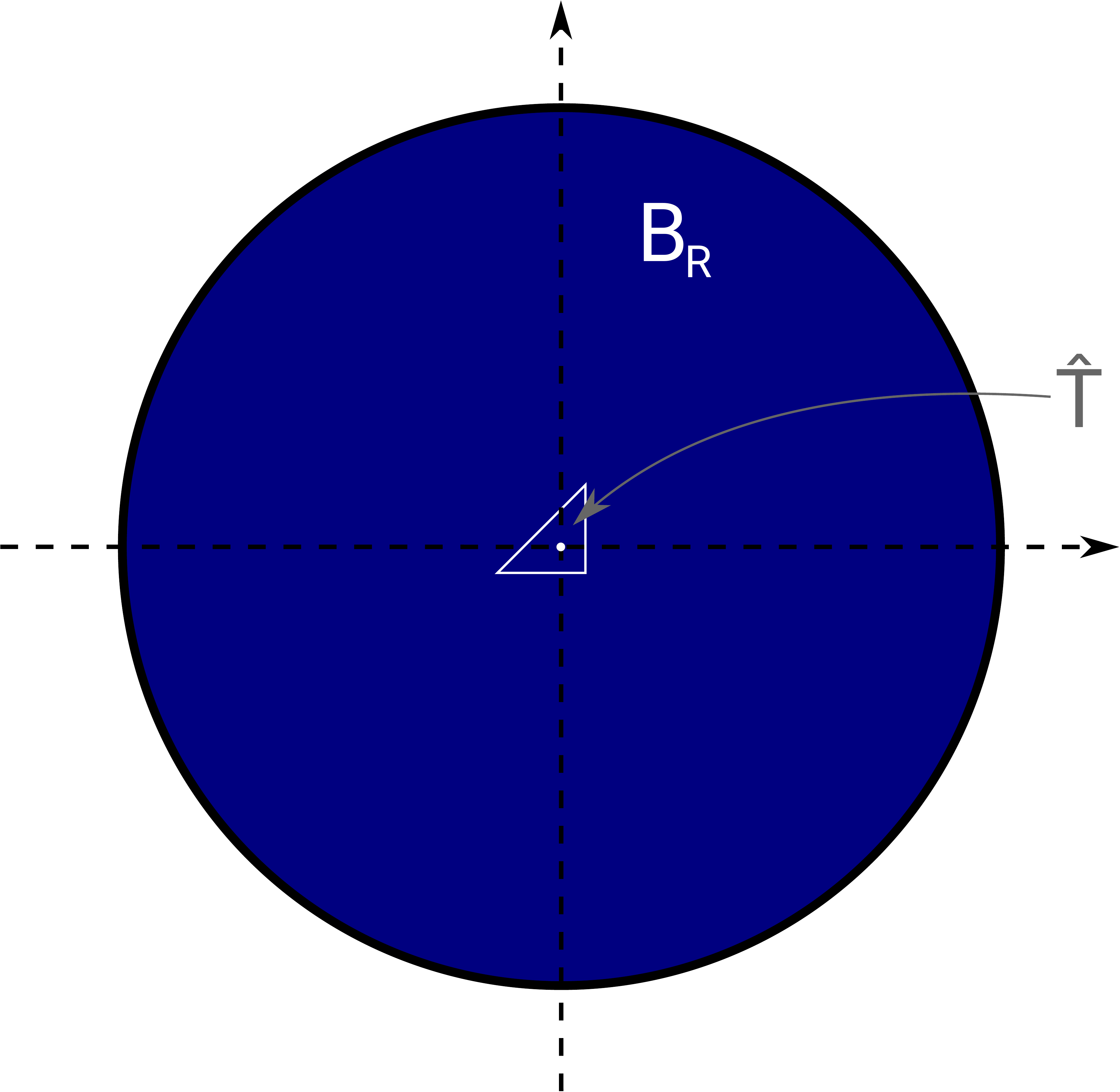}\\
        (a) && (b)
        \end{tabular}
    \end{center}
    \caption{(a) Unbounded domain with reference triangle $\hat T$. (b) Truncated domain $B_R(0)$ with reference triangle $\hat T$.}
    \label{fig_smwapprox_pert_rescaled_h0__trunc}
\end{figure}

\section{Relationships} \label{sec_relations}

We investigate the relationship between the model \eqref{eq_defTDnum_model} of Section \ref{sec_modelTD} that is motivated by the continuous concept of topological derivatives and the model \eqref{eq_SMW_W} of Section \ref{sec_SMW_W} which is meant to approximate the Sherman-Morrison-Woodbury matrix identity model \eqref{eq_SMWexact}. We restrict our presentation to the case of homogeneous material distribution around the fixed element $T_\ell \in  \mathcal T$.

We start by recalling the definitions of the discretized weak polarization matrix $\mathcal P_{\hat T, h}[\lamOut, \lamIn]$ \eqref{eq_P_That_FEM} and the discretization of the matrix $\BGam_{\hat T}$ \eqref{eq_Gamma_approx} in the homogeneous setting, 
\begin{align}
    \mathcal P_{\hat T, h}[\lamOut, \lamIn] =& \left[ \frac{1}{|\hat T|} \int_{\hat T}\nabla \corr_{\hat T,h}[\lamOut, \lamIn;\Be^{(1)}]\mbox dx \right.&&\left. \frac{1}{|\hat T|} \int_{\hat T}  \nabla \corr_{\hat T,h}[\lamOut, \lamIn;\Be^{(2)}] \mbox dx \right] \\
    \BGam_{\hat T, h}[\lamOut] =& \left[ \frac{1}{\hat T} \int_{\hat T}\nabla W_{\hat T, h}[\lamOut;\Be^{(1)}] \mbox dx \right.&&\left.  \frac{1}{\hat T} \int_{\hat T}\nabla W_{\hat T, h}[\lamOut;\Be^{(2)}]  \mbox dx\right]
\end{align}
where $K_{\hat T, h}[\lamOut, \lamIn; \Be^{(k)}] \in V_{h}^R$ is the solution to \eqref{eq_K_That_FEM} and $W_{\hat T, h}[\lamOut; \Be^{(k)}] \in V_{h}^R$ is the finite element approximation to \eqref{eq_W_That_homo}, i.e.,
\begin{align}
     \int_{B_R(0)}\lambda_{\hat T } \nabla \corr_{\hat T,h}[\lamOut, \lamIn;\Be^{(k)}] \cdot \nabla \psi_h \; \mbox dx =& -(\lamIn - \lamOut) \int_{\hat T} \Be^{(k)} \cdot \nabla \psi_h \; \mbox dx \label{eq_K_homo_discr}\\
    \int_{B_R(0)} \lamOut \nabla W_{\hat T,h}[\lamOut;\Be^{(k)}] \cdot \nabla \psi_h \; \mbox dx =& - \int_{\hat T} \Be^{(k)} \cdot \nabla \psi_h \; \mbox dx \label{eq_W_homo_discr}
\end{align}
for all $\psi_h \in V_{h}^R$. Here, recall that $\lambda_{\hat T}(x) = \chi_{\hat T}(x) \lamIn + \chi_{B_R(0) \setminus \hat T}(x) \lamOut$.

We show the following relation between $\mathcal P_{\hat T, h}[\lamOut, \lamIn]$ and $\BGam_{\hat T, h}[\lamOut]$:
\begin{theorem} \label{thm_P_vs_Gamma}
    It holds that
    \begin{align} \label{eq_P_Gamma_identity1}
        \mathcal P_{\hat T, h}[\lamOut, \lamIn] = (\lamIn - \lamOut) \BGam_{\hat T, h}[\lamOut] \left(\BI_2 - (\lamIn - \lamOut) \BGam_{\hat T, h}[\lamOut] \right)^{-1}
    \end{align}
    and further
    \begin{align} \label{eq_P_Gamma_identity2}
        \BI_2 +\mathcal P_{\hat T, h}[\lamOut, \lamIn] =  (\BI_2 - (\lamIn - \lamOut) \BGam_{\hat T, h}[\lamOut] )^{-1}.
    \end{align}
\end{theorem}

\begin{proof}

Recall that we use piecewise linear and globally continuous finite elements on a triangular mesh of $M$ elements of $B_R(0)$ where, according to Remark \ref{rem_oneTrig}, the subdomain $\hat T$ is resolved by exactly one triangle of the mesh. Let now the finite element stiffness matrix of \eqref{eq_K_homo_discr} be denoted by $\tilde \BA$ and the one of \eqref{eq_W_homo_discr} by $\BA$ where we use the same mesh and finite element space for both equations. Let $\hat \ell \in \{1, \dots, M\}$ be the element index corresponding to the triangle $\hat T$. Note that the material distribution in \eqref{eq_K_homo_discr} differs from that in \eqref{eq_W_homo_discr} only in element $\hat \ell$ and we have $\BA = \sum_{k=1}^M \lamOut B_k B_k^\top$ and
\begin{align*}
    \tilde \BA = \BA + (\lamIn - \lamOut) \Bellh \Bellh^\top,
\end{align*}
thus, an application of the Sherman-Morrison-Woodbury formula of Lemma \ref{lem_SMW} yields
\begin{align} \label{eq_Atilde_SMW}
    \tilde \BA^{-1} = \BA^{-1} - (\lamIn - \lamOut) \BA^{-1} \Bellh \left(\BI_2 + (\lamIn- \lamOut) \Bellh^\top \BA^{-1} \Bellh \right)^{-1} \Bellh^\top \BA^{-1}.
 \end{align}
 On the other hand, we know from Lemma \ref{lem_Gamma_w} that for the chosen piecewise linear finite elements where $\hat T$ is resolved by only one triangle (i.e. $\nabla W_{\hat T,h}[\lamOut; \Be^{(k)}](x)$, $\nabla K_{\hat T, h}[\lamOut, \lamIn;\Be^{(k)}](x)$ are constant on $\hat T$), we have
 \begin{align*}
     (\lamIn - \lamOut) \BGam_{\hat T,h}[\lamOut] = - (\lamIn - \lamOut) \Bellh^\top \BA^{-1} \Bellh \quad \mbox{ and } \quad \mathcal P_{\hat T,h}[\lamOut, \lamIn]= - (\lamIn - \lamOut) \Bellh^\top \tilde \BA^{-1} \Bellh.
 \end{align*}
Thus, denoting $\BGam_{\hat T, h}^\lambda := (\lamIn-\lamOut)\BGam_{\hat T, h}[\lamOut]$ and plugging in \eqref{eq_Atilde_SMW} yields
\begin{align*}
     \mathcal P_{\hat T,h}[\lamOut, \lamIn] =& \BGam_{\hat T, h}^\lambda + \BGam_{\hat T, h}^\lambda(\BI_2 - \BGam_{\hat T, h}^\lambda)^{-1} \BGam_{\hat T, h}^\lambda \\
     =& \BGam_{\hat T, h}^\lambda (\BI_2 +(\BI_2 - \BGam_{\hat T, h}^\lambda)^{-1} \BGam_{\hat T, h}^\lambda) \\
     =& \BGam_{\hat T, h}^\lambda (\BI_2 - \BGam_{\hat T, h}^\lambda)^{-1}
\end{align*}
where we used the identity $ (\BI - \BB)^{-1} = \BI + (\BI -\BB)^{-1} \BB$ for any matrix $\BB$ such that $\BI-\BB$ is invertible in the last step. This proves \eqref{eq_P_Gamma_identity1}. In order to see \eqref{eq_P_Gamma_identity2} note that, by \eqref{eq_P_Gamma_identity1}, it holds $\BI_2 + \mathcal P_{\hat T,h}[\lamOut, \lamIn] = \BI_2 + \BGam_{\hat T, h}^\lambda  (\BI_2 - \BGam_{\hat T, h}^\lambda )^{-1} = (\BI_2 - \BGam_{\hat T, h}^\lambda  + \BGam_{\hat T, h}^\lambda ) (\BI - \BGam_{\hat T, h}^\lambda )^{-1} = (\BI - \BGam_{\hat T, h}^\lambda )^{-1}$.\end{proof}

\begin{corollary} \label{cor_TDnum_SMWapprox}
    From Theorem \ref{thm_P_vs_Gamma} it follows immediately that the two models $\hat{\mathcal J}_{\text{TDnum}}$ defined in \eqref{eq_defTDnum_model} and $\hat{\mathcal J}_{\text{SMWapprox}}$ defined in \eqref{eq_SMW_W} coincide.
\end{corollary}

\begin{remark}
    We remark that the same proof can be conducted in the case of inhomogeneous material around the element of interest and the statements of Theorem \ref{thm_P_vs_Gamma} and Corollary \ref{cor_TDnum_SMWapprox} remain valid also in that case.
\end{remark}

At the first glance, Theorem \ref{thm_P_vs_Gamma} and Corollary \ref{cor_TDnum_SMWapprox} seem very surprising since they state that the model that is based on the continuous concept of topological derivatives brought to a discrete setting coincides with a model that is based on a certain approximation of a term generated by the purely algebraic Sherman-Morrison-Woodbury matrix identity. This resemblance, however, has been identified in \cite{Nazarov2009} on the purely continuous setting for the case of elliptic inclusions.
\begin{lemma}[\cite{Nazarov2009}] \label{lem_Nazarov}
    Assume that $\omega$ is an ellipse. Then
    \begin{align}\label{eq_PNazarov}
        \mathcal P_\omega[\lamOut, \lamIn] = -(\lamIn - \lamOut) \left(\BI_2 + (\lamIn - \lamOut) \BPsi[\lamOut] \right)^{-1} \BPsi[\lamOut]
    \end{align}
    where $\BPsi[\lamOut] \in \mathbb R^{2\times 2}$ is given by
    \begin{align} \label{eq_defPsi} 
        \BPsi[\lamOut]_{i,j} = - \left( \int_{\partial \omega} n(x) \nabla_x \Phi[\lamOut](x)^\top \; \mbox ds_x \right)_{i,j} =  -  \int_{\partial \omega} n_i \partial_{x_j} \Phi[\lamOut](x) \; \mbox ds_x
    \end{align}
    with the fundamental solution $\Phi[\lamOut]$ of the operator $u \mapsto - \mbox{div}(\lamOut \nabla u)$, i.e., $\Phi[\lamOut](x) = - 1/ (2 \pi \lamOut) \mbox{ln}(|x|)$. 
\end{lemma}

\begin{proof}
    This follows straightforwardly from \cite[Sec. 8]{Nazarov2009} by restricting the (vector-valued) elasticity problem treated there to the scalar Laplace-type problem considered here.
\end{proof}

From Lemma \ref{lem_Nazarov}, it follows in the same way as in the proof of Theorem \ref{thm_P_vs_Gamma} that
\begin{align}
    \BI_2 + \mathcal P_\omega[\lamOut, \lamIn] = (\BI_2 + (\lamIn - \lamOut) \BPsi[\lamOut] )^{-1},
\end{align}
and thus, from \eqref{eq_TD_IpP}, we get the alternative representation of the topological derivative for elliptic inclusion shapes $\omega$
\begin{align}
    d \mathcal J[\lamOut, \lamIn](\Omega)(z, \omega) = (\lamIn - \lamOut) \nabla u(z)^\top (\BI_2 + (\lamIn - \lamOut) \BPsi[\lamOut] )^{-1}  \nabla p(z).
\end{align}
Thus, the matrix $\BGam_{\hat T, h}[\lamOut]$ used in $\hat{\mathcal J}_{\text{SMWapprox}}$ can also be seen as an approximation to the negative fundamental matrix $-\BPsi[\lamOut]$.
Finally, note that the proof of Lemma \ref{lem_Nazarov} in \cite{Nazarov2009} is not valid for triangular inclusion shapes, however, at the discrete level, relation \eqref{eq_P_Gamma_identity1} still holds. Furthermore, note that the matrix $\BPsi[\lamOut]$ is symmetric such that the right hand side of \eqref{eq_PNazarov} can also be written as $-(\lamIn - \lamOut) \BPsi[\lamOut]\left(\BI_2 + (\lamIn - \lamOut) \BPsi[\lamOut] \right)^{-1} $ which coincides with the structure of \eqref{eq_P_Gamma_identity1}.

\section{Numerical experiments} \label{sec_numExp}

In this section, we examine the models introduced in Sections \ref{sec_modelTD} and \ref{sec_SMW_W} and compare them to the exact solution as well as the diagonal approximation model introduced in Section \ref{sec_SMWdiag}. Since we noted in Section 6 that model \eqref{eq_defTDnum_model} of Section \ref{sec_modelTD} and model \eqref{eq_SMW_W} of Section \ref{sec_SMW_W} coincide, we will here only consider the latter model. We remark that coincidence of the two models was observed also in all numerical examples.

All numerical results are illustrated for the model problem introduced in Section \ref{sec_modelProb} with the two-dimensional computational domain $\Dsf = (0,1)^2$ with Dirichlet and Neumann boundaries $\Gamma_D = \{(0,y), y \in (0,1)\}\cup \{(x,0), x \in (0,1)\}$, $\Gamma_N = \partial \Dsf \setminus \Gamma_N$ with corresponding data $g_D = 0$, $g_N(x_1, x_2) = x_1 x_2$ and the constant source term $f(x_1, x_2) = 1$. The material coefficient $\lambda(x)$ will vary between the values $\lambdaMin=1$ and $\lambdaMax=1000$. 

We begin by considering a homogeneous setting.
\subsection{Homogeneous material distribution}
Here, we consider a constant material distribution, i.e., $\lambda_\Omega(x) = \lamOut$ for all $x \in \Dsf$, which corresponds to setting $\Omega = \emptyset$ for some value $\lamOut$ in the setting of Section \ref{sec_modelProb_cont}. See also Figure \ref{fig_design_solution_homo} for plots of the material distribution and the finite element solution on a mesh with $n=1089$ nodes and $m=2048$ elements.

\begin{figure}
    \begin{tabular}{cc}
        \includegraphics[width=.5\textwidth, trim=100 0 50 0, clip]{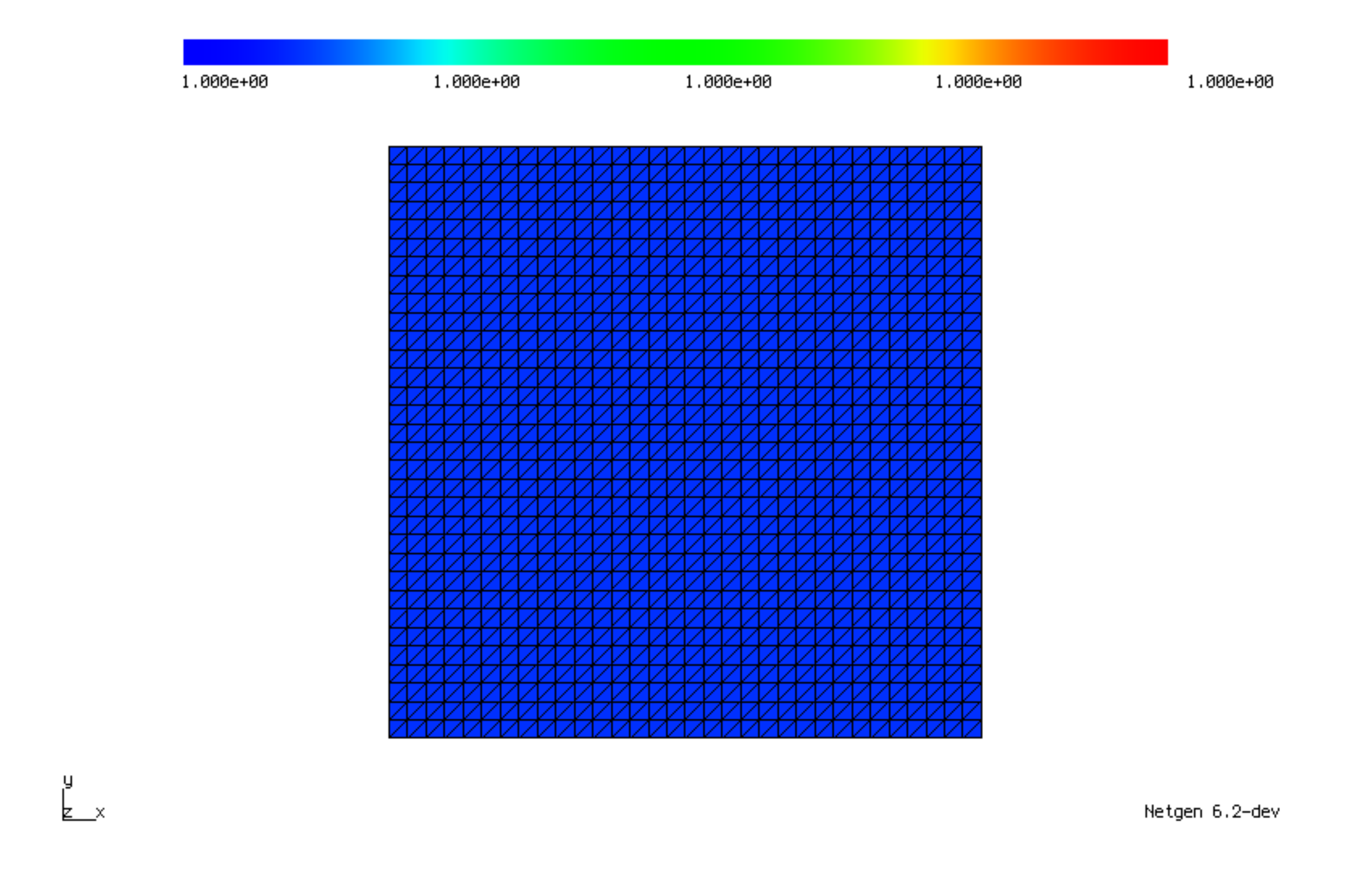} &
        \includegraphics[width=.5\textwidth, trim=100 0 50 0, clip]{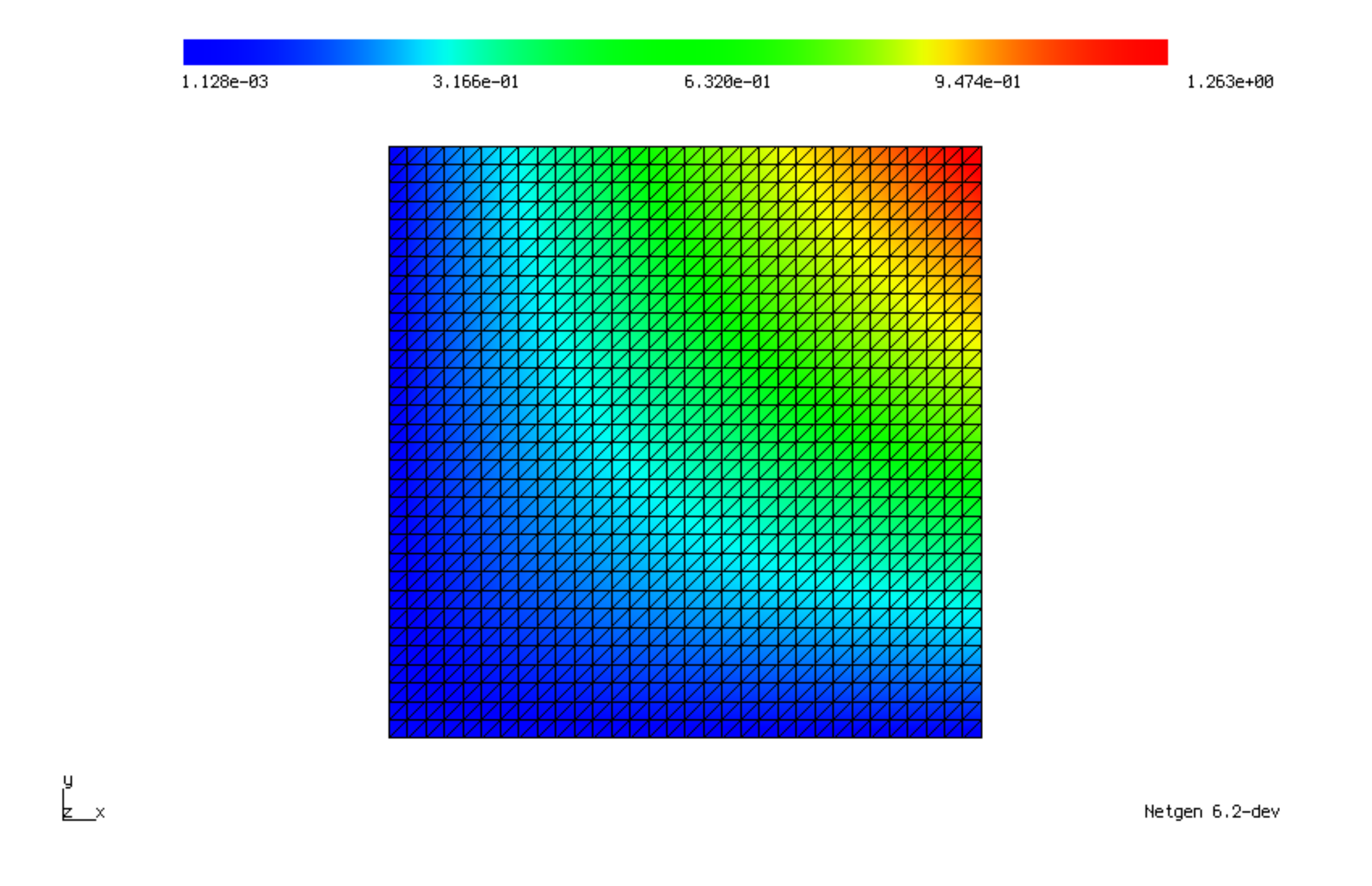} \\
        (a) & (b)
    \end{tabular}
    \caption{(a) Material coefficient $\lambda(x)$ for homogeneous setting. (b) Finite element solution $u_h$ of problem \eqref{eq_probpde} with data specified in Section \ref{sec_numExp} for homogeneous material distribution.}
    \label{fig_design_solution_homo}
\end{figure}

\subsubsection{Numerical comparison of different models for fixed element}
We fix the interior element $T_\ell \in \mathcal T$ as that element of type 1 (cf. Figure \ref{fig_mesh_elTypes}) that has the point $(0.5, 0.25)$ as its bottom right vertex and compare the different models we introduced in the previous sections for the case where the homogeneous material distribution is perturbed only in that one element $T_\ell$, i.e., $\Beta = \Blam + (\eta - \Blam_\ell)\Be^{(\ell)}$. Figure \ref{fig_comparison_oneTrig} shows the different models as functions of the perturbed value $\eta \in [\lambdaMin, \lambdaMax] = [1, 1000]$ for three different background material values $\lamOut = 1$, $\lamOut \approx 145.834$, $\lamOut = 1000$ (cf. Remark \ref{rem_equiError}). 
In Fig. \ref{fig_comparison_oneTrig} we can see the exact solution $\mathcal J(\Beta)$ (where the perturbed stiffness matrix is inverted) for certain values of $\eta$ along with the exact Sherman-Morrison-Woodbury model $\hat{\mathcal J}_{\text{SMW}}$ \eqref{eq_SMWexact} which shows, as expected, perfect coincidence. Moreover, we can see the diagonal approximation of the Sherman-Morrison-Woodbury model $\hat{\mathcal J}_{\text{SMWdiag}}$ \eqref{eq_SMWdiag_model} which shows a certain error, but can be evaluated more efficiently. The models $\hat{\mathcal J}_{\text{TDnum}}$ \eqref{eq_defTDnum_model} and $\hat{\mathcal J}_{\text{SMWapprox}}$ \eqref{eq_SMW_W} can be seen to match exactly, as was predicted by Theorem \ref{thm_P_vs_Gamma} and Corollary \ref{cor_TDnum_SMWapprox}. Moreover, it can be seen from Fig. \ref{fig_comparison_oneTrig} that these two models approximate the exact solution extraordinarily well while being cheap to evaluate during the online phase (after precomputations have been done in an offline phase, cf. Remark \ref{rem_precomputation}).
For comparison, we also included the linearization $\mathcal J(\Beta) \approx \mathcal J(\Blam) - |T_\ell| (\eta - \Blam_\ell) |\nabla u_h|_{T_\ell}|^2$ and the topological derivative model when the analytical formula for the topological derivative of circular inclusions \eqref{eq_TDformula_analytic} is used. It can be seen that the linearization model is far away from the true solution. But also the latter model shows a significant error which confirms the necessity to account for the triangular inclusion shape as it was done in Section \ref{sec_modelTD}.

In order to quantify these errors, let us define the relative error measure of a model $\hat{\mathcal J}$ in an element $T_\ell$ for a given material distribution $\Blam$ by
\begin{align} \label{eq_max_eta}
   \delta \hat{\mathcal J}[T_\ell] := \frac{1}{\Delta \mathcal J[T_\ell]} \; \underset{\eta \in [\lambdaMin, \lambdaMax]}{\mbox{max }} \left \lvert \hat{\mathcal J}\left(\Blam + (\eta - \Blam_\ell)\Be^{(\ell)}\right)-\mathcal J\left(\Blam + (\eta - \Blam_\ell)\Be^{(\ell)}\right) \right \rvert
 \end{align}
 where $\Delta \mathcal J[T_\ell] = \mbox{max}_{t\in (\lambdaMin, \lambdaMax)} \mathcal J(\Blam + (t - \Blam_\ell)\Be^{(\ell)}) - \mbox{min }_{t\in (\lambdaMin, \lambdaMax)} \mathcal J(\Blam + (t - \Blam_\ell)\Be^{(\ell)})$ is the difference of maximal and minimal values of the exact model in $T_\ell$. Thus, $\delta \hat{\mathcal J}[T_\ell]$ measures the maximum relative error of a model $\hat{\mathcal J}$ in element $T_\ell$ relative to the variation of the exact cost function $\mathcal J$. The relative errors according to \eqref{eq_max_eta} for the three expansion points $\lambda(x) = \lamOut$ investigated in Figure \ref{fig_comparison_oneTrig} are as follows: For the linearization model the errors are as high as (20 954$\%$, 85.61$\%$, 25.21$\%$), for the model $\hat{\mathcal J}_{\text{SMWdiag}}$ they are (16.65$\%$, 3.41$\%$, 0.99$\%$), for $\hat{\mathcal J}_{\text{TDcirc}}$ we have (57.93$\%$, 27.23$\%$, 49.43$\%$) and for the coinciding models $\hat{\mathcal J}_{\text{TDnum}}$ and $\hat{\mathcal J}_{\text{SMWapprox}}$ we have the values (1.18$\%$, 0.74$\%$, 1.46$\%$).

\begin{figure}
    \begin{tabular}{ccc}
       \includegraphics[width=.33\textwidth, trim=0 0 0 0, clip]{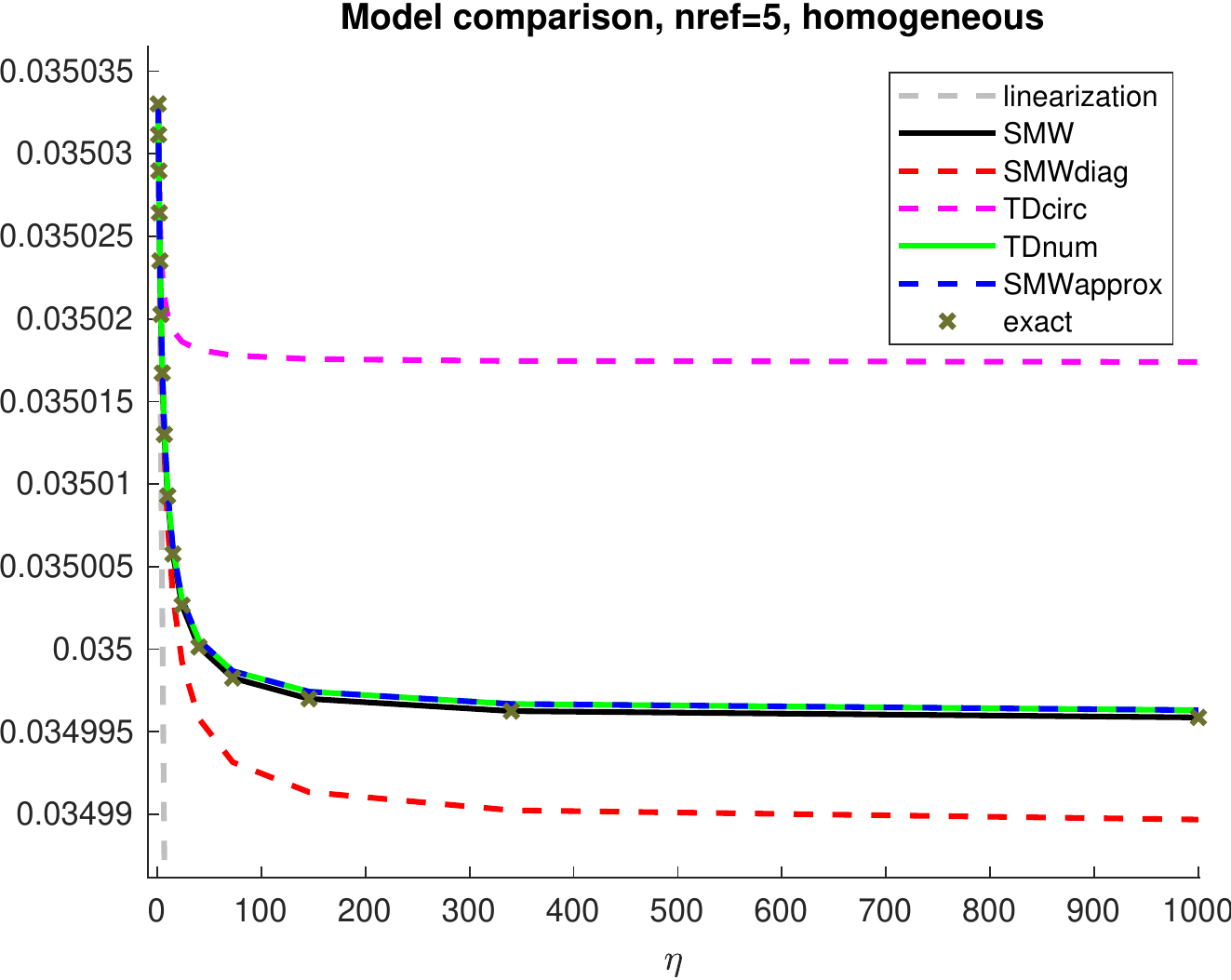} & \includegraphics[width=.33\textwidth, trim=0 0 0 0, clip]{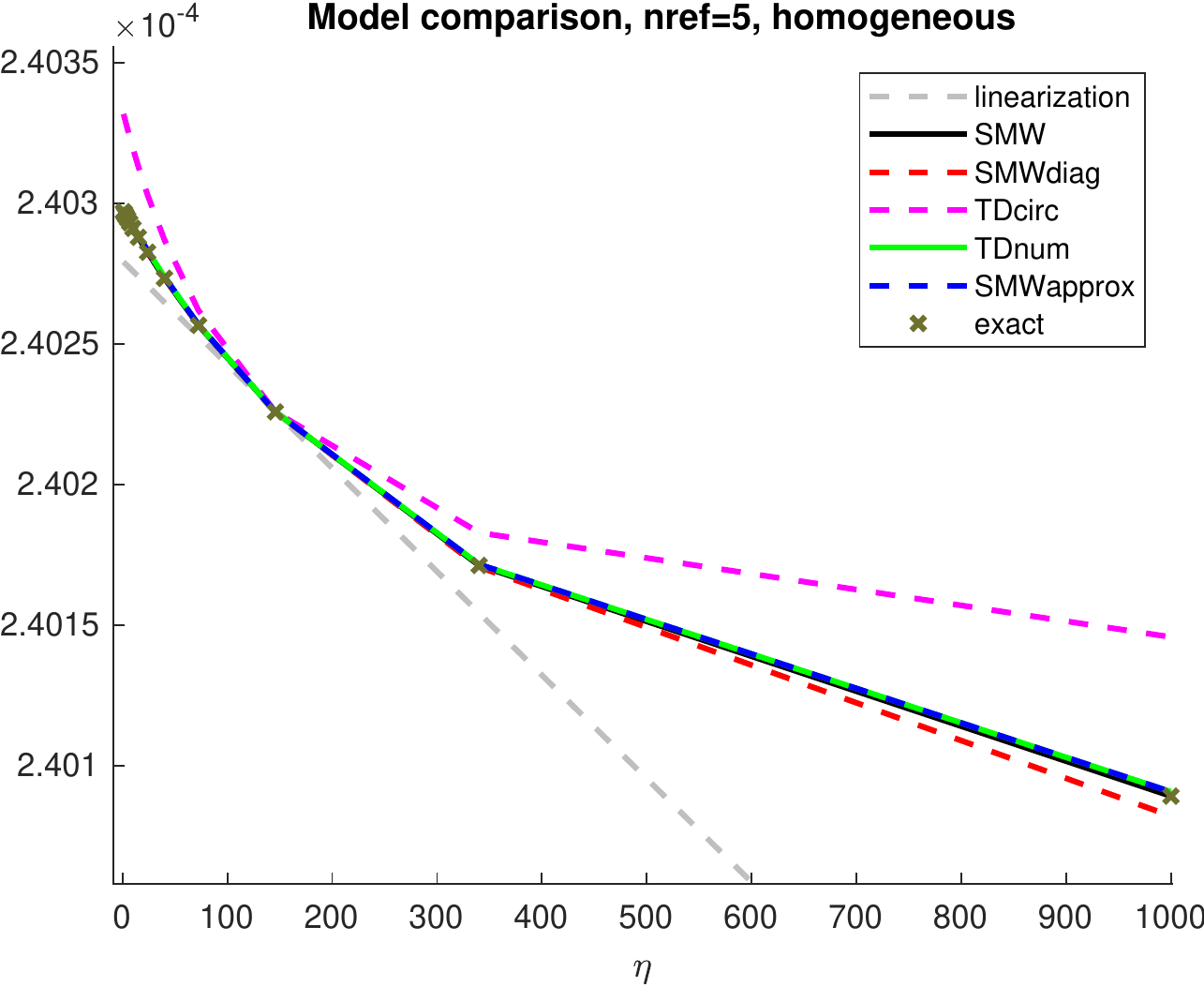}& \includegraphics[width=.33\textwidth, trim=0 0 0 0, clip]{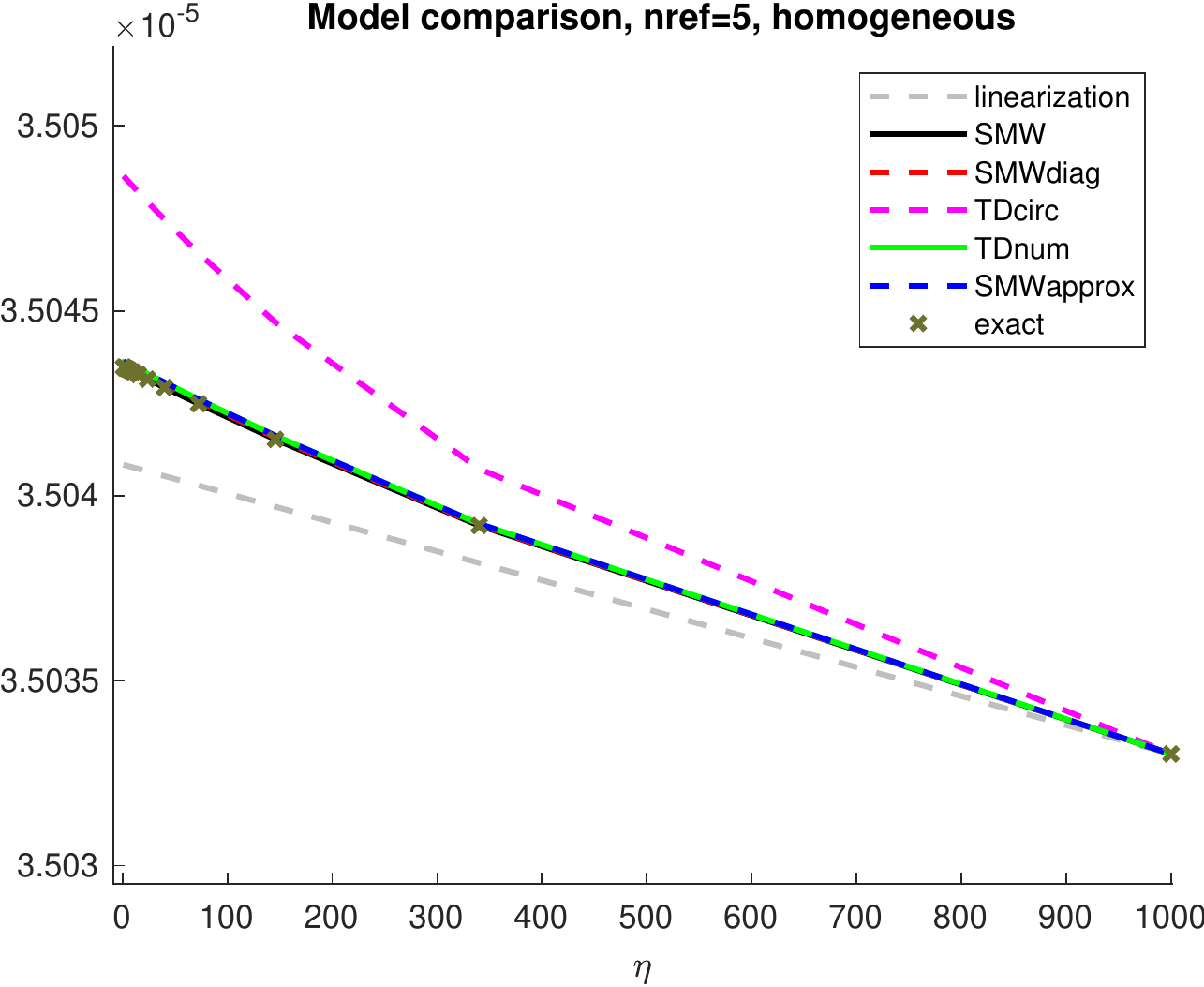} \\
       (a) & (b) & (c)
    \end{tabular}
    \caption{Comparison of different models for homogeneous setting with background material (a) $\lamOut = 1$, (b) $\lamOut \approx 145.834$, (c) $\lamOut = 1000$ as functions of the perturbed material coefficient in a fixed element $T_\ell$.}
    \label{fig_comparison_oneTrig}
\end{figure}

\begin{remark} \label{rem_equiError}
    In Figure \ref{fig_comparison_oneTrig}, the models as well as the exact solution were evaluated at certain perturbation values $\eta^{(k)} \in [1, 1000]$, $k=1, \dots, 16$. These values have been chosen in the following way: It was observed numerically that the exact solution in Fig. \ref{fig_comparison_oneTrig} behaves similarly to $a+b \eta^{-0.5}$ for some constants $a, b$. Based on this observation, the points $\eta^{(k)}$ were chosen in such a way that, using these points as interpolation nodes, a piecewise linear interpolation of $a+b \eta^{-0.5}$ yields an equilibrated error, see Fig. \ref{fig_equi_error}. This was achieved by solving a system of nonlinear equations ensuring that the maximum interpolation error between any two neighboring nodes is equal. The values for $N=16$ points are given in Table \ref{tab_valuesEtak16}. While these points here are used solely for visualization purposes, their role will become more important in Section \ref{sec_num_inhomo} to decide for which values of the material coefficient the (computationally expensive) precomputation should be carried out.
\end{remark}
\begin{remark} \label{rem_hoelderPara}
    Also the choice of the H\"older parameter $\alpha=-0.5$ in \eqref{eq_lambdaSj} was motivated by the observation that the exact solution in Fig. \ref{fig_comparison_oneTrig}(a) behaves roughly like $a+b \eta^{-0.5} =: f(\eta)$. Thus, an effective value for $\lambda$ can be obtained by the relation $f(\lambda) = \sum w_i f(\lambda_i)$ for given values $\lambda_i$ with corresponding weights $w_i$ such that $\sum w_i = 1$, which, by inverting $f$, results in the chosen average value \eqref{eq_lambdaSj}.
\end{remark}
\begin{table}
    \begin{center}
    \begin{tabular}{|c|c|c|c|c|c|c|c|}
     1 & 1.252 & 1.590 & 2.050 & 2.688 & 3.596 & 4.921 & 6.917   \\ \hline
     10.035 & 15.127 & 23.901 & 40.072 & 72.563 & 145.834 & 340.187 & 1000
    \end{tabular} \vspace{-6mm}
    \end{center}
    \caption{Values $\eta^{(k)}$, $k=1,\dots, 16$, used for visualization and precomputation}
    \label{tab_valuesEtak16}
\end{table}

\begin{figure}
    \begin{tabular}{cc}
    \includegraphics[width=.5\textwidth, trim=0 0 0 0, clip]{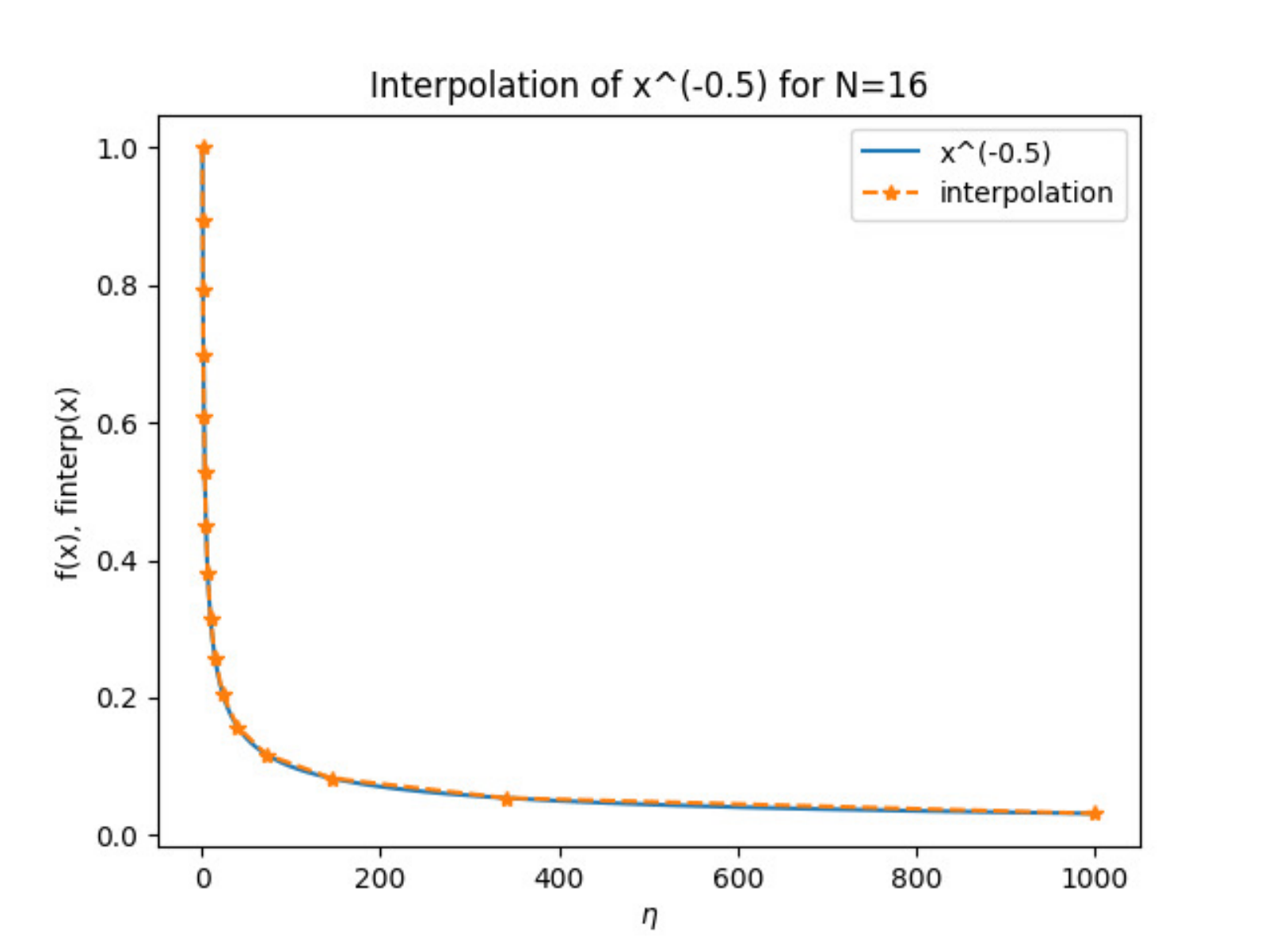} &
    \includegraphics[width=.5\textwidth, trim=0 0 0 0, clip]{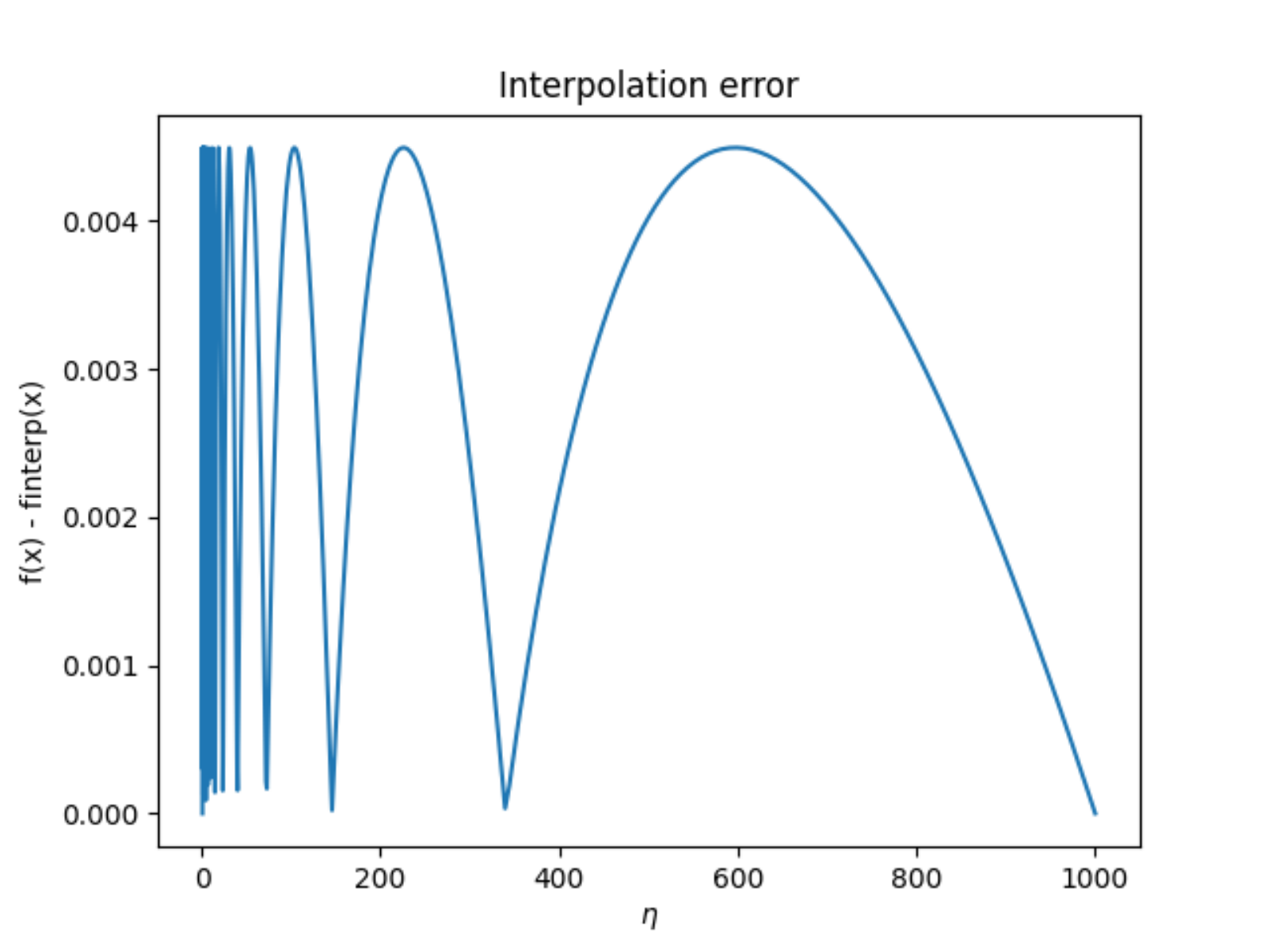}
    \end{tabular}
    \caption{The interpolation nodes $\eta^{(k)}$, $k=1, \dots, 16$ are chosen in such a way that the maximum interpolation error between the function $x^{-0.5}$ and its piecewise linear interpolant attains the same value in each interval $(\eta^{(k)}, \eta^{(k+1)})$.}
    \label{fig_equi_error}
\end{figure}

\subsubsection{Numerical comparison of different models on whole domain} \label{sec_comp_domain_homo}
In Fig. \ref{fig_comparison_oneTrig}, we compared several different models for a fixed triangle in the mesh when the material perturbation parameter is varied between $\lambdaMin=1$ and $\lambdaMax=1000$. Next, we investigate the maximum relative error $\delta \hat{\mathcal J}$ as defined in \eqref{eq_max_eta} of a model $\hat{\mathcal J}$ as a function of the position in space.

 \begin{figure}
    \begin{tabular}{cc}
        $\delta \hat{\mathcal J}_{\text{SMWdiag}}$&
        $\delta \hat{\mathcal J}_{\text{SMWapprox}}$\\
        \includegraphics[width=.45\textwidth, trim=100 0 50 0, clip]{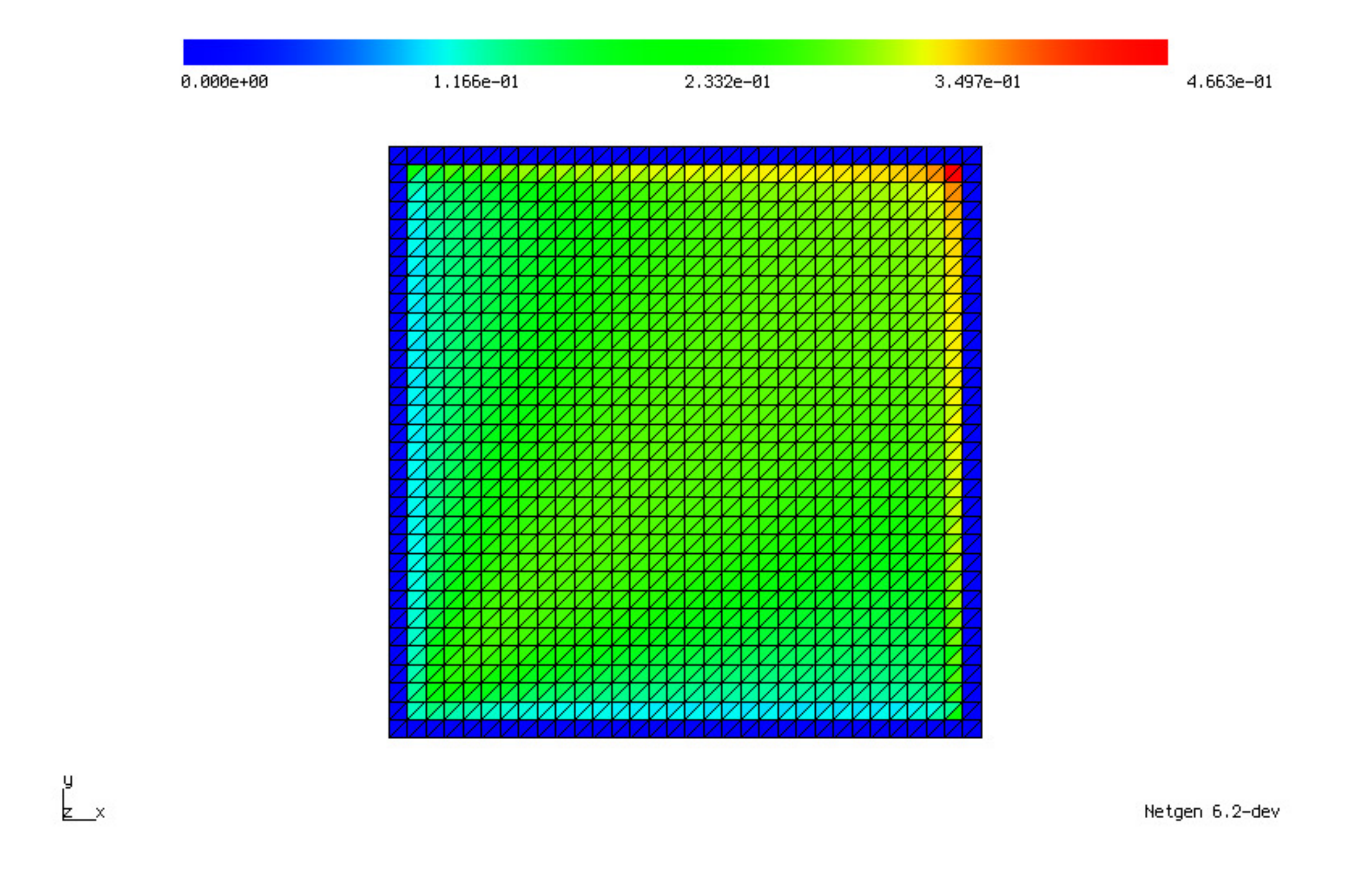} &
        \includegraphics[width=.45\textwidth, trim=100 0 50 0, clip]{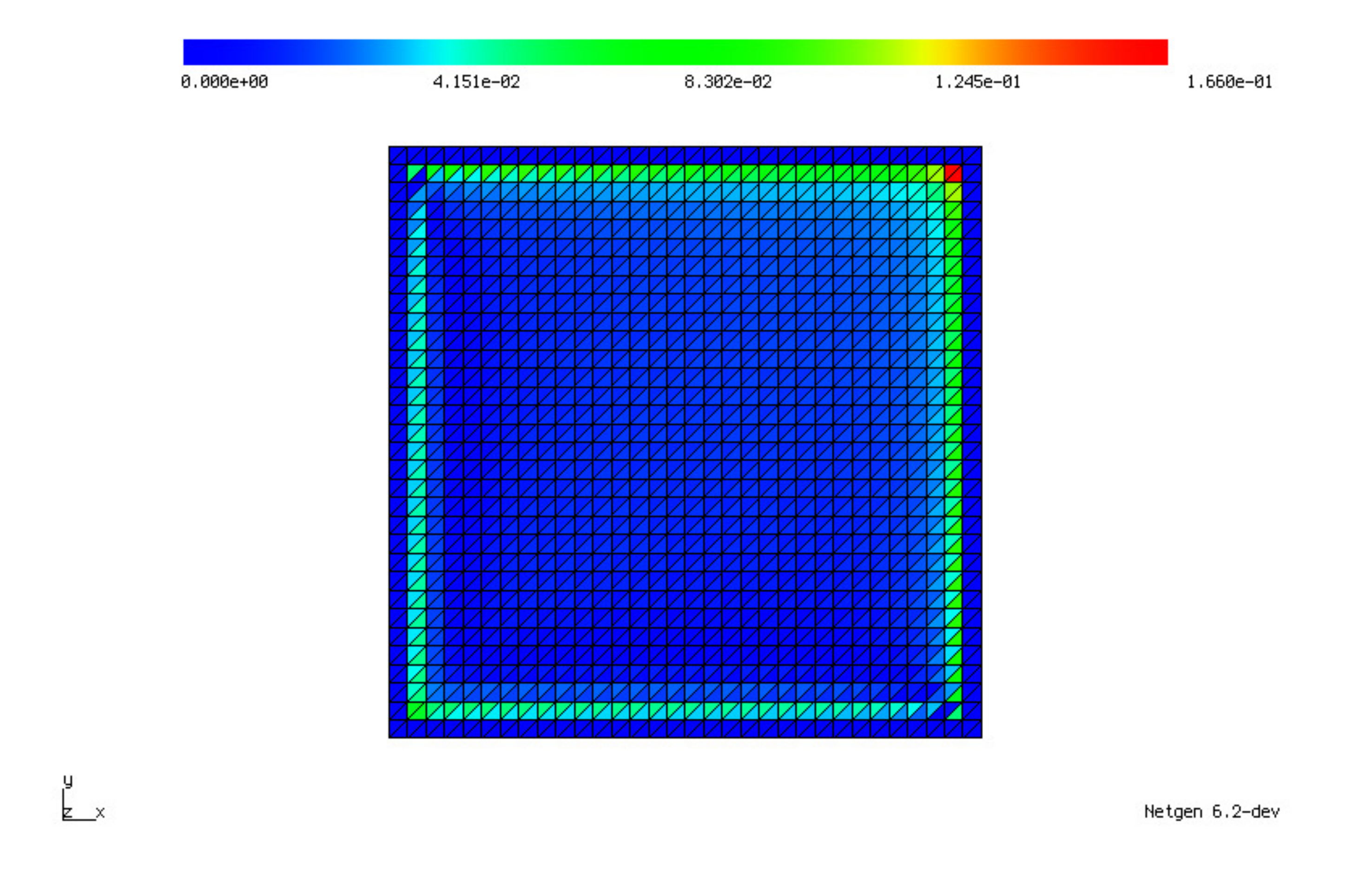} \\
        \includegraphics[width=.45\textwidth, trim=100 0 50 0, clip]{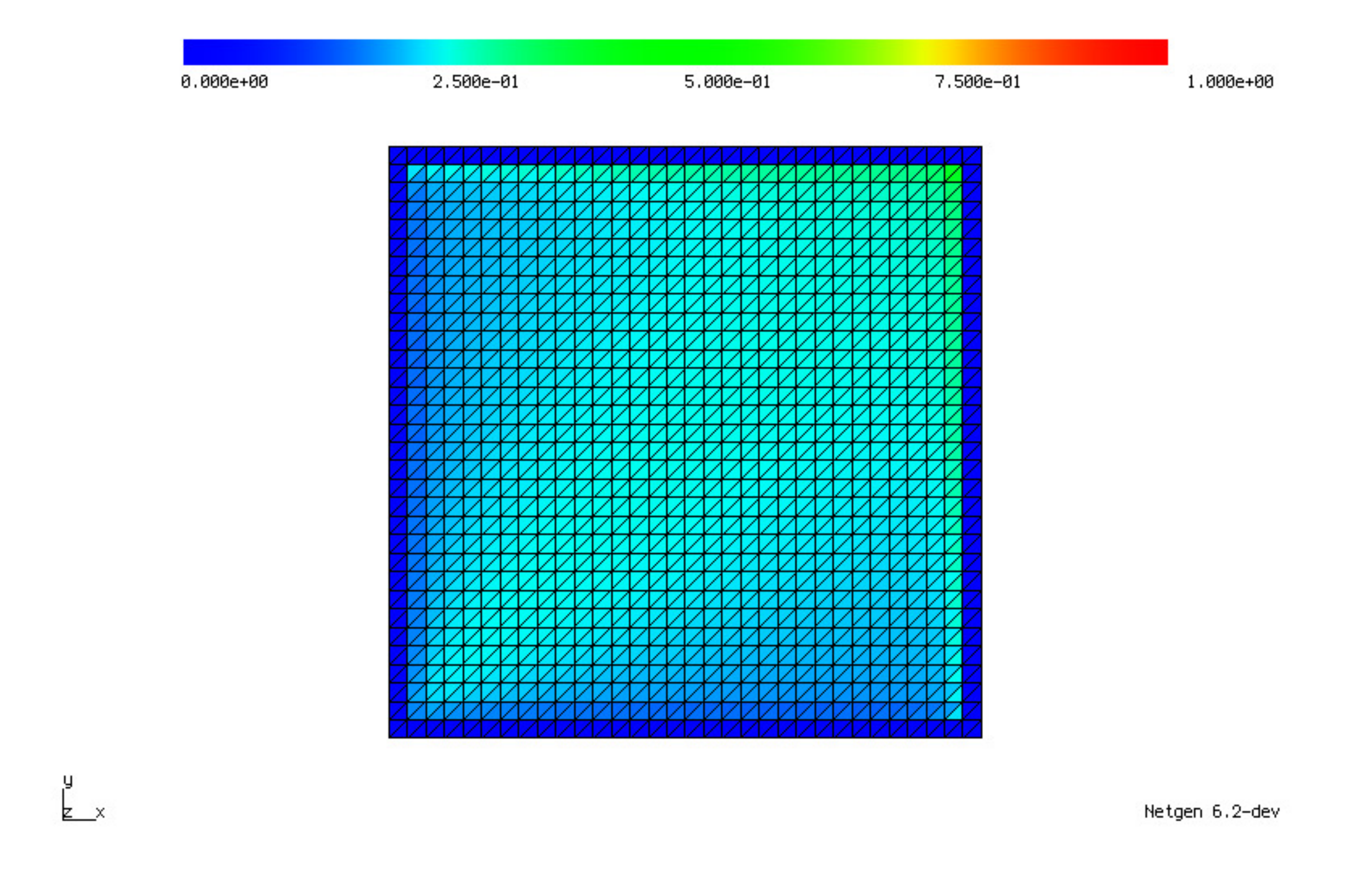} &
        \includegraphics[width=.45\textwidth, trim=100 0 50 0, clip]{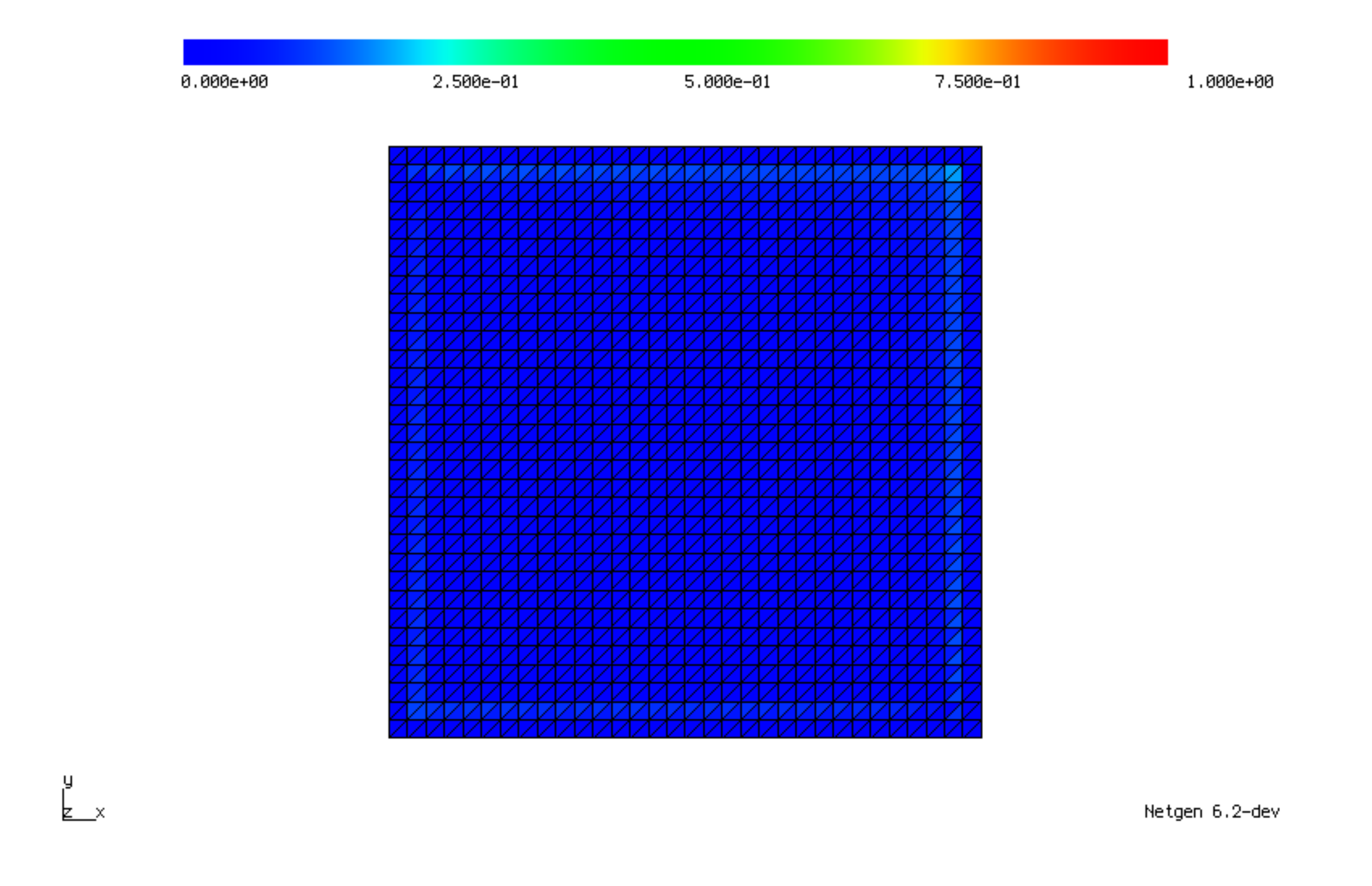} \\
        \includegraphics[width=.45\textwidth, trim=100 0 50 0, clip]{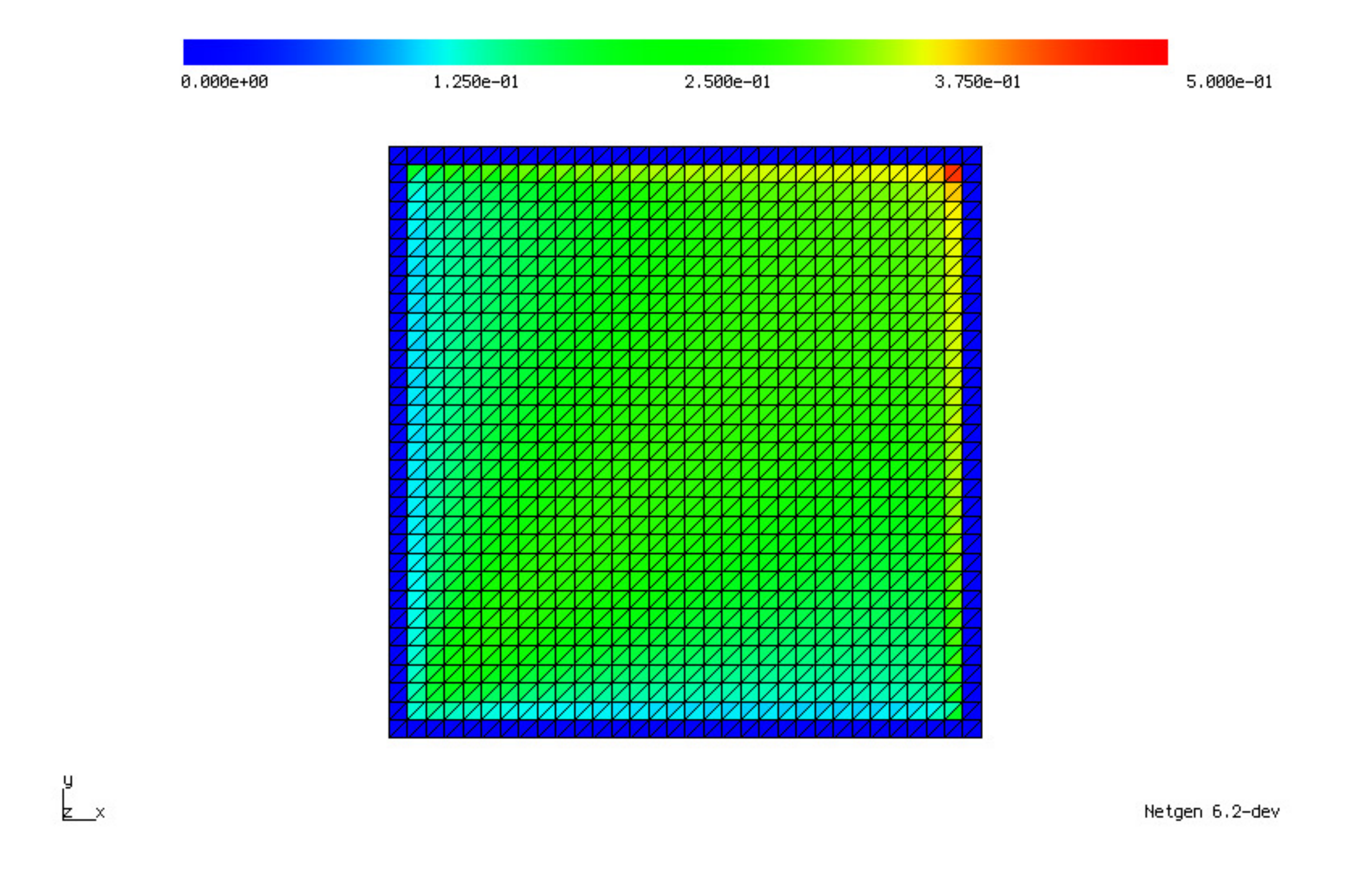} &
        \includegraphics[width=.45\textwidth, trim=100 0 50 0, clip]{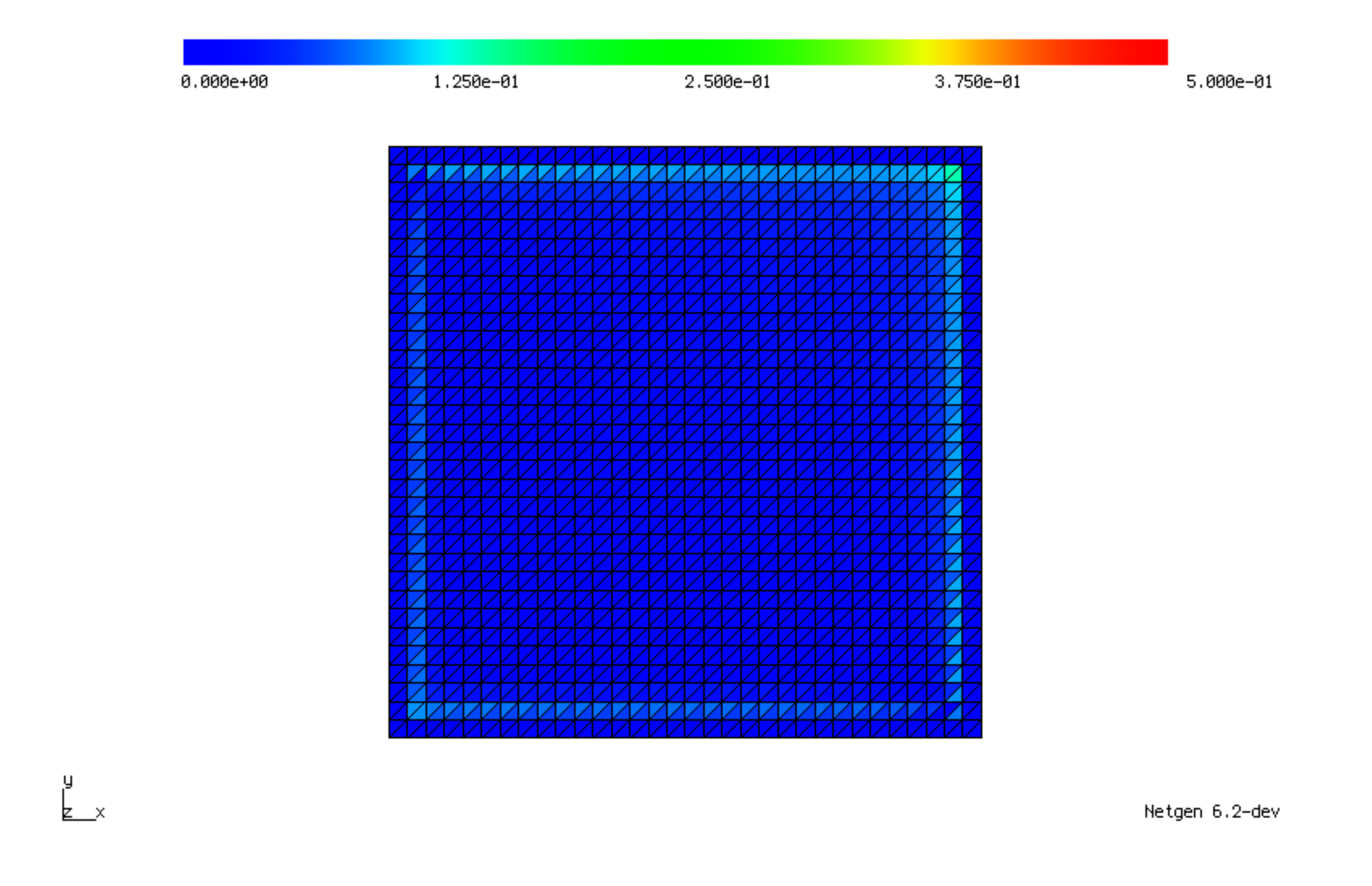} \\
        \includegraphics[width=.45\textwidth, trim=100 0 50 0, clip]{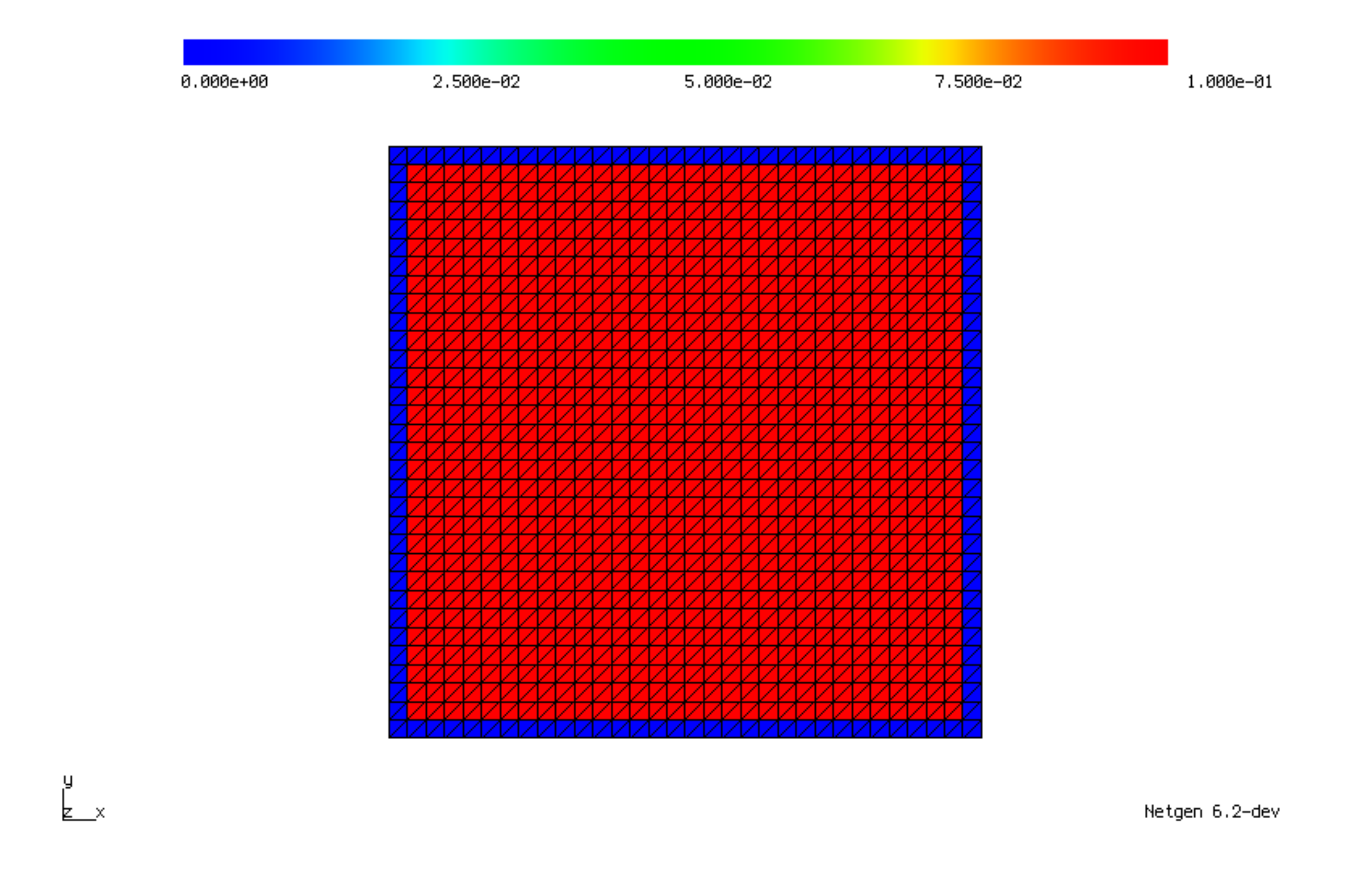} &
        \includegraphics[width=.45\textwidth, trim=100 0 50 0, clip]{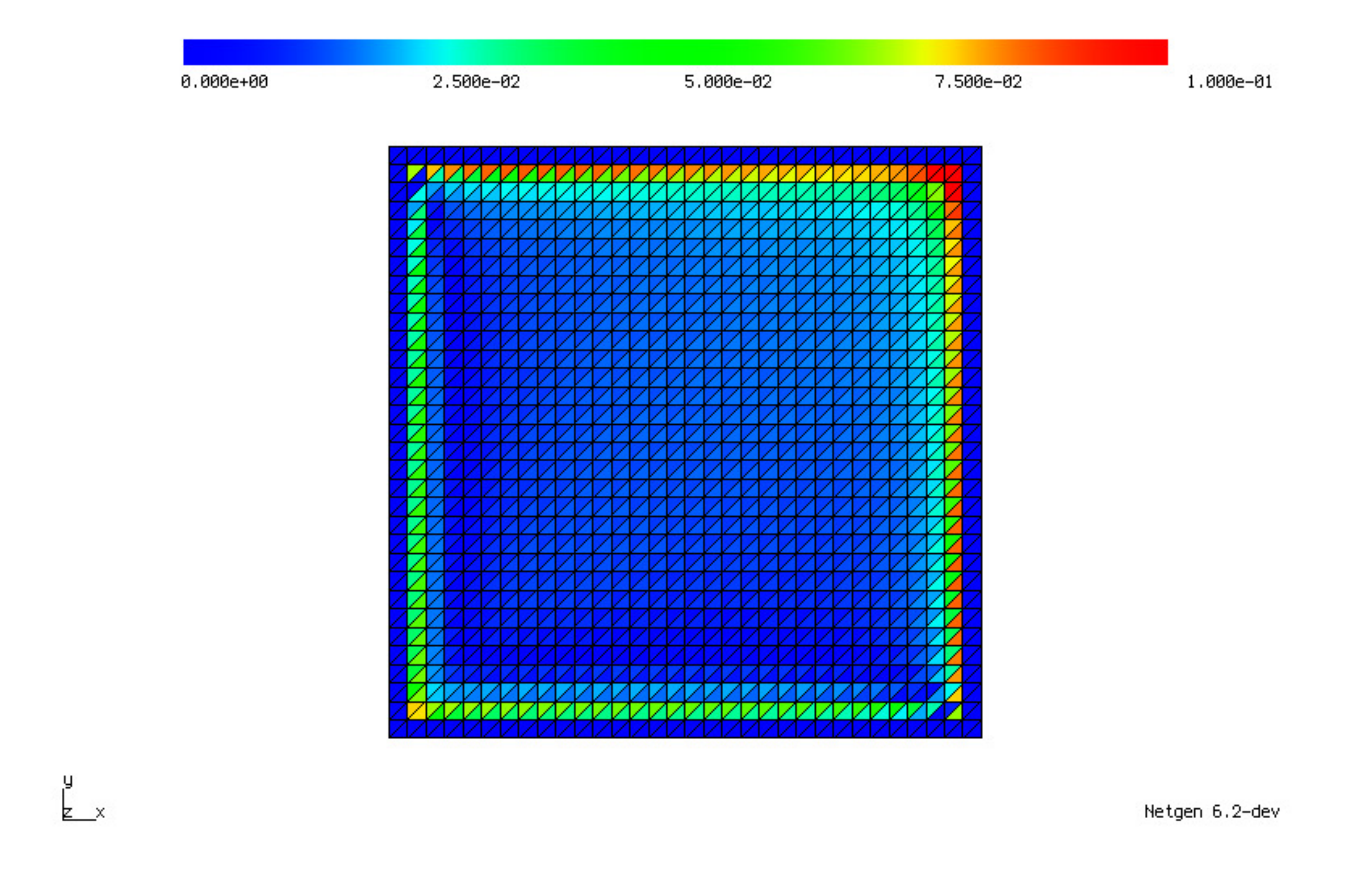}
    \end{tabular}
    \caption{Comparison of relative errors $\delta \hat{\mathcal J}[T_\ell]$ according to \eqref{eq_max_eta} for models $\hat{\mathcal J}_{\text{SMWdiag}}$ \eqref{eq_SMWdiag_model} in left column and $\hat{\mathcal J}_{\text{SMWapprox}}$ \eqref{eq_SMW_W} ($=\hat{\mathcal J}_{\text{TDnum}}$ \eqref{eq_defTDnum_model}) in right column for all interior elements $T_\ell$ in homogeneous setting. First line shows color plot according to their respective maximum errors. Second to fourth line show threshold for maximum relative error at $100 \%$, $50\%$ and $10\%$, respectively. Errors in elements touching the boundary are not computed.}

    \label{fig_errorDomain_homo}
\end{figure}

 Given the findings of Fig. \ref{fig_comparison_oneTrig}, we will focus on the two approximate Sherman-Morrison-Woodbury models $\hat{\mathcal J}_{\text{SMWdiag}}$ and $\hat{\mathcal J}_{\text{SMWapprox}}$ (recall that $\hat{\mathcal J}_{\text{TDnum}}$ coincides with $\hat{\mathcal J}_{\text{SMWapprox}}$).
 In Figure \ref{fig_errorDomain_homo}, we plot the maximum relative errors \eqref{eq_max_eta} for these two models for all interior elements of the computational domain for the homogeneous material distribution $\lambda(x) = \lamOut = 1$, i.e., we are in the setting of Figure \ref{fig_comparison_oneTrig}(a). Elements touching the boundary are discussed separately in Section \ref{sec_elemsBdy}. Here, we can see that the maximum relative error of the model $\hat{\mathcal J}_{\text{SMWdiag}}$ is around $47\%$ whereas it is around $17\%$ for $\hat{\mathcal J}_{\text{SMWapprox}}$. The four different rows of Figure \ref{fig_errorDomain_homo} show different threshold values for the relative errors in the color bars. Moreover, it can be seen from the right column in Fig. \ref{fig_errorDomain_homo} that the model $\hat{\mathcal J}_{\text{SMWapprox}}$ behaves particularly well in the center of the homogeneous domain and that the error increases slightly the closer one gets to a boundary. This was to be expected since the idea of the model $\hat{\mathcal J}_{\text{SMWapprox}}$ in Section \ref{sec_SMW_W} (and equivalently model $\hat{\mathcal J}_{\text{TDnum}}$ of Section \ref{sec_topder}) was to zoom in locally around the fixed element $T_\ell$ and assume that all boundaries are sufficiently far away, cf. Fig. \ref{fig_smwapprox_pert_rescaled} and Fig. \ref{fig_smwapprox_pert_rescaled_h0__trunc}. Nevertheless, the maximum error attained by $\hat{\mathcal J}_{\text{SMWapprox}}$ is still significantly smaller than that of the model $\hat{\mathcal J}_{\text{SMWdiag}}$.

\subsubsection{Towards topology optimization using approximate models}

We want to illustrate the potential of the introduced models $\hat{\mathcal J}_{\text{SMWdiag}}$ and $\hat{\mathcal J}_{\text{SMWapprox}}$ (which coincides with $\hat{\mathcal J}_{\text{TDnum}}$) in the course of a binary topology optimization algorithm. Here, we simply decide for each element $T_\ell$ if it should be occupied by $\lamOut = 1$ or $\lamIn = 1000$ based on the values of a model $\hat{\mathcal J}(\Blam + (\eta-\Blam_\ell)\Be^{(\ell)})$ at $\eta=\lamOut$ and $\eta = \lamIn$, i.e., we do not allow for intermediate material values. Here, we again start out from the homogeneous design where $\Blam \in \mathbb R^m$ is the constant one vector.
As a reference, we consider the separable exact model $\hat{\mathcal J}_{\text{SMW}}$ \eqref{eq_SMWexact}. Due to this model's properties, the material distribution obtained by the mentioned procedure is a local minimum, which cannot be improved by switching the state of only one element. Note that this is a stronger notion of optimality than a design being solely a stationary point of the relaxed optimization problem. From a theoretical point of view it is not entirely clear that using approximations of exact separable models such as $\hat{\mathcal J}_{\text{SMWdiag}}$ the same effect can be achieved. However, already in \cite{NeesEtAl2022} it was reported that the SGP concept combined with an approximation of the $\hat{\mathcal J}_{\text{SMWdiag}}$ type lead to a much better local minimizer for a binary topology optimization problem than the MMA method utilizing convex separable approximations. Here, we investigate this effect in more detail using a selection of the previously suggested models. In addition, we also make a comparison with an MMA model 
\begin{align} \label{eq_model_MMA}
    \hat{\mathcal J}_{\text{MMA}}(\Beta) &= \mathcal J(\Blam) - \sum_{\ell=1}^m |T_\ell| (\Beta_\ell  - \Blam_\ell) (\nabla u_h|_{T_\ell})^\top \left( \BI_2 -\frac{\Beta_\ell - \Blam_\ell}{L-\Blam_\ell} \BI_2 \right)^{-1}\nabla u_h|_{T_\ell}\nonumber \\ 
    & = \mathcal J(\Blam) - \sum_{\ell=1}^m |T_\ell| | \nabla u_h|_{T_\ell} |^2  \frac{(\Beta_\ell  - \Blam_\ell)(\Blam_\ell  - L)}{ (\Beta_\ell  - L)}.
\end{align}
Here $L$ plays the role of a vertical asymptote, which is chosen individually for each element by a heuristic update scheme in the original MMA method, see \cite{Svanberg1987}. As we consider only a single update step here, the heuristic for the choice of $L$ can not be applied. Instead we test three different constant choices of the asymptote, $L=0$, $L=-5$, $L=-10$. 
Note that \eqref{eq_model_MMA} can be obtained from \eqref{eq_SMWexactGam} by replacing $\BGam^{(\ell)}$ by $\frac{1}{L-\Blam_\ell} \BI_2$. 

Since, as it is well-known, the optimum material design for compliance minimization without limitation on the volume is the full design, we here include a simple volume penalization in the cost function and use the augmented cost function
\begin{align} \label{eq_augmentedCost}
    \mathcal L(\Blam):= \mathcal J(\Blam) + \omega \text{Vol}(\Blam)
\end{align}
with a fixed weight $\omega = 7.5$ and the volume of the strong material $\text{Vol}(\Blam) = \sum_{\ell =1}^m  |T_\ell| (\Blam_\ell - \lamOut) / (\lamIn - \lamOut)$. Note that $\text{Vol}(\Blam)$ itself is a separable function which can be dealt with without approximation error. Figure \ref{fig_towardsTopOpti} shows the designs obtained after one step of the procedure mentioned above when using (a) the exact (but expensive) model $\hat{\mathcal J}_{\text{SMW}}$, (b) the diagonal approximation $\hat{\mathcal J}_{\text{SMWdiag}}$, (c) the proposed model $\hat{\mathcal J}_{\text{SMWapprox}}$ and (d)--(f) the MMA model with $L=0$, $L=-5$ and $L=-10$. Comparing pictures (b) and (c) to (a), we see that the error in the design produced by the model in (c) is almost zero, whereas it is a bit larger for the diagonal approximation model in (b). The performance of the MMA model here depends heavily on the choice of the parameter $L$. For illustration of the method, we also plotted the curves corresponding to the exact and the two mentioned approximate models in three fixed elements. Figure \ref{fig_towardsTopOpti}(g) shows that, in the leftmost of the three highlighted elements in (a)--(f), the value of all six considered models at $\lamIn=1000$ is higher than at $\lamOut = 1$, thus making a switching of the material from $\lamOut$ to $\lamIn$ unattractive. In the same way, in the rightmost of the three marked elements, all models except for MMA with $L=0$ show smaller values at $\lamIn$ than at $\lamOut$, thus suggesting switching the material to decrease the cost function, see Fig. \ref{fig_towardsTopOpti}(i). In the central one out of these three elements, however, the diagonal approximation model $\hat{\mathcal J}_{\text{SMWdiag}}$ suggests to switch the material since its value is smaller at $\lamIn$ than at $\lamOut$, whereas the exact model as well as the proposed approximation $\hat{\mathcal J}_{\text{SMWapprox}}$ suggest not to switch it, see Fig. \ref{fig_towardsTopOpti}(h). The MMA model shows good behavior for the choice $L=-5$, but large errors for the choices $L=0$ and $L=-10$. Finally, we remark that, when comparing with the exact model $\hat{\mathcal J}_{\text{SMW}}$, for $\hat{\mathcal J}_{\text{SMWdiag}}$ the wrong decision was taken for 280 out of 1800 interior elements whereas this was the case only for 19 elements in the case of $\hat{\mathcal J}_{\text{SMWapprox}}$. For the MMA model with $L=0$, $L=-5$, $L=-10$, the numbers of wrongly switched elements were 930, 102 and 586 elements, respectively. 

\begin{figure}
    \centering
    \begin{tabular}{ccc} 
    \includegraphics[width=.33\textwidth, trim=250 0 250 0, clip]{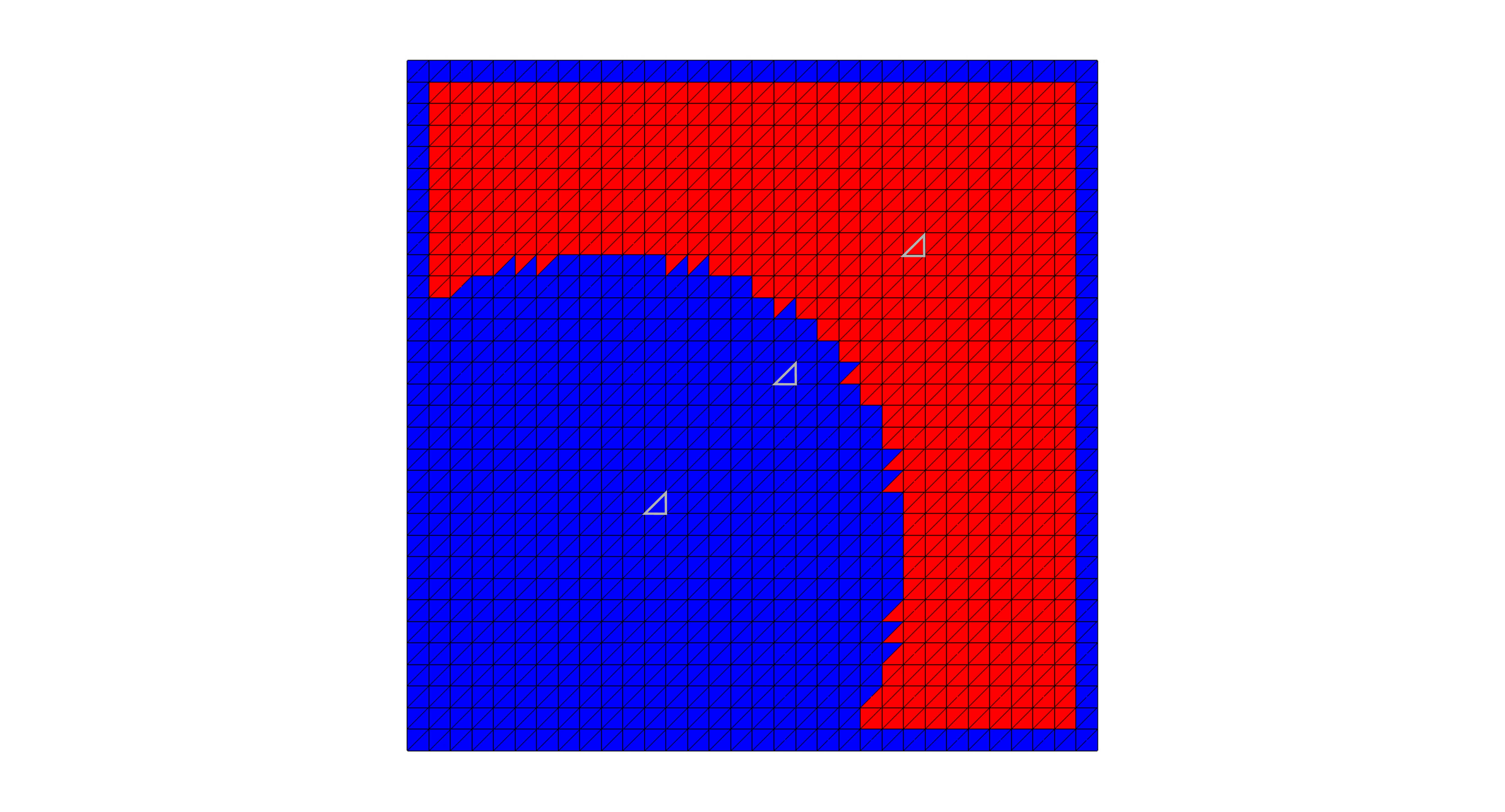}
    &\includegraphics[width=.33\textwidth, trim=250 0 250 0, clip]{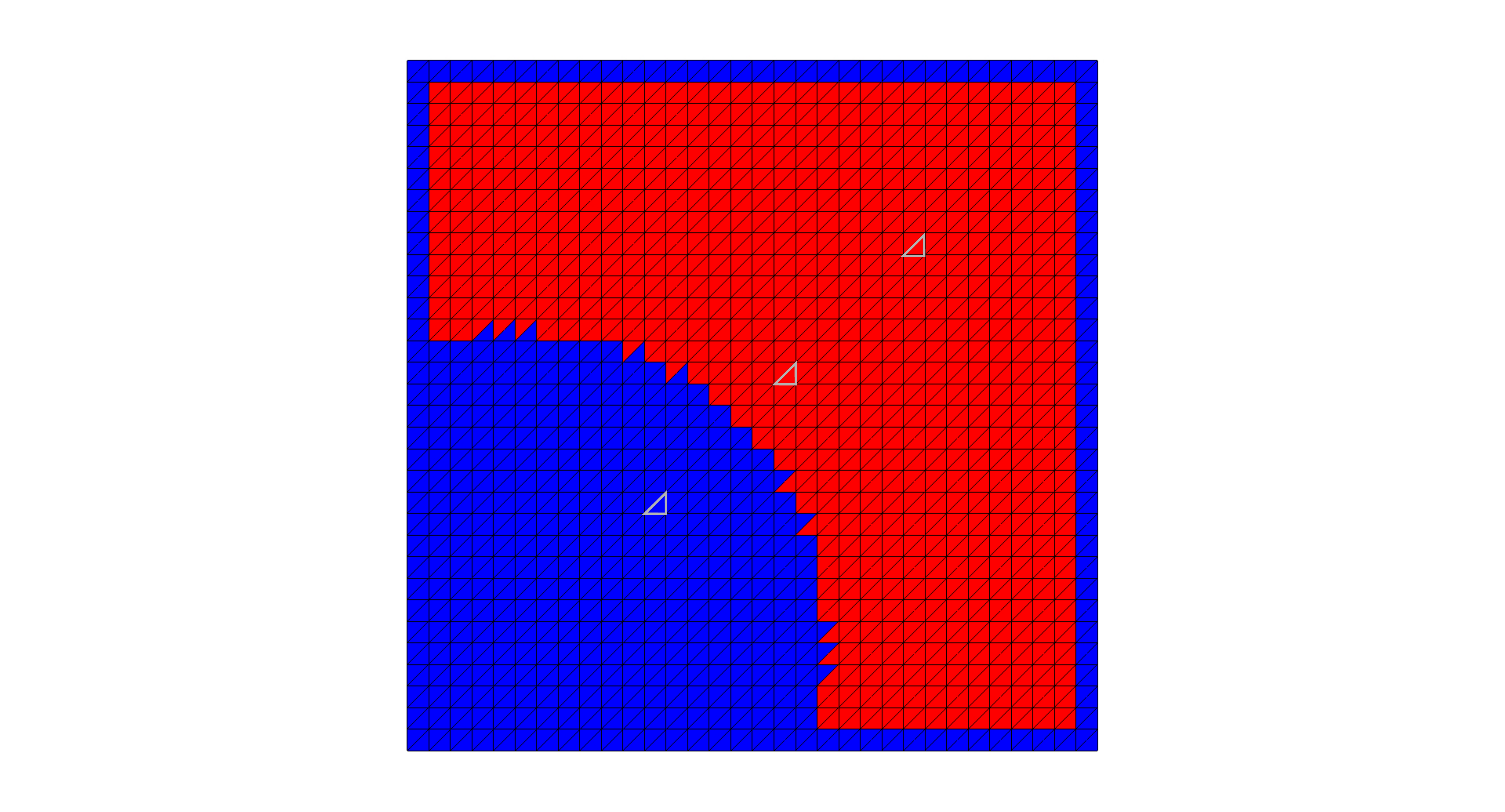}
    &\includegraphics[width=.33\textwidth, trim=250 0 250 0, clip]{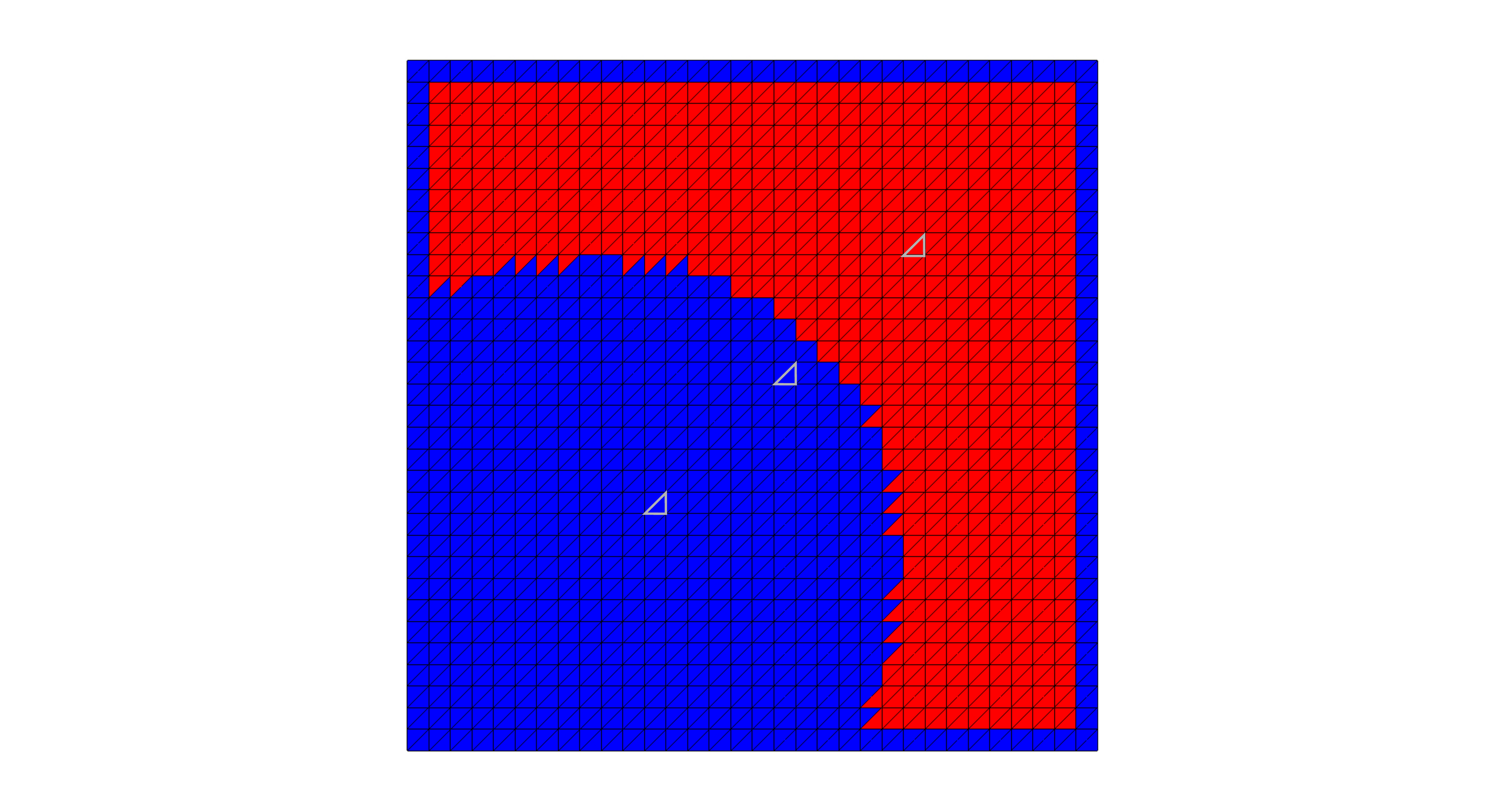} \\
    (a) & (b) & (c) \\
    \includegraphics[width=.33\textwidth, trim=250 0 250 0, clip]{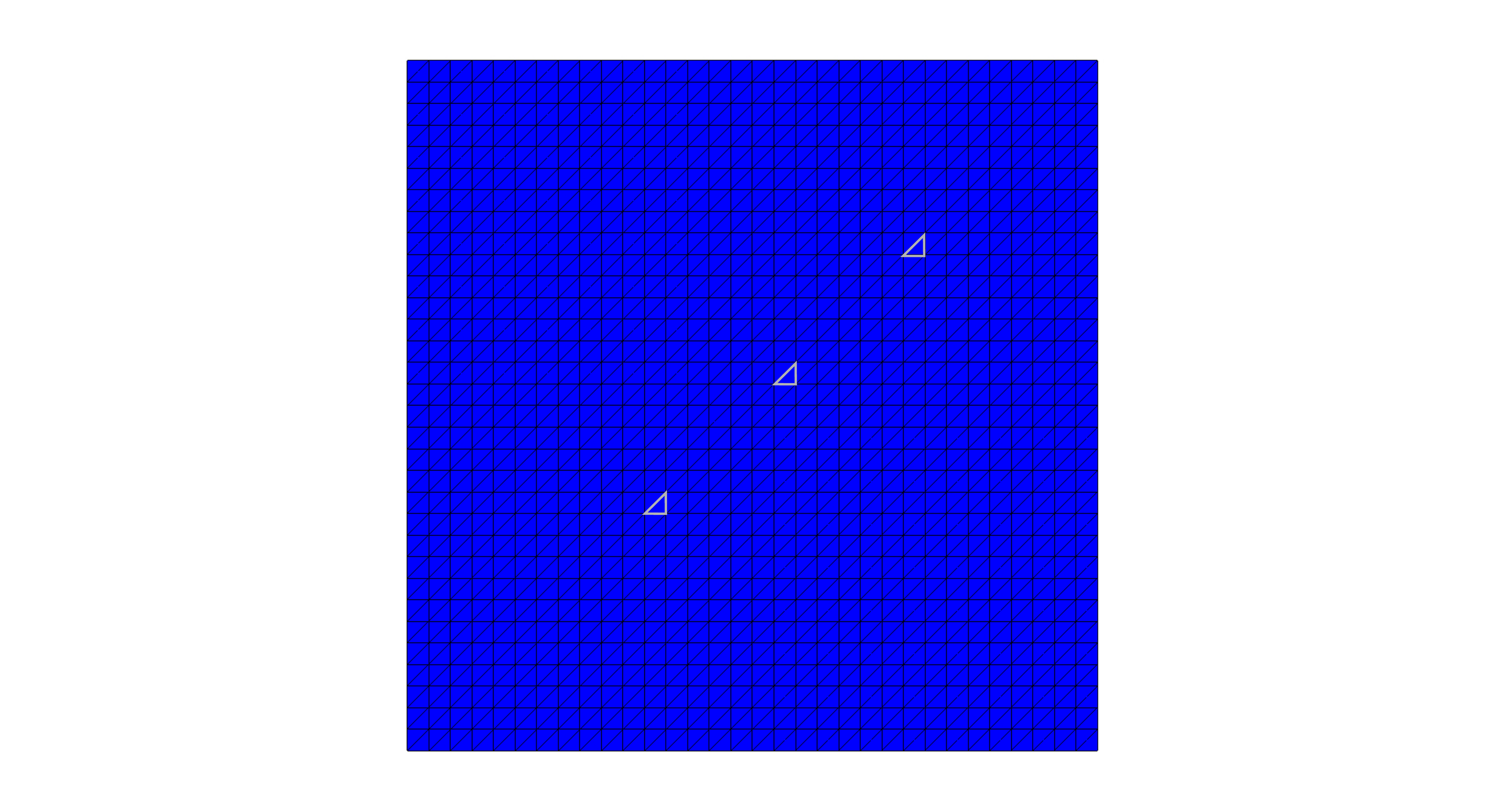}
    &\includegraphics[width=.33\textwidth, trim=250 0 250 0, clip]{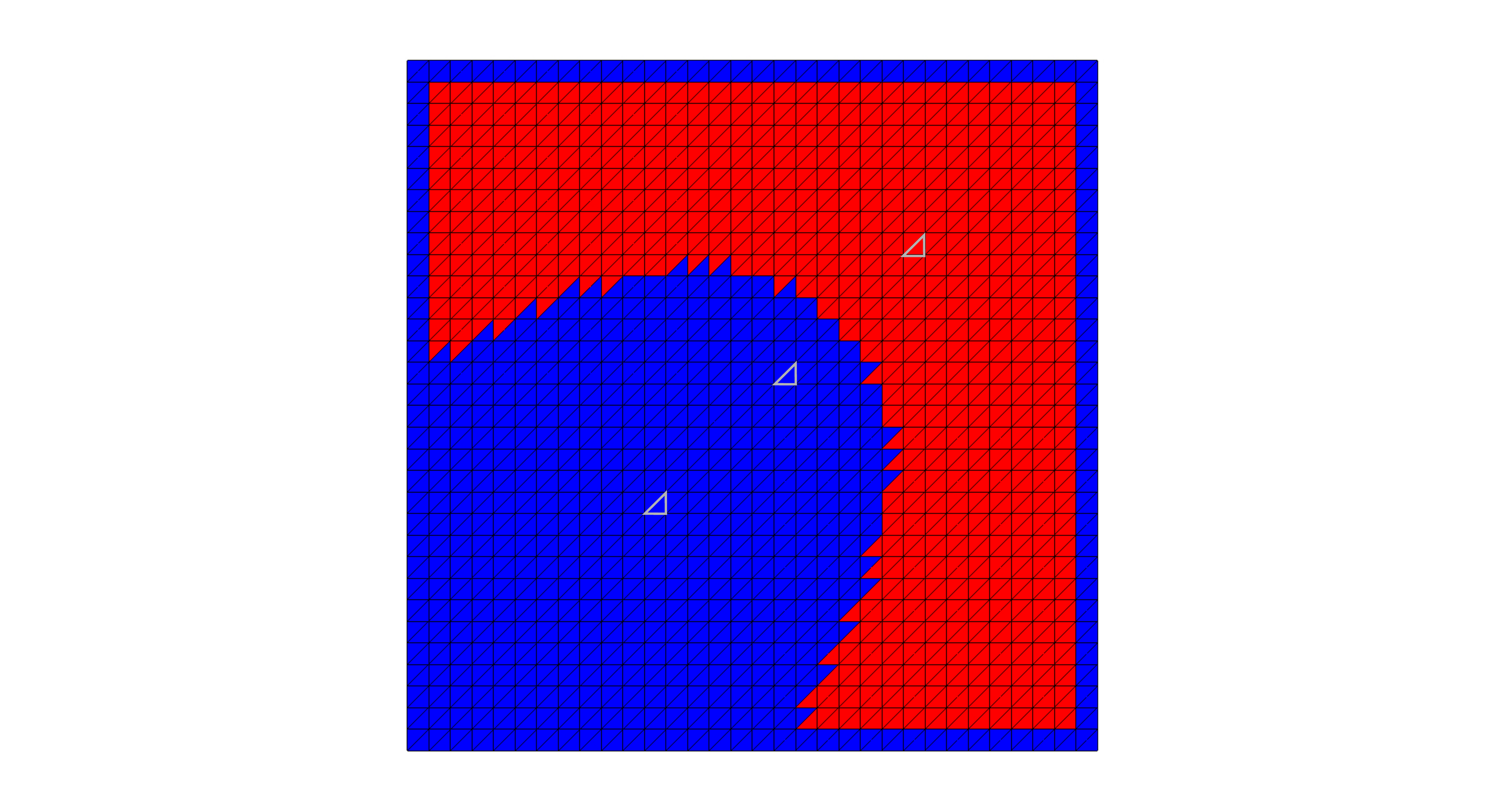}
    &\includegraphics[width=.33\textwidth, trim=250 0 250 0, clip]{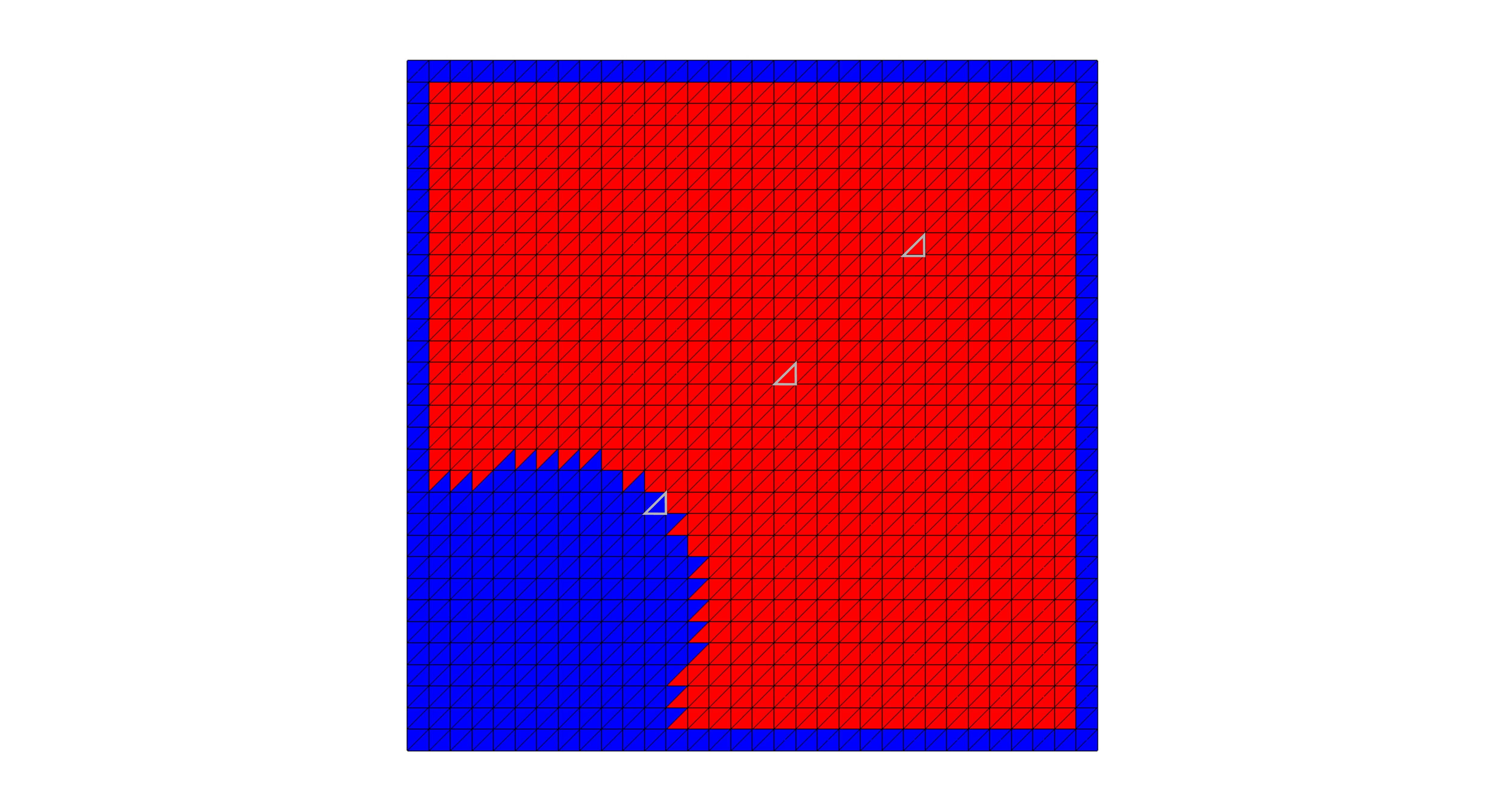} \\
    (d) & (e) & (f) \\
    \includegraphics[width=.3\textwidth]{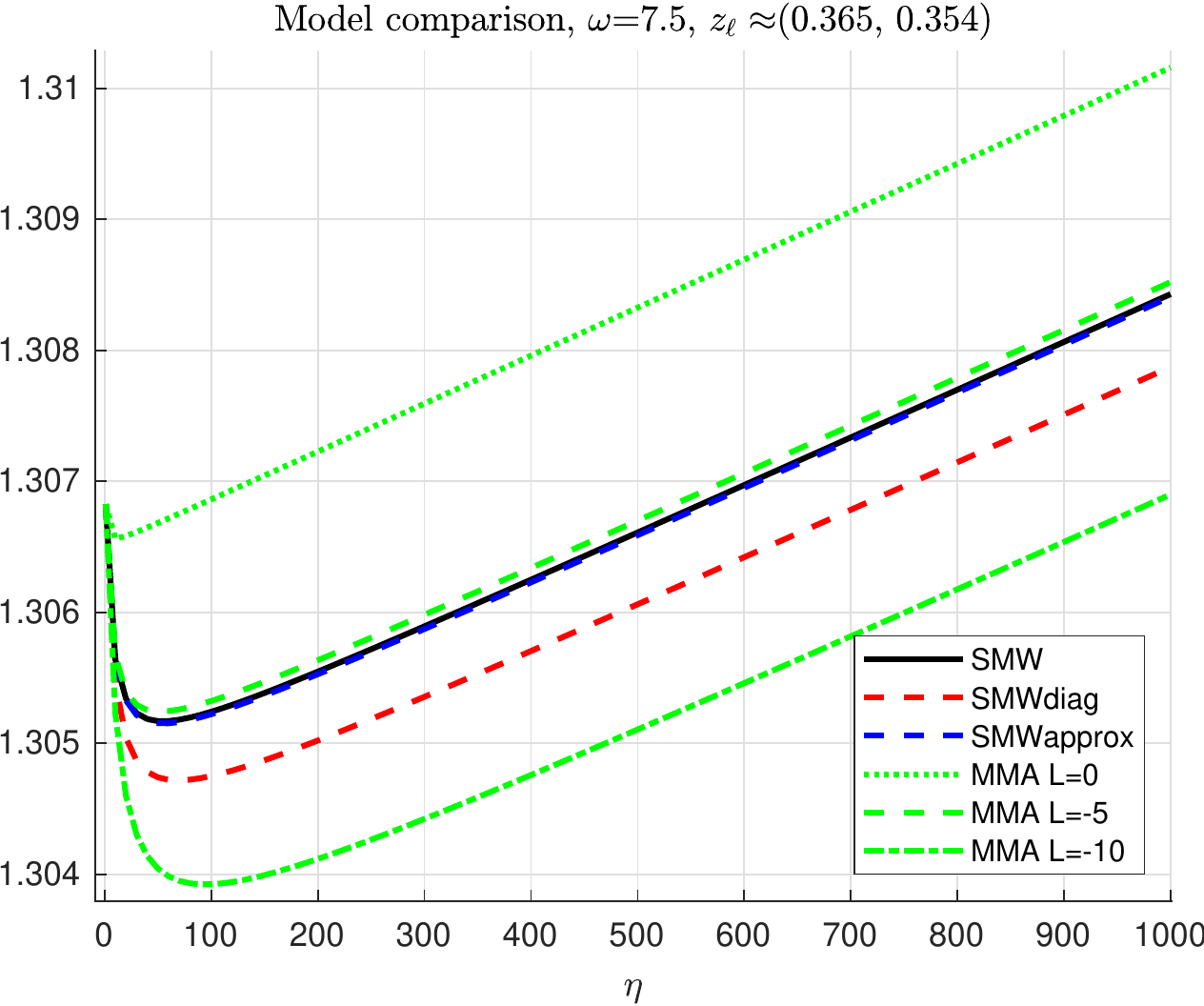} &
    \includegraphics[width=.3\textwidth]{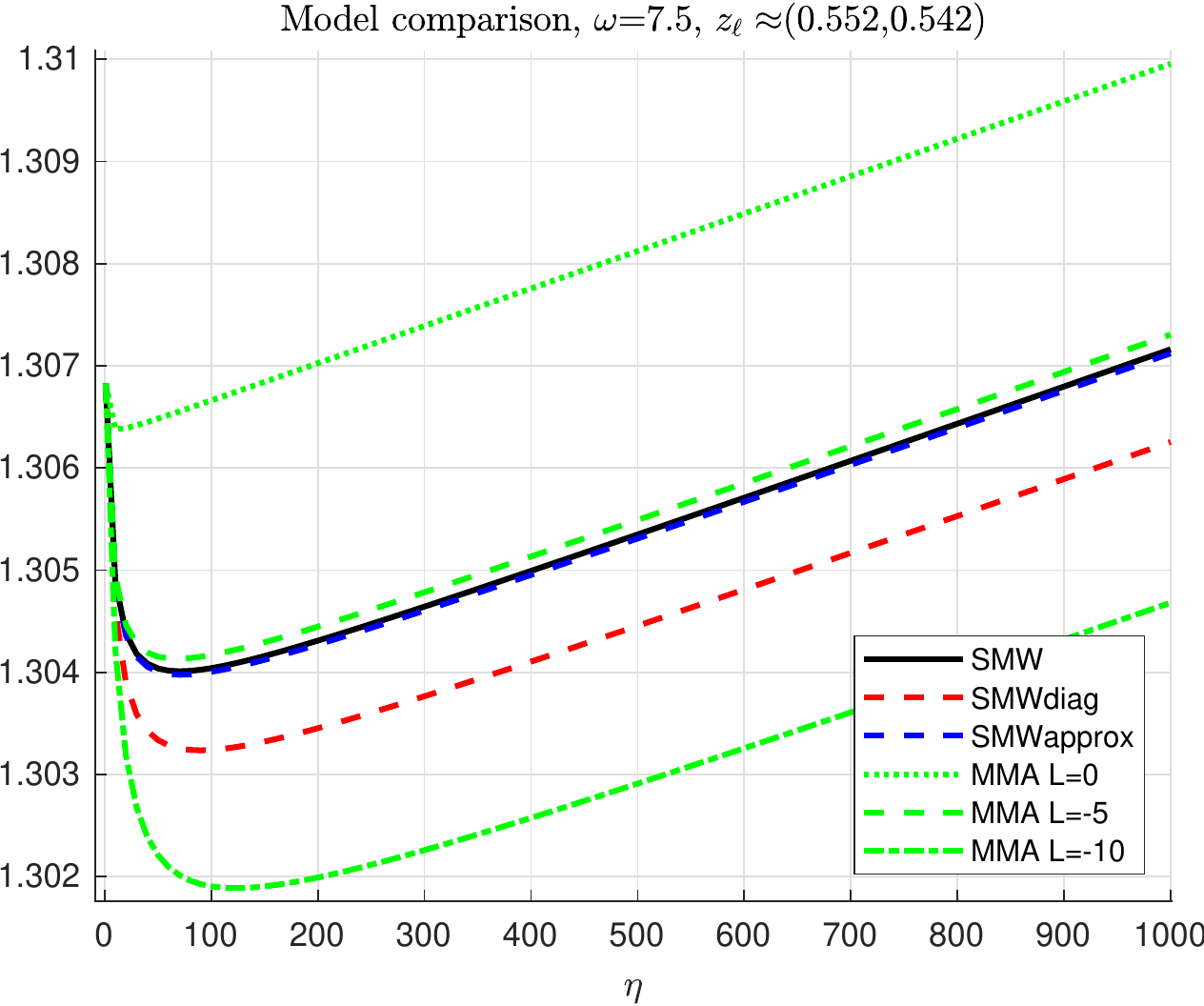} &
    \includegraphics[width=.3\textwidth]{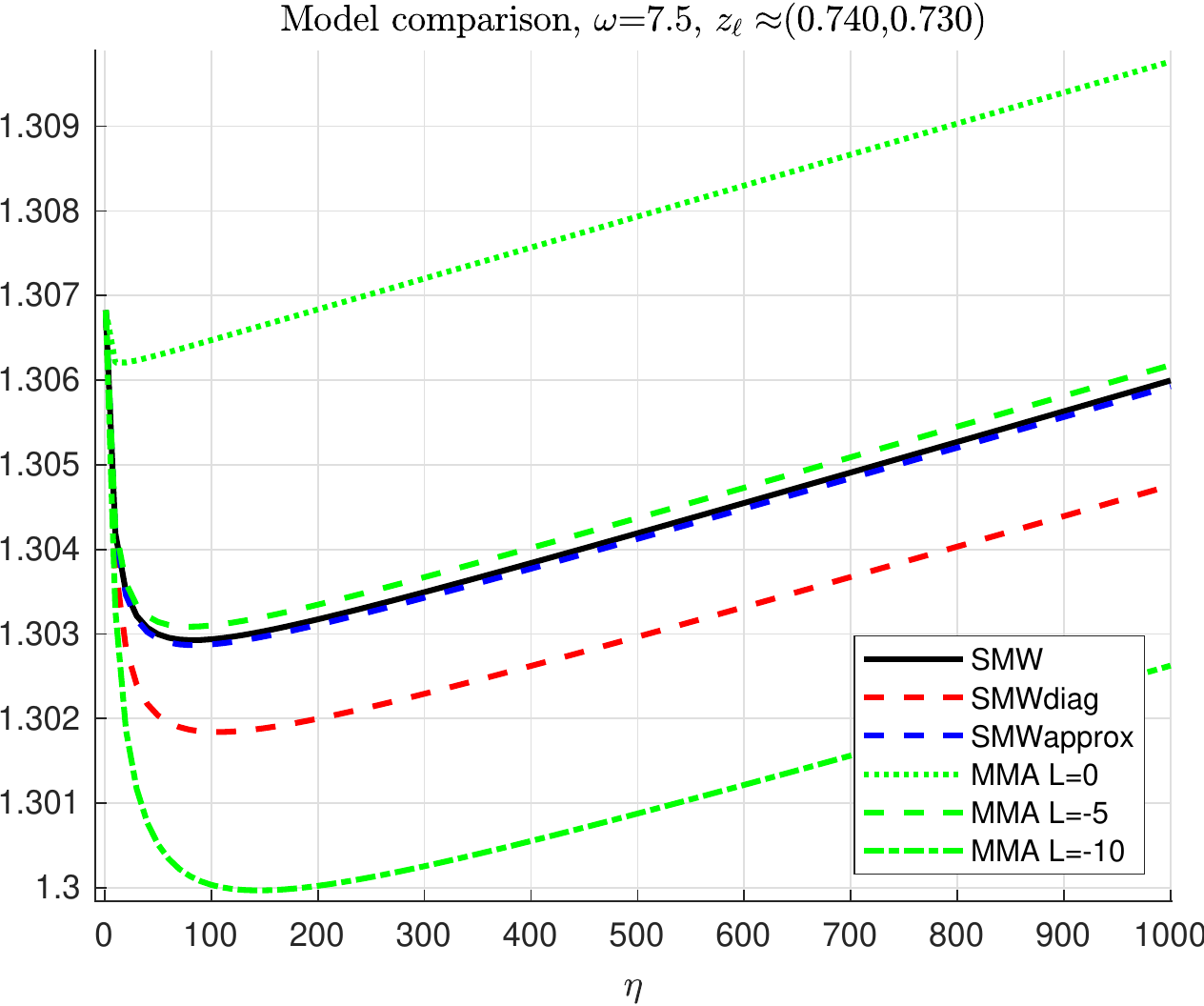} \\
    (g) & (h) & (i) 
    \end{tabular}
    \caption{Top and central row: Material distribution after one step of binary topology optimization for \eqref{eq_augmentedCost} with $\omega = 7.5$ starting out from homogeneous material $\lambda(x) = 1$ (see Fig. \ref{fig_design_solution_homo}(a)) when using (a) exact compliance model $\hat{\mathcal J}_{\text{SMW}}$, (b) diagonal approximation to Sherman-Morrison-Woodbury model $\hat{\mathcal J}_{\text{SMWdiag}}$, (c) approximate compliance model $\hat{\mathcal J}_{\text{SMWapprox}}$, (d) MMA model with $L=0$, (e) MMA model with $L=-5$, (f) MMA model with $L=-10$. Bottom row: Illustration of local models in three triangles marked in top row from bottom left (g) to top right (i).}
    \label{fig_towardsTopOpti}
\end{figure}

\subsection{Inhomogeneous material distribution} \label{sec_num_inhomo}
Next we consider a numerical example with an inhomogeneous material distribution as it may appear in the course of a density-based topology optimization algorithm, which is the motivation for this work. We consider a material coefficient $\lambda(x)$ that continuously varies between $\lambdaMin = 1$ and $\lambdaMax = 1000$ as
\begin{align*}
    \lambda(x) = \begin{cases}
                    \lambdaMin, & |x-m| \geq r_2, \\
                    \lambdaMin +\frac{|x-m|-r_1}{r_2 - r_1}(\lambdaMax - \lambdaMin), & |x-m| \in( r_1, r_2), \\
                    \lambdaMax, & |x-m| \leq r_1,
                 \end{cases}
\end{align*}
with $r_1 = 0.15$, $r_2 = 0.35$ and $m = (0.5, 0.5)^\top$. The material distribution and the corresponding finite element solution of \eqref{eq_probpde} with the data defined in the beginning of Section \ref{sec_numExp} are depicted in Fig. \ref{fig_design_solution_inhomo}.

\begin{figure}
    \begin{tabular}{cc}
        \includegraphics[width=.5\textwidth, trim=100 0 50 0, clip]{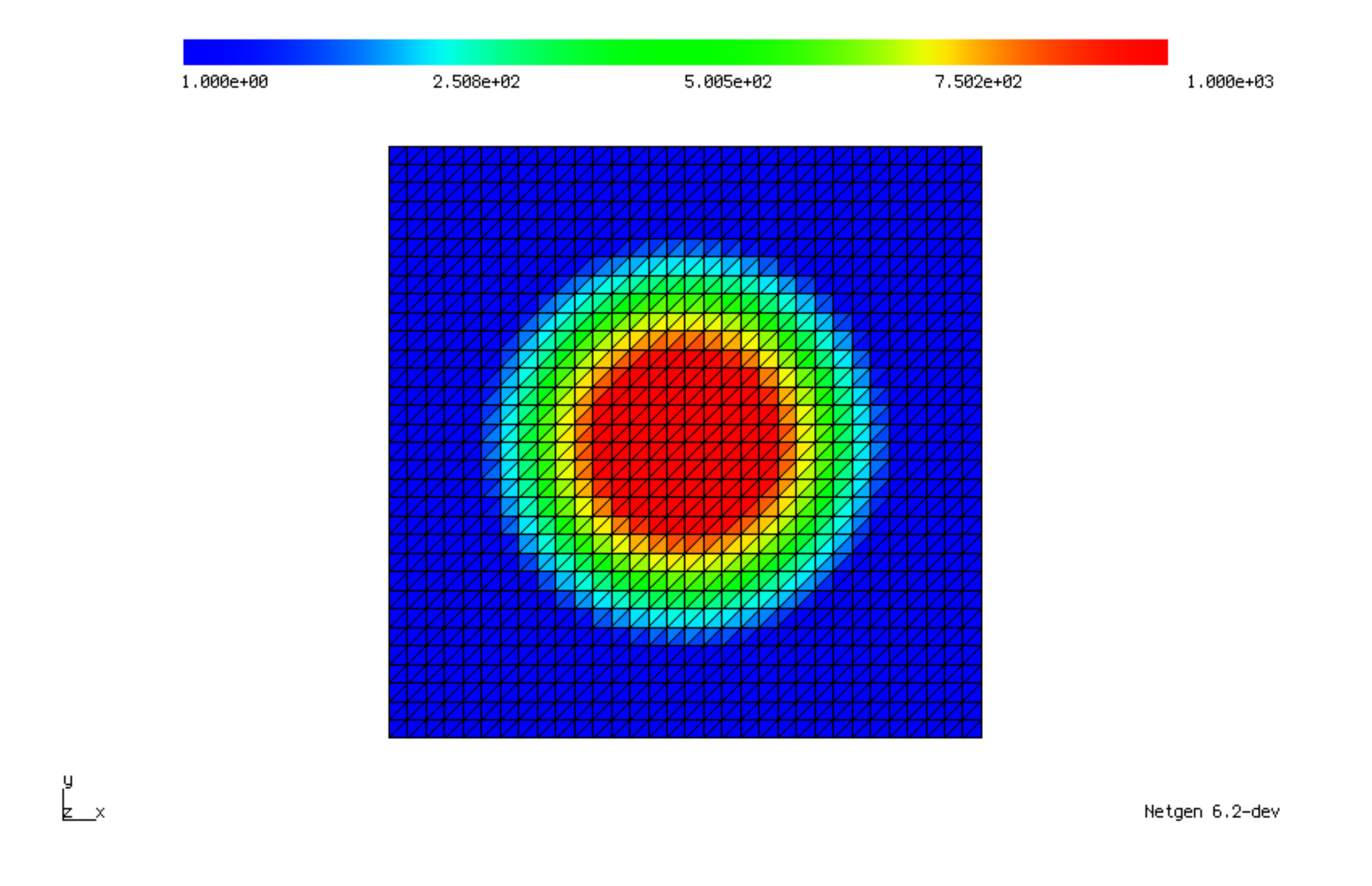} &
        \includegraphics[width=.5\textwidth, trim=100 0 50 0, clip]{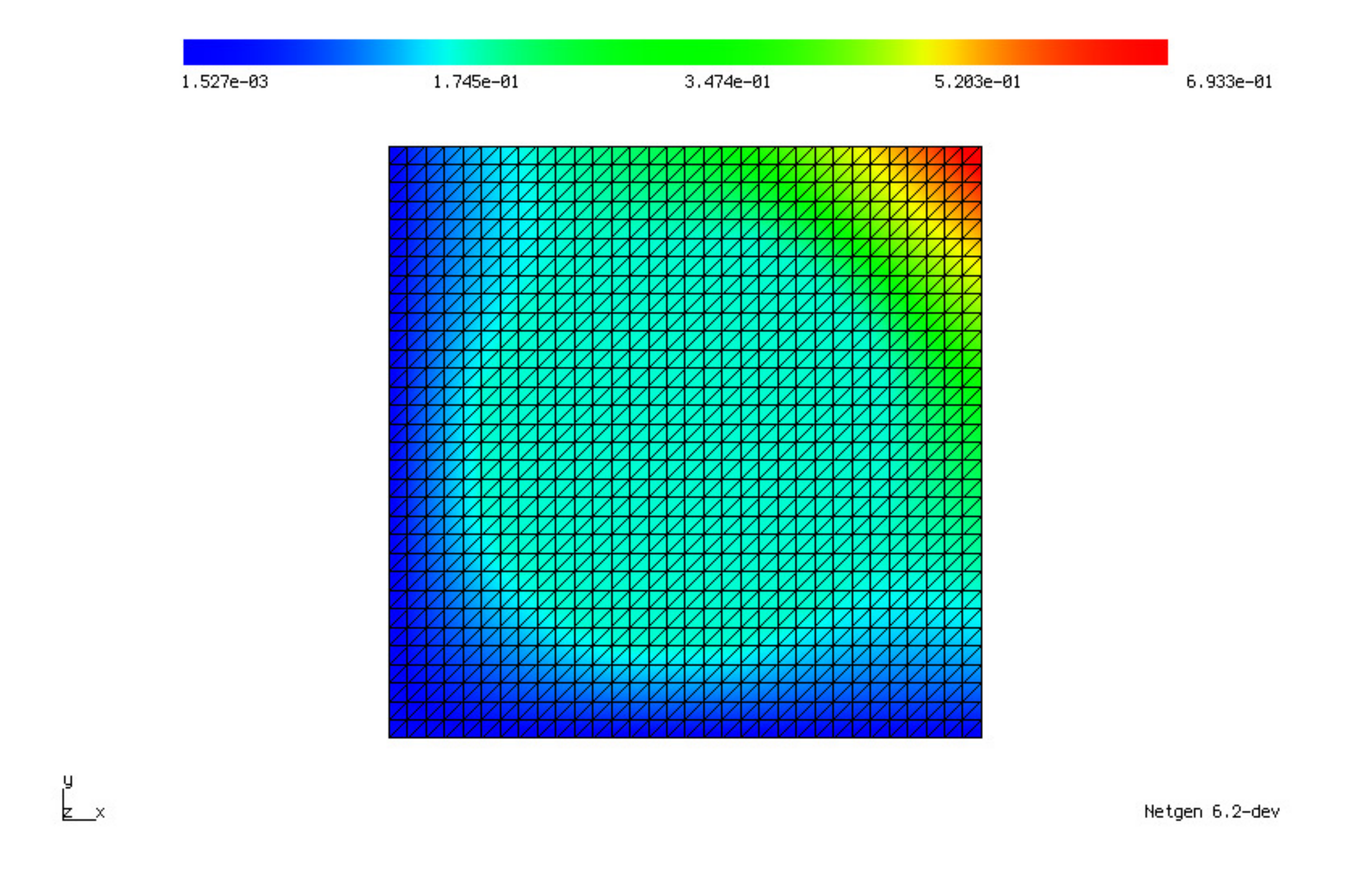} \\
        (a) & (b)
    \end{tabular}
    \caption{(a) Material coefficient $\lambda$ for inhomogeneous setting. (b) Finite element solution $u_h$ of problem \eqref{eq_probpde} with data specified in Section \ref{sec_numExp} for inhomogeneous material distribution.}
    \label{fig_design_solution_inhomo}
\end{figure}

\subsubsection{Numerical evaluation of $\hat{\mathcal J}_{\text{SMWapprox}}$ in inhomogeneous setting}
In order to evaluate $\hat{\mathcal J}_{\text{SMWapprox}}$ in inhomogeneous regions of the computational domain, recall that model \eqref{eq_SMW_W} involves the matrix $\BGam_{\hat T, \ell} = \BGam_{\hat T}[(\Blam_\ell, \lambda_{T_\ell}^{S_1}, \lambda_{T_\ell}^{S_2}, \lambda_{T_\ell}^{S_3})]$ defined in \eqref{eq_Gamma_approx} where $ \lambda_{T_\ell}^{S_1}, \lambda_{T_\ell}^{S_2}, \lambda_{T_\ell}^{S_3}$ are averaged material coefficients according to \eqref{eq_lambdaSj}, see also Fig. \ref{fig_illustrateAvg}. The matrix $\BGam_{\hat T, \ell}$, in turn, is based on the solution to the truncated transmission problem \eqref{eq_W_That_inhomo} for the given averaged material values. Thus, in order to evaluate $\hat{\mathcal J}_{\text{SMWapprox}}$, that exterior problem would have to be solved for the averaged material values of every single element, which would make the model computationally intractable.

For that reason, we introduce another approximation step: We precompute the matrix $\BGam_{\hat T}[(\Blam_\ell, \lambda_{T_\ell}^{S_1}, \lambda_{T_\ell}^{S_2}, \lambda_{T_\ell}^{S_3})]$ for a large discrete set of combinations of material coefficients in an offline phase. More precisely, we precompute the matrix $\BGam_{\hat T}[(\Blam_\ell, \lambda_{T_\ell}^{S_1}, \lambda_{T_\ell}^{S_2}, \lambda_{T_\ell}^{S_3})]$ for all combinations of tuples $(\Blam_\ell, \lambda_{T_\ell}^{S_1}, \lambda_{T_\ell}^{S_2}, \lambda_{T_\ell}^{S_3}) \in \{\eta^{(1)} ,  \dots, \eta^{(N)} \}^4$ where $\lambdaMin = \eta^{(1)} < \dots < \eta^{(N)}=\lambdaMax$ and for the two types of elements $\hat T = \hat T^{(1)}$ and  $\hat T = \hat T^{(2)}$, cf. Fig. \ref{fig_mesh_elTypes}. Since \eqref{eq_W_That_inhomo} has to be solved for $k=1,2$, this yields a total of $4 N^4$ finite element solutions of truncated transmission problems. In the online phase, given averaged values $(\Blam_\ell, \lambda_{T_\ell}^{S_1}, \lambda_{T_\ell}^{S_2}, \lambda_{T_\ell}^{S_3}) \notin \{\eta^{(1)} ,  \dots, \eta^{(N)} \}^4$, the matrix $\BGam_{\hat T}[(\Blam_\ell, \lambda_{T_\ell}^{S_1}, \lambda_{T_\ell}^{S_2}, \lambda_{T_\ell}^{S_3})]$ is approximated by piecewise linear interpolation of the precomputed data.

In our experiments, in order to numerically approximate \eqref{eq_W_That_inhomo}, we used the moderately large value $R=30$ for the radius of the computational domain and discretized it by a mesh consisting of about 4400 triangular elements and about 2300 vertices. We chose $N=16$ material points between $\lambdaMin = 1$ and $\lambdaMax = 1000$ which were chosen as stated in Table \ref{tab_valuesEtak16}. Since increasing the number of material points $N$ will drastically increase the precomputation time, the concrete choice of these points is of big importance. Thus, the points were chosen such that the interpolation error that is made in the online phase is as small as possible, see Remark \ref{rem_equiError}.
The total precomputation time for this setting was about two hours on a single core.
Note that, for given PDE constraint, discretization method and material catalogue $\{\eta^{(1)},\dots, \eta^{(N)}\}$, this precomputation step has to be performed only once and can henceforth be used in all optimization runs.

\subsubsection{Numerical comparison of models on computational domain} \label{sec_comp_domain_inhomo}
We make the same comparison of the two most promising models $\hat{\mathcal J}_{\text{SMWdiag}}$ \eqref{eq_SMWdiag_model} and $\hat{\mathcal J}_{\text{SMWapprox}}$ \eqref{eq_SMW_W} as it was done for the homogeneous setting in Section \ref{sec_comp_domain_homo}. Again, recall that $\hat{\mathcal J}_{\text{TDnum}}$ coincides with $\hat{\mathcal J}_{\text{SMWapprox}}$ and is thus not examined separately.

Figure \ref{fig_errorDomain_inhomo} again shows the maximum relative error $\delta \hat{\mathcal J}$ of these two models over the computational domain. Again, different thresholds of the color bar are shown. As it was already observed in Section \ref{sec_comp_domain_homo}, the model $\hat{\mathcal J}_{\text{SMWapprox}}$ behaves particularly well in regions of homogeneous material. For both models, the largest errors occur at the transition from homogeneous material $\lambda(x) = \lambdaMin$ to inhomogeneous material.

From Fig. \ref{fig_errorDomain_inhomo} it can also be seen that the largest error of $\hat{\mathcal J}_{\text{SMWapprox}}$ is around $315\%$ compared to only about $100\%$ for $\hat{\mathcal J}_{\text{SMWdiag}}$. However, we mention that this effect disappears when a finer mesh is chosen as it is illustrated in Fig. \ref{fig_errorDomain_inhomo_different_h}. There, it can be seen that the maximum error of $\hat{\mathcal J}_{\text{SMWapprox}}$ in the refined mesh is only around $86\%$ which is in the same range as for $\hat{\mathcal J}_{\text{SMWdiag}}$. The reason for this improvement of $\hat{\mathcal J}_{\text{SMWapprox}}$ is that,
for a given function $\lambda(x)$, as the mesh size decreases, the range of values to be averaged in the direct neighborhood of an element becomes smaller which results in a smaller error when computing the average values $\lambda_{T_\ell}^{S_k}$ \eqref{eq_lambdaSj}, $k=1,2,3$. In general, the model $\hat{\mathcal J}_{\text{SMWapprox}}$ behaves well if material variations around a fixed element are small and makes larger approximation errors when large ranges of material values have to be averaged.

\begin{figure}
    \begin{tabular}{cc}
        $\delta \hat{\mathcal J}_{\text{SMWdiag}}$ &
        $\delta \hat{\mathcal J}_{\text{SMWapprox}}$\\
        \includegraphics[width=.45\textwidth, trim=100 0 50 0, clip]{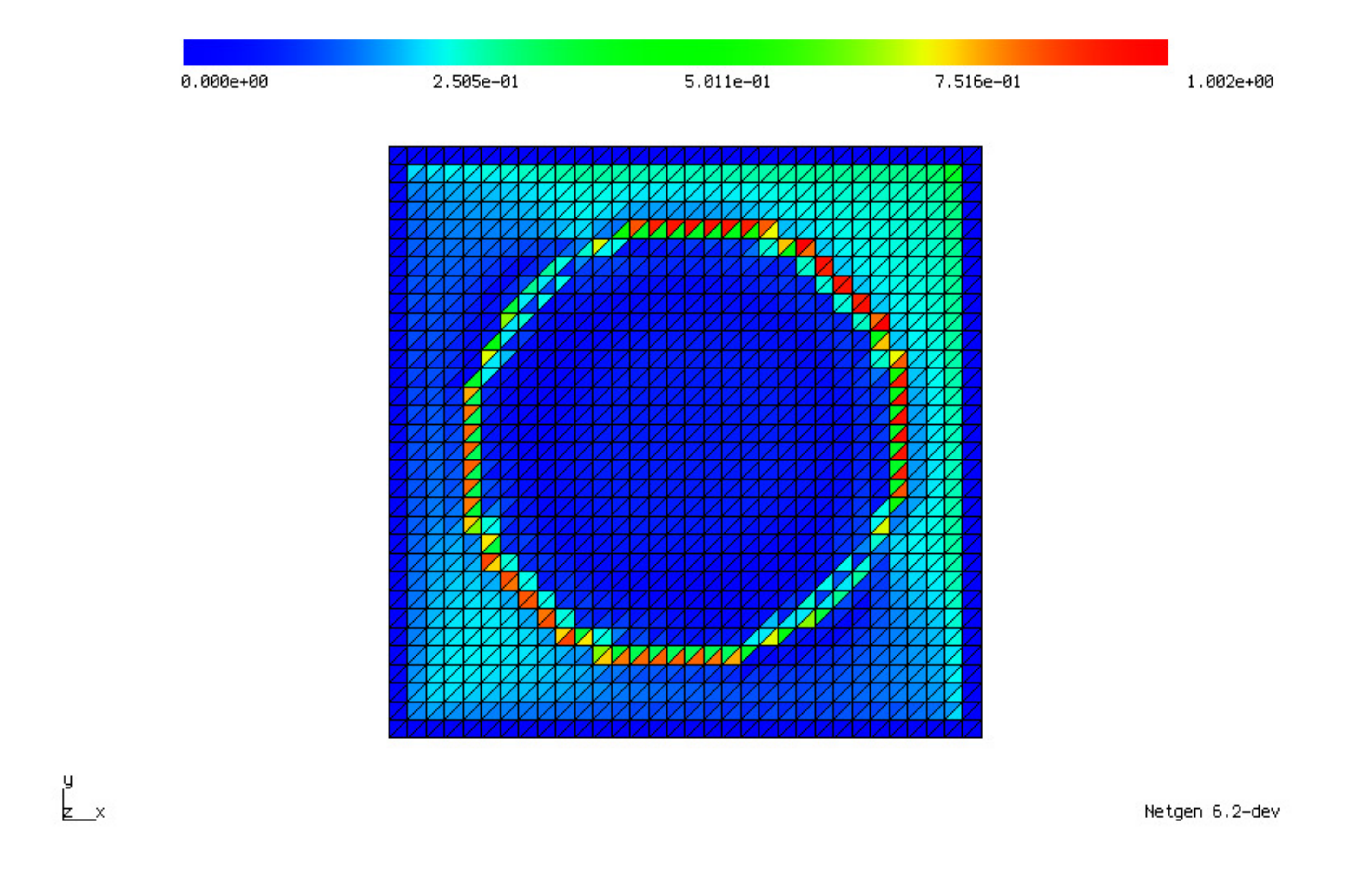} &
        \includegraphics[width=.45\textwidth, trim=100 0 50 0, clip]{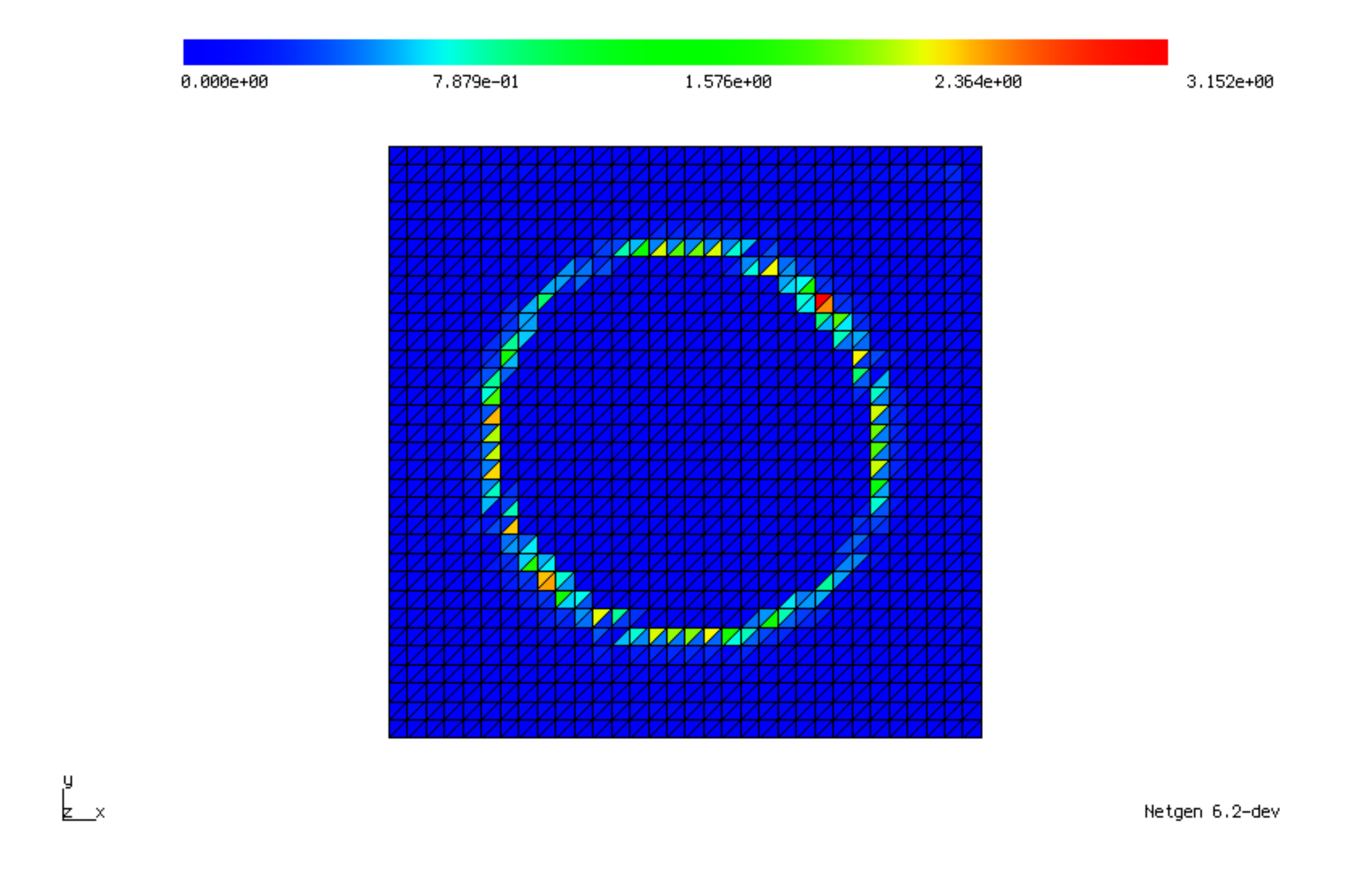} \\
        \includegraphics[width=.45\textwidth, trim=100 0 50 0, clip]{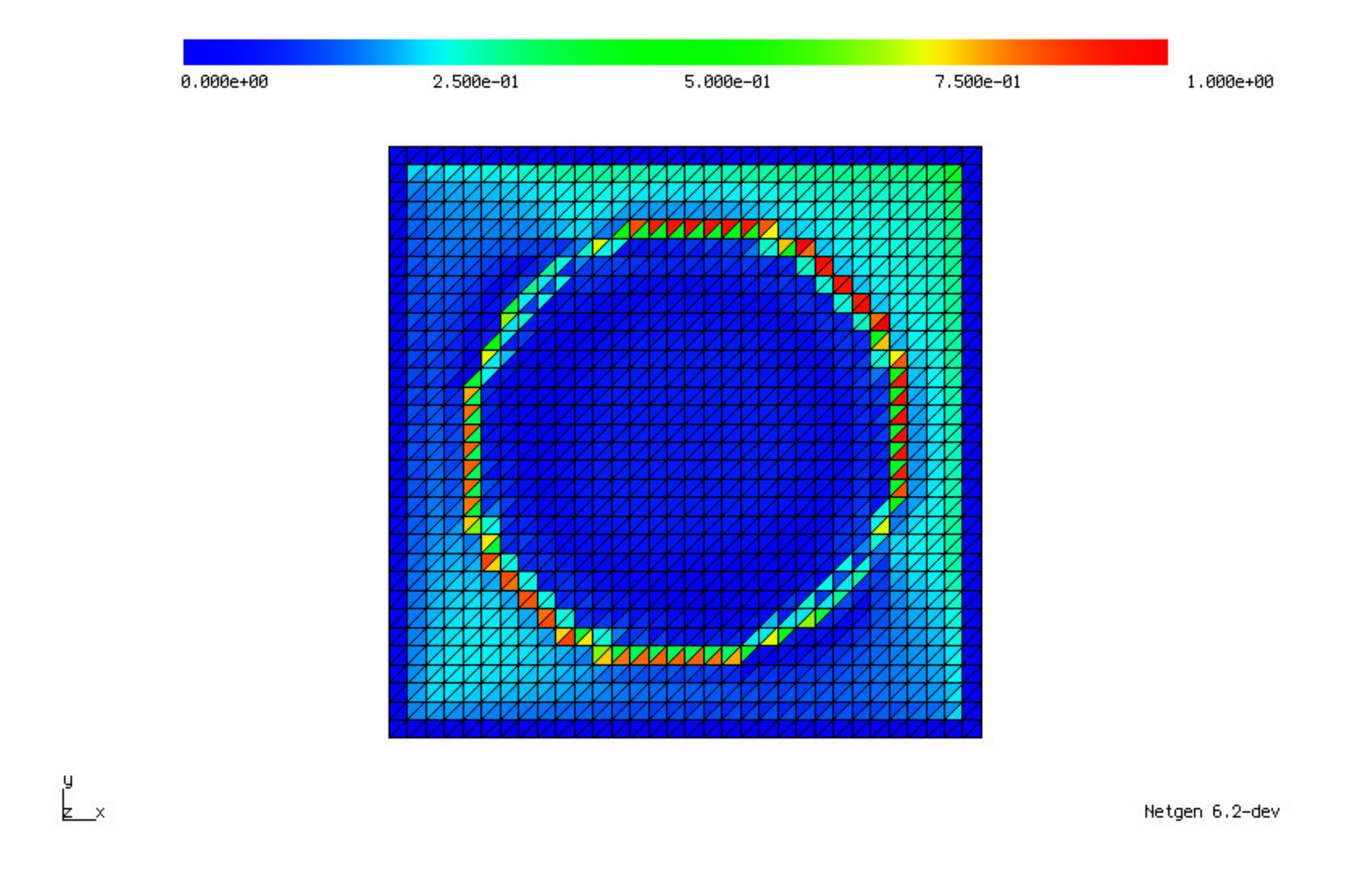} &
        \includegraphics[width=.45\textwidth, trim=100 0 50 0, clip]{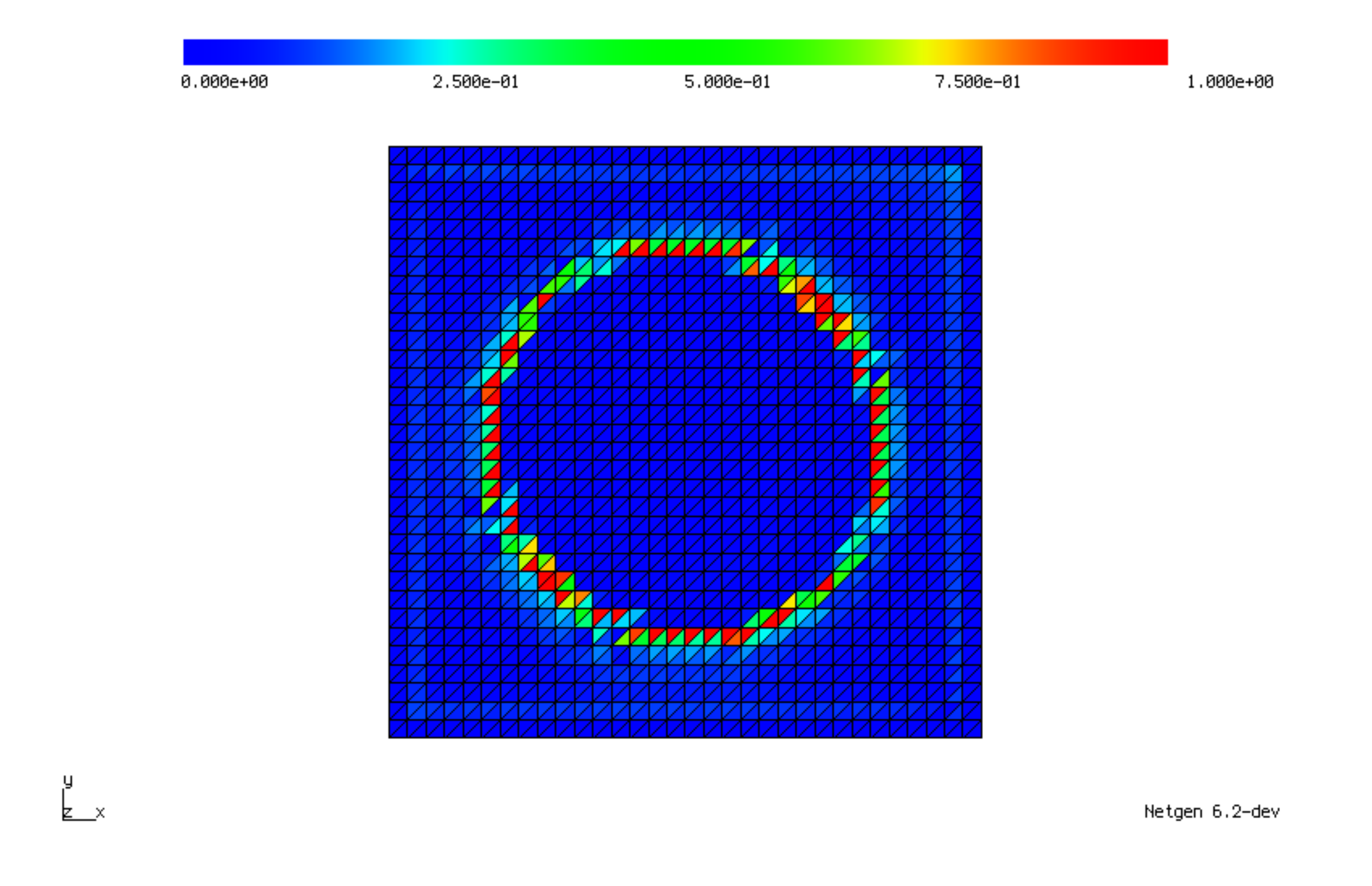} \\
        \includegraphics[width=.45\textwidth, trim=100 0 50 0, clip]{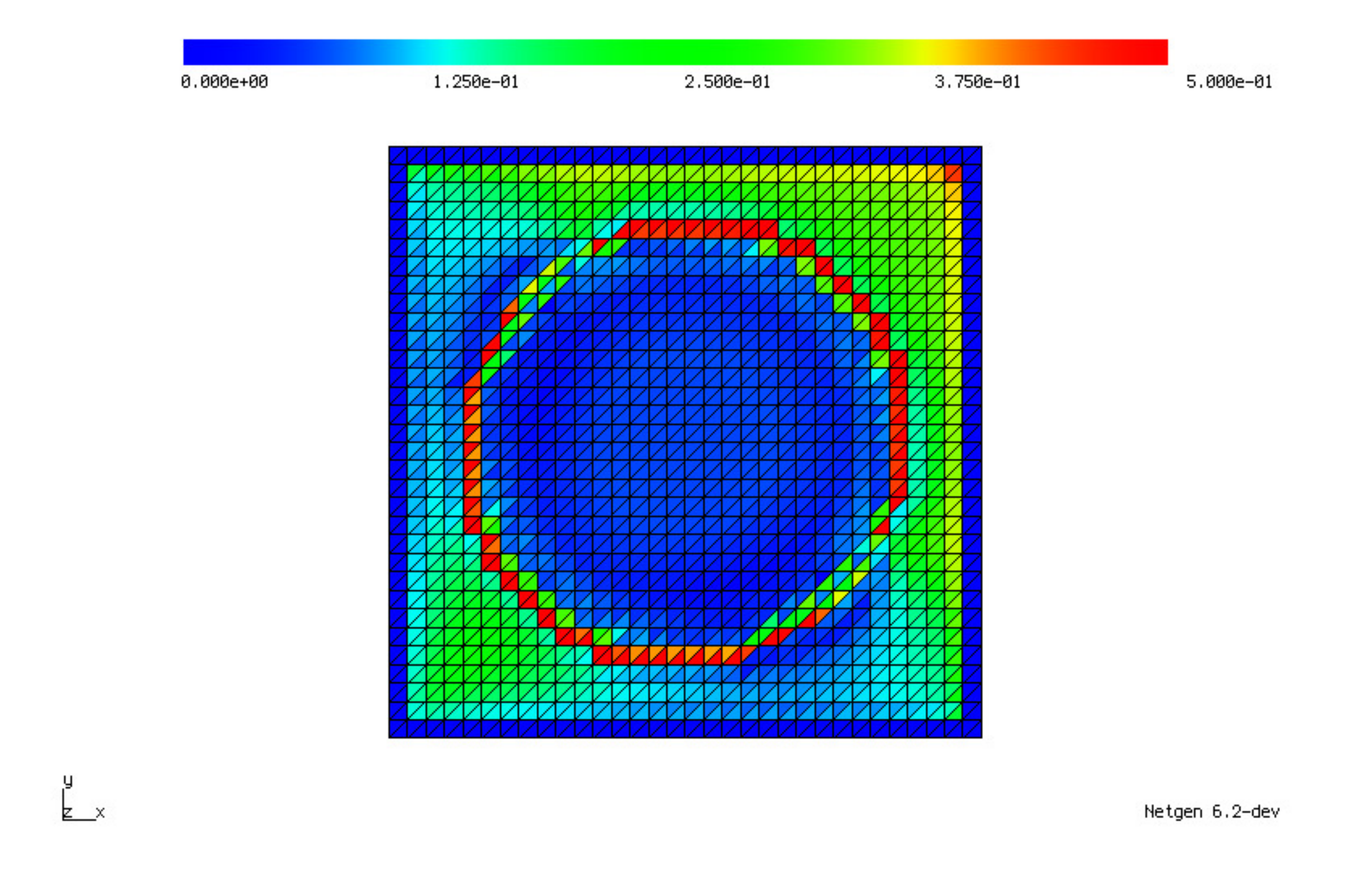} &
        \includegraphics[width=.45\textwidth, trim=100 0 50 0, clip]{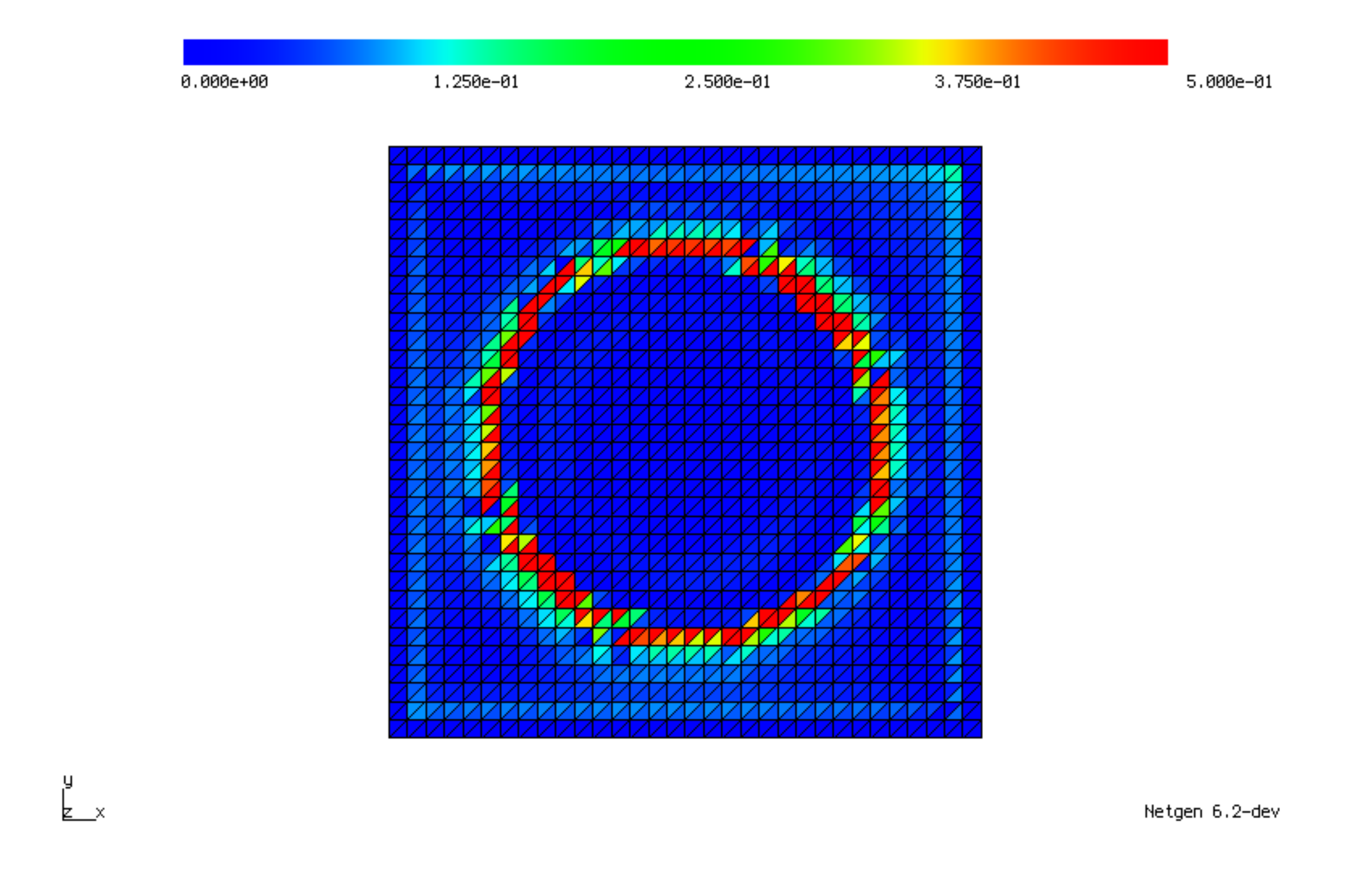} \\
        \includegraphics[width=.45\textwidth, trim=100 0 50 0, clip]{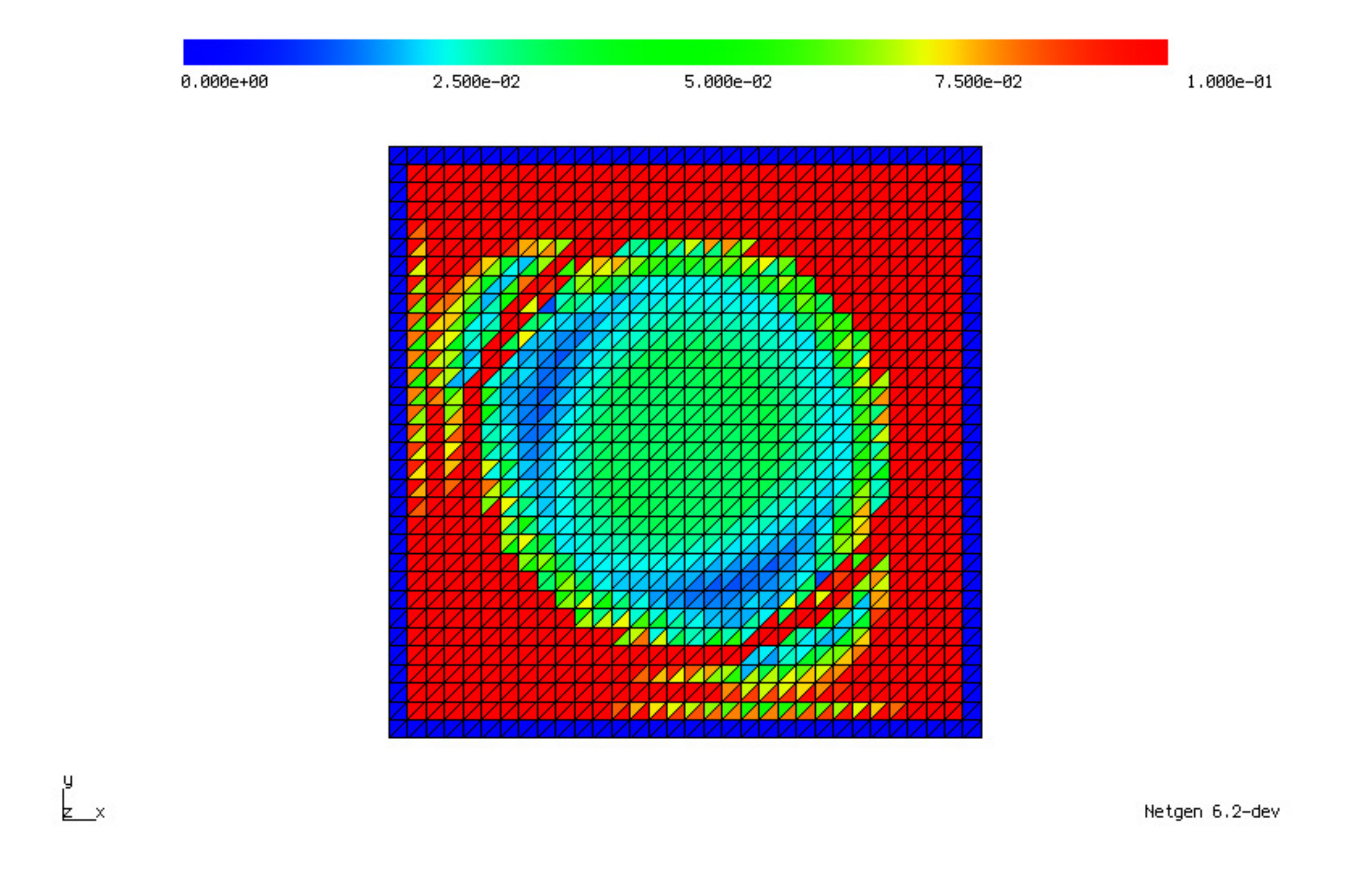} &
        \includegraphics[width=.45\textwidth, trim=100 0 50 0, clip]{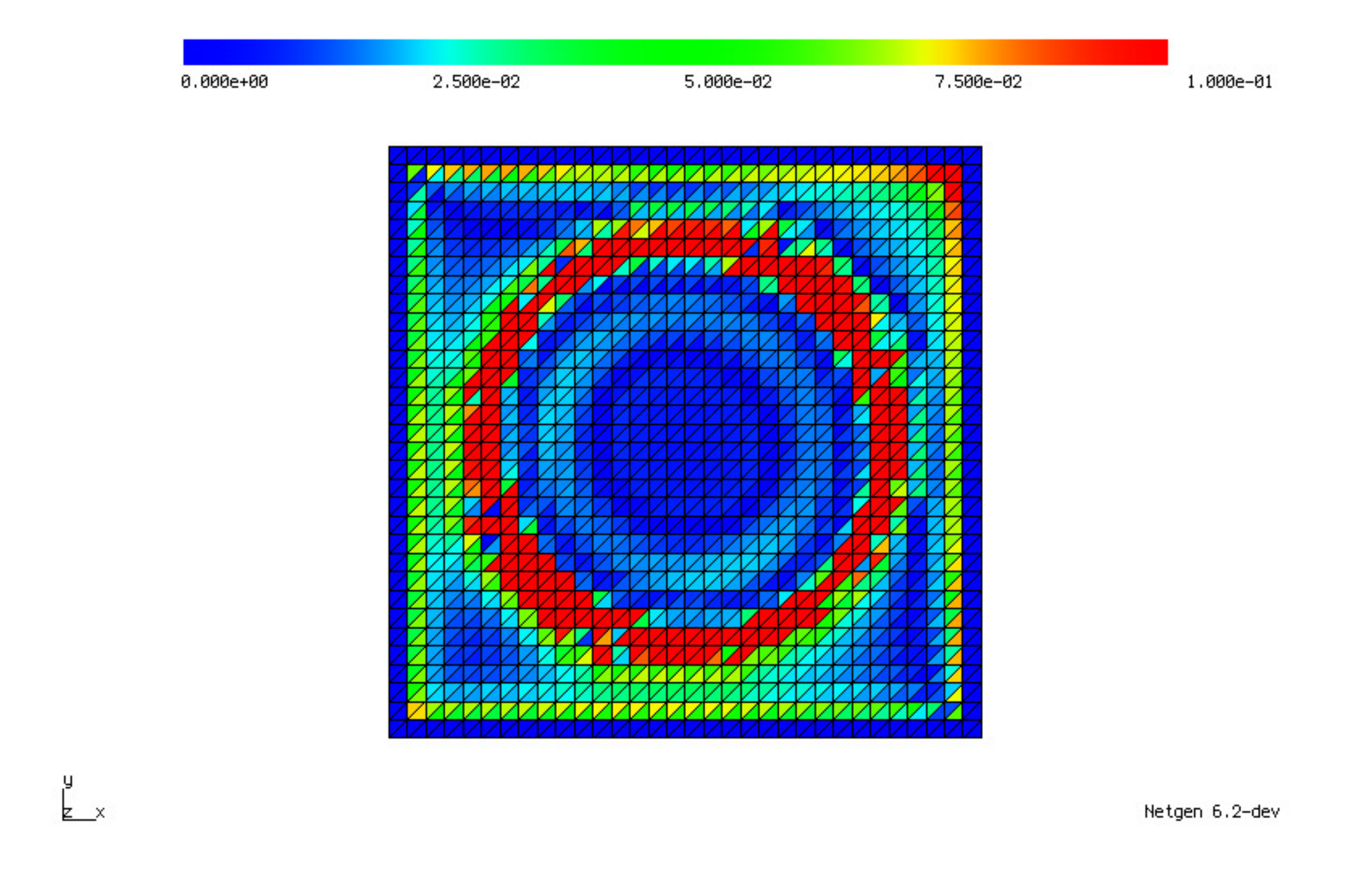}
    \end{tabular}
    \caption{Comparison of relative errors $\delta \hat{\mathcal J}[T_\ell]$ according to \eqref{eq_max_eta} for models $\hat{\mathcal J}_{\text{SMWdiag}}$ \eqref{eq_SMWdiag_model} in left column and $\hat{\mathcal J}_{\text{SMWapprox}}$ \eqref{eq_SMW_W} ($=\hat{\mathcal J}_{\text{TDnum}}$ \eqref{eq_defTDnum_model}) in right column for all interior elements $T_\ell$ in inhomogeneous setting. First line shows color plot according to their respective maximum errors. Second to fourth line show threshold for maximum relative error at $100 \%$, $50\%$ and $10\%$, respectively. Errors in elements touching the boundary are not computed.}
    \label{fig_errorDomain_inhomo}
\end{figure}

\begin{figure}
    \begin{tabular}{ccc}
     \includegraphics[width=.33\textwidth, trim=100 0 50 0, clip]{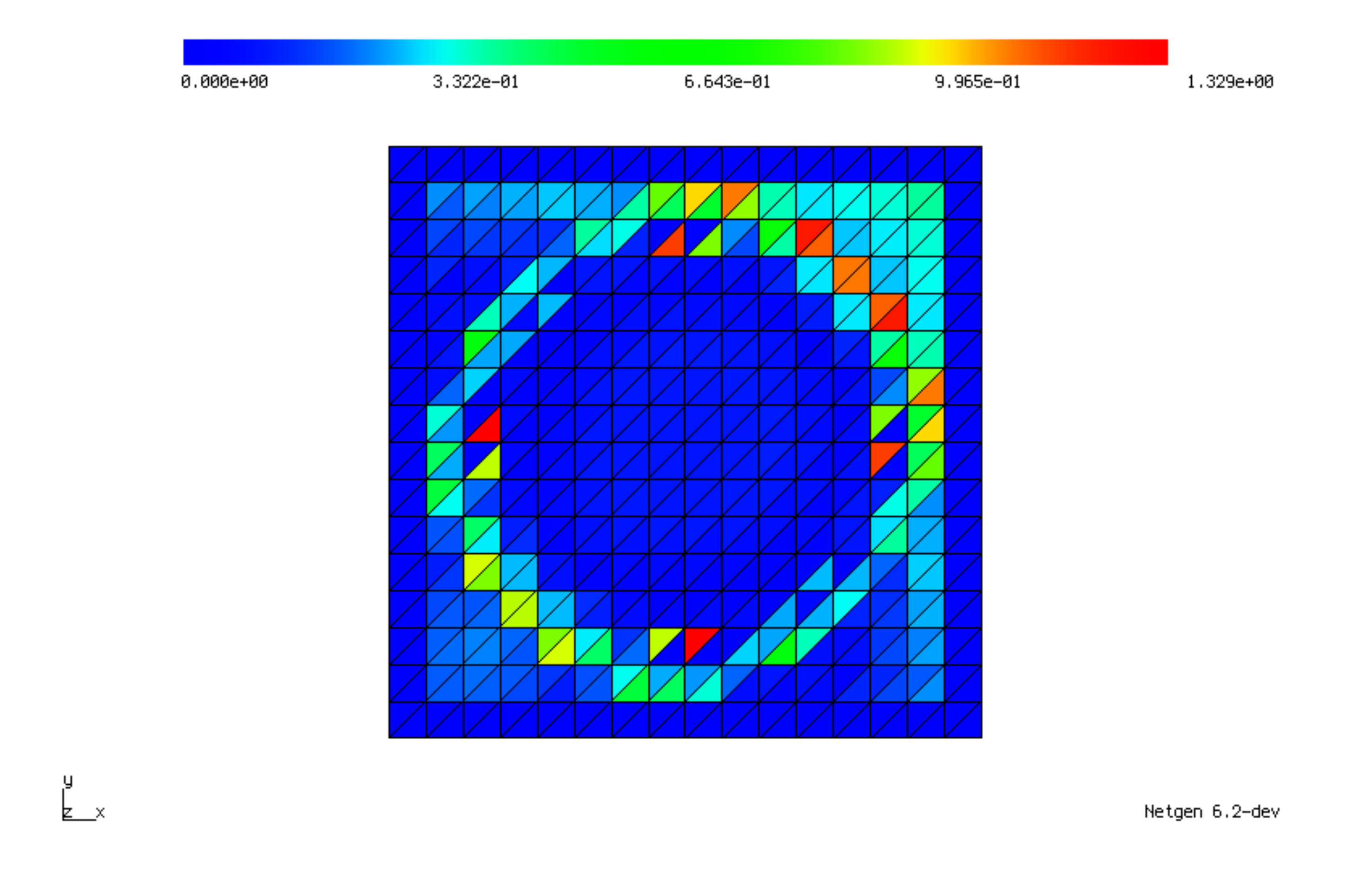} &
     \includegraphics[width=.33\textwidth, trim=100 0 50 0, clip]{rel_SMWdiag-eps-converted-to.pdf} &
     \includegraphics[width=.33\textwidth, trim=100 0 50 0, clip]{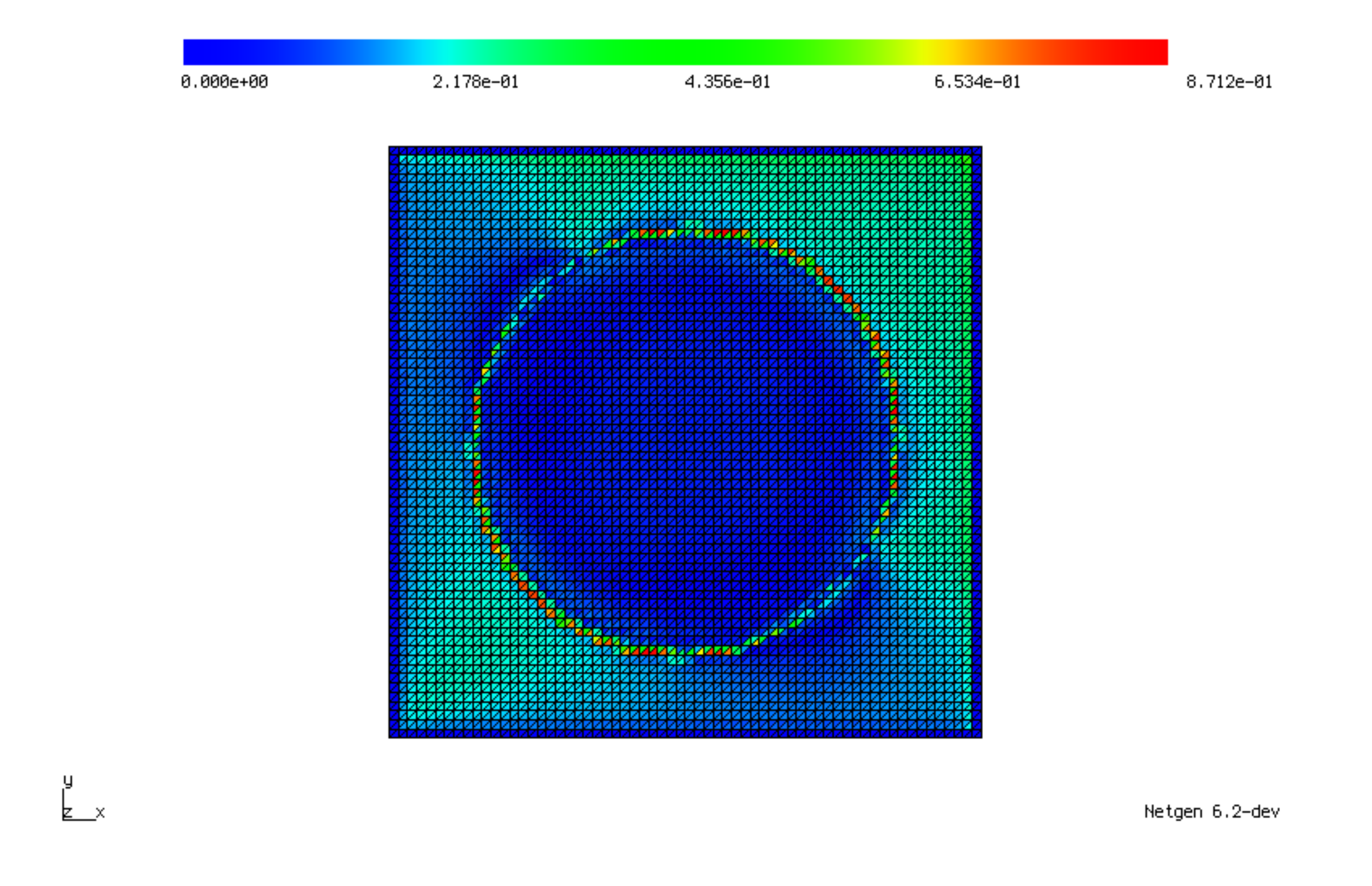}\\
     \includegraphics[width=.33\textwidth, trim=100 0 50 0, clip]{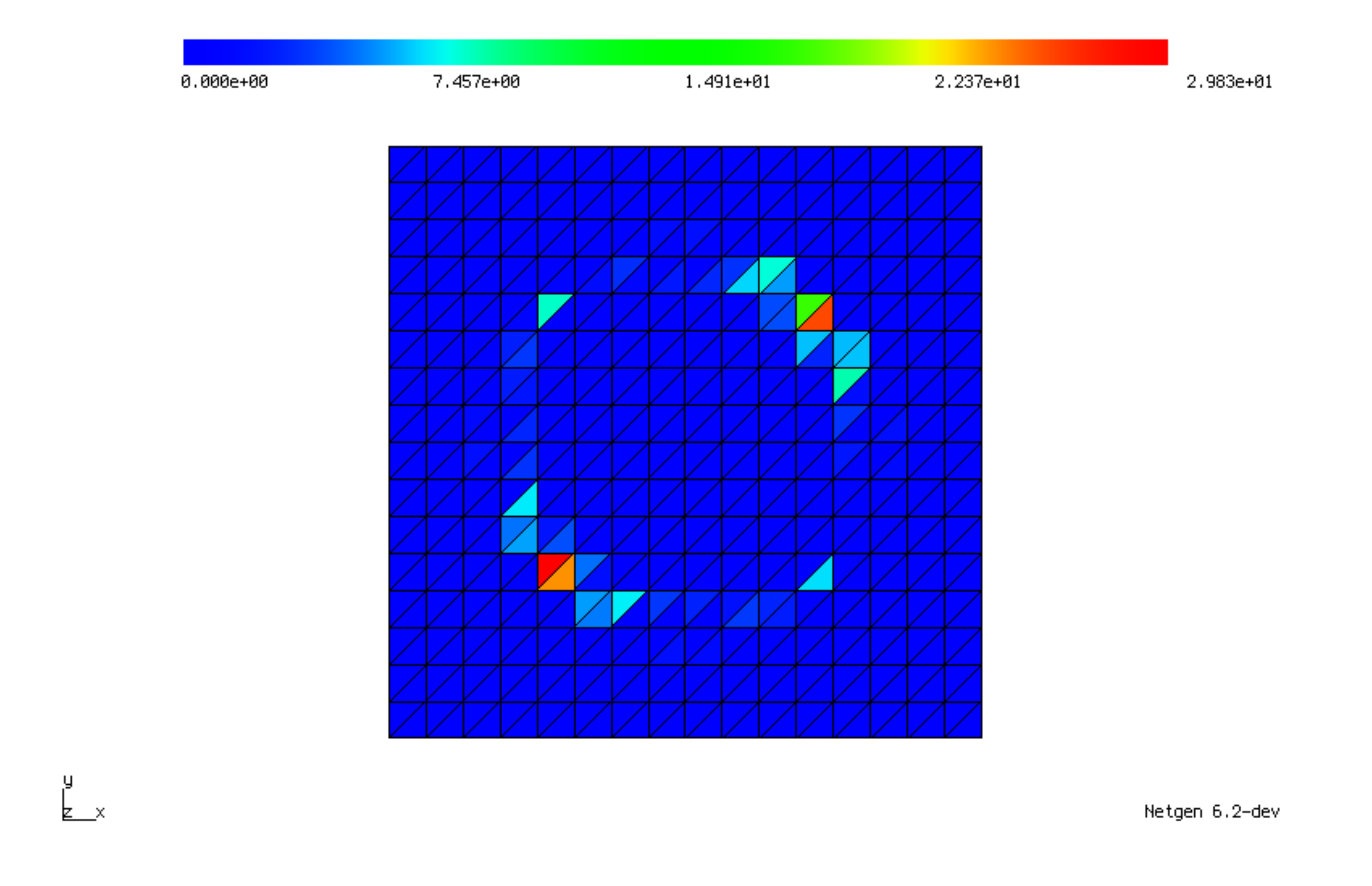} &
     \includegraphics[width=.33\textwidth, trim=100 0 50 0, clip]{rel_SMWUBprec_alpham0p5-eps-converted-to.pdf} &
     \includegraphics[width=.33\textwidth, trim=100 0 50 0, clip]{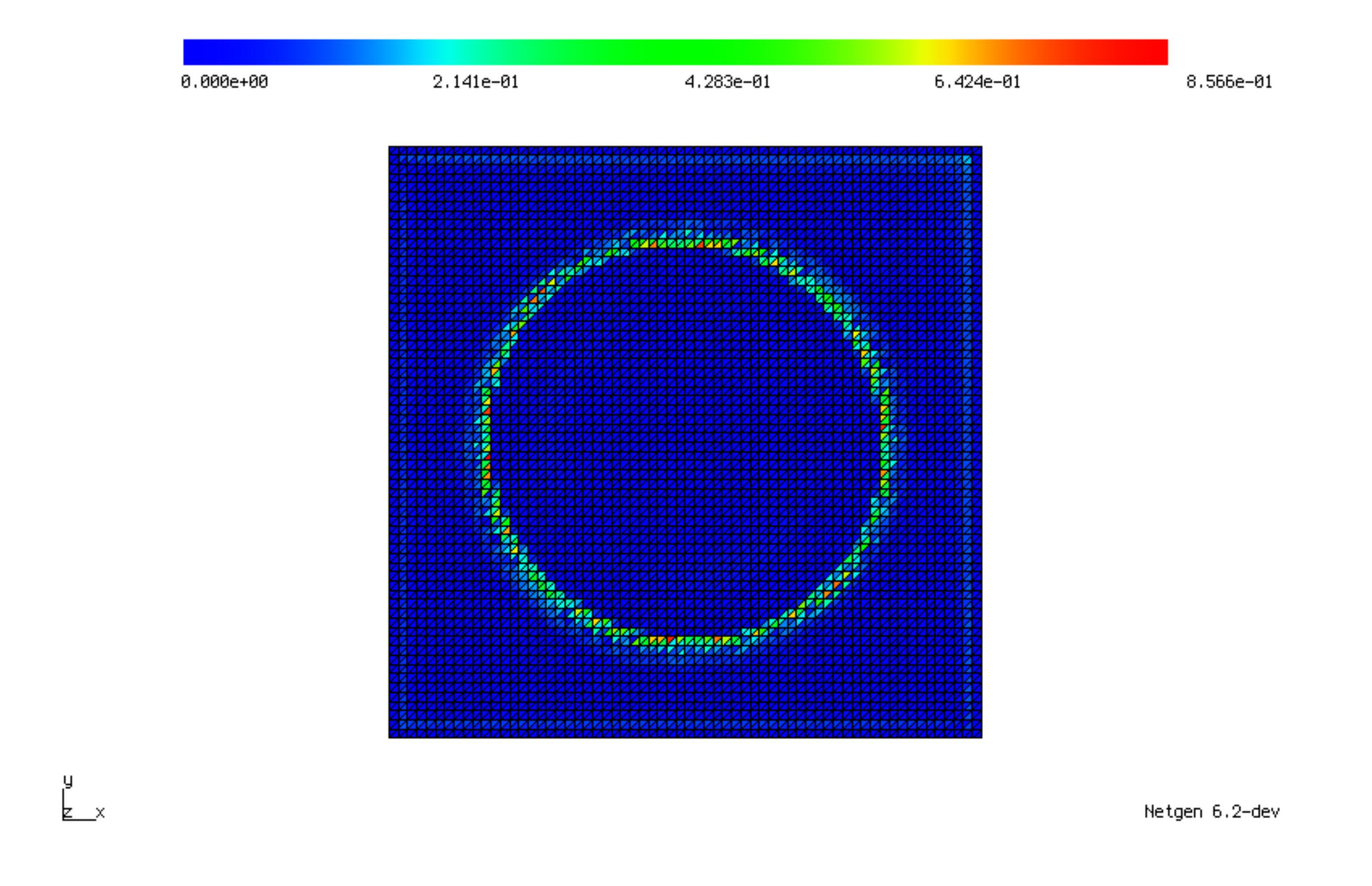}
    \end{tabular}
    \caption{Comparison of relative errors in inhomogeneous setting for different mesh sizes for models $\hat{\mathcal J}_{\text{SMWdiag}}$ (top row) and $\hat{\mathcal J}_{\text{SMWapprox}} = \hat{\mathcal J}_{\text{TDnum}}$ (bottom row). Errors in elements touching the boundary are not computed.}
    \label{fig_errorDomain_inhomo_different_h}
\end{figure}

\subsection{Further improvements}
Finally, we point out several directions in which this research could be extended to further improve the models $\hat{\mathcal J}_{\text{SMWapprox}}$, $\hat{\mathcal J}_{\text{TDnum}}$.

\subsubsection{Boundary regions} \label{sec_elemsBdy}
So far, we restricted our numerical results to regions away from the boundary and did not treat elements that touch the boundary. The reason for this is that, in the derivation of the models \eqref{eq_defTDnum_model} and \eqref{eq_SMW_W}, the truncated exterior problems \eqref{eq_K_That_FEM} and \eqref{eq_W_That_inhomo} are obtained by zooming in around the fixed element $T_\ell$ and rescaling. Thus, in the case where $T_\ell$ touches the boundary, the truncated domains depicted in Figures \ref{fig_pert_rescaled_h0__trunc} and \ref{fig_smwapprox_pert_rescaled_h0__trunc} do not mimic the neighborhood of $T_\ell$. Instead, it would be more appropriate to perform precomputations on truncated half spaces as depicted in Figure \ref{fig_extProblem_bdy}. The figure shows the setting of problem \eqref{eq_W_That_homo} in the case of a homogeneous material distribution corresponding to elements touching the top boundary. Also, here, inhomogeneous material can be treated by precomputing a range of combinations of material values in an offline phase and interpolating averaged sector values in the online phase. Here, the precomputation becomes a bit more involved since, in addition to accounting for different materials and different element types and $k=1,2$ in \eqref{eq_W_That_homo}, one also has to distinguish between a left, bottom, right or top boundary as well as between Dirichlet or Neumann conditions imposed on that boundary. Thus, in order to also treat boundary regions of a rectangular domain $\Dsf$, an additional $32N^4$ truncated half space problems have to be solved in the offline phase. Here, $N$ is the number of used material values, e.g., $N=16$.

\begin{figure}
    \begin{tabular}{cc}
        \includegraphics[width=.5\textwidth]{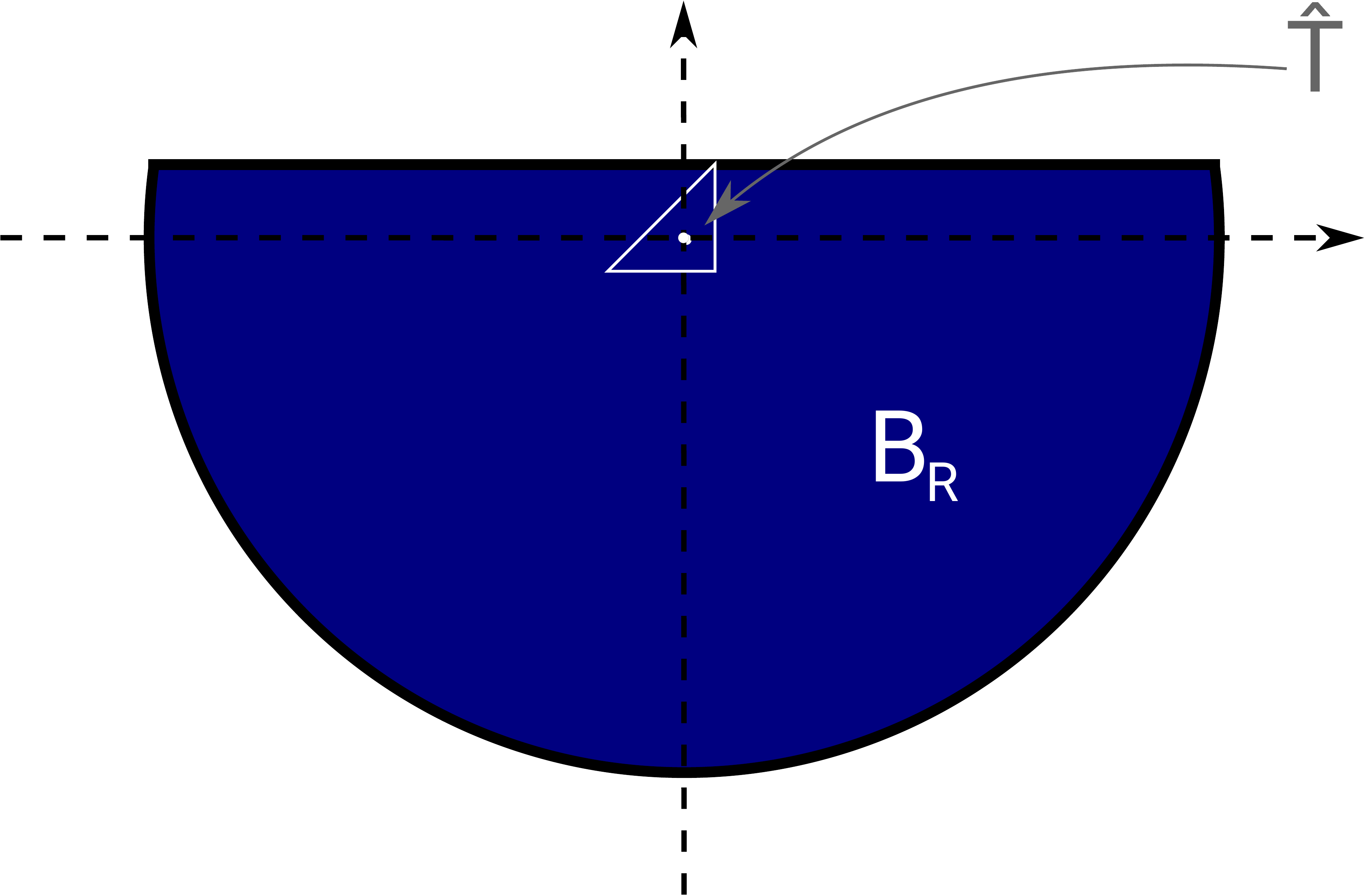} &
        \includegraphics[width=.5\textwidth]{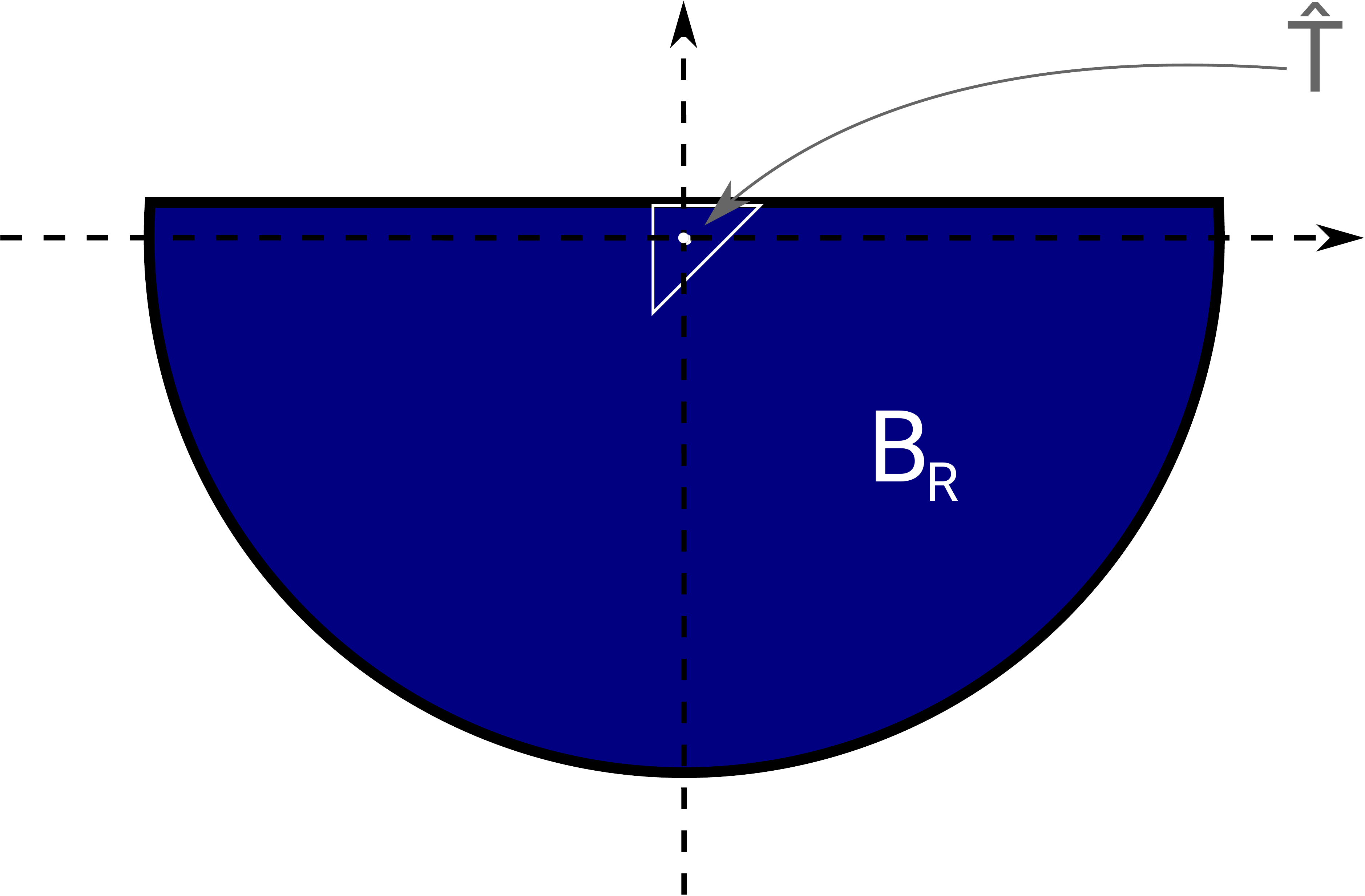} \\
        (a) & (b)
    \end{tabular}
    \caption{Setting for exterior problem \eqref{eq_K_That_FEM} corresponding to elements (a) of type 1 or (b) of type 2 that touch a top boundary. On the circular part of the boundary of $B_R$, homogeneous Dirichlet conditions are set. The boundary conditions at top can be either of Dirichlet or Neumann type, depending on the physical problem.}
    \label{fig_extProblem_bdy}
\end{figure}

In Fig. \ref{fig_improvementBdy}, we illustrate the improvement when elements touching the top boundary are given a special treatment by precomputing the matrix $\BGam_{\hat T}[(\Blam_\ell, \lambda_{T_\ell}^{S_1}, \lambda_{T_\ell}^{S_2}, \lambda_{T_\ell}^{S_3})]$ also for the reference elements $\hat T$ depicted in Figure \ref{fig_extProblem_bdy}. In addition to the data presented in Figure \ref{fig_errorDomain_inhomo}, we also computed the maximum relative errors in all elements touching the top boundary (of Neumann type). If the same data as in the interior is used, the maximum relative error is attained in the elements at the boundary and is as high as $77.89\%$. When the mentioned treatment of the boundary elements is used, the maximum error is still attained at an interior element and is only $16.6\%$.
\begin{figure}
    \begin{tabular}{cc}
        \includegraphics[width=.5\textwidth, trim = 0 500 0 0, clip]{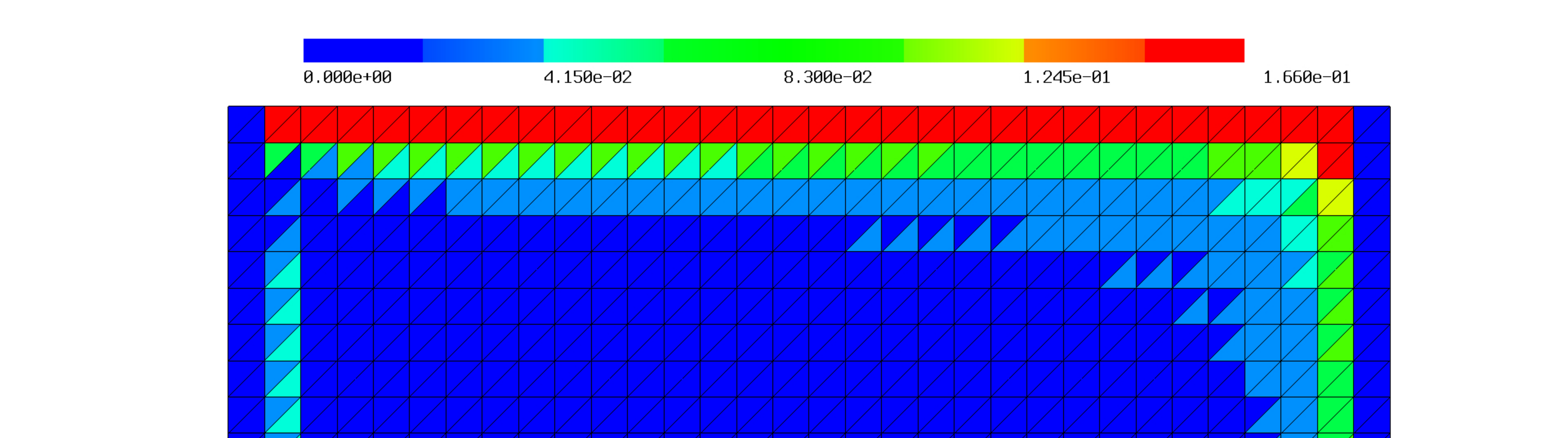} &
        \includegraphics[width=.5\textwidth, trim = 0 500 0 0, clip]{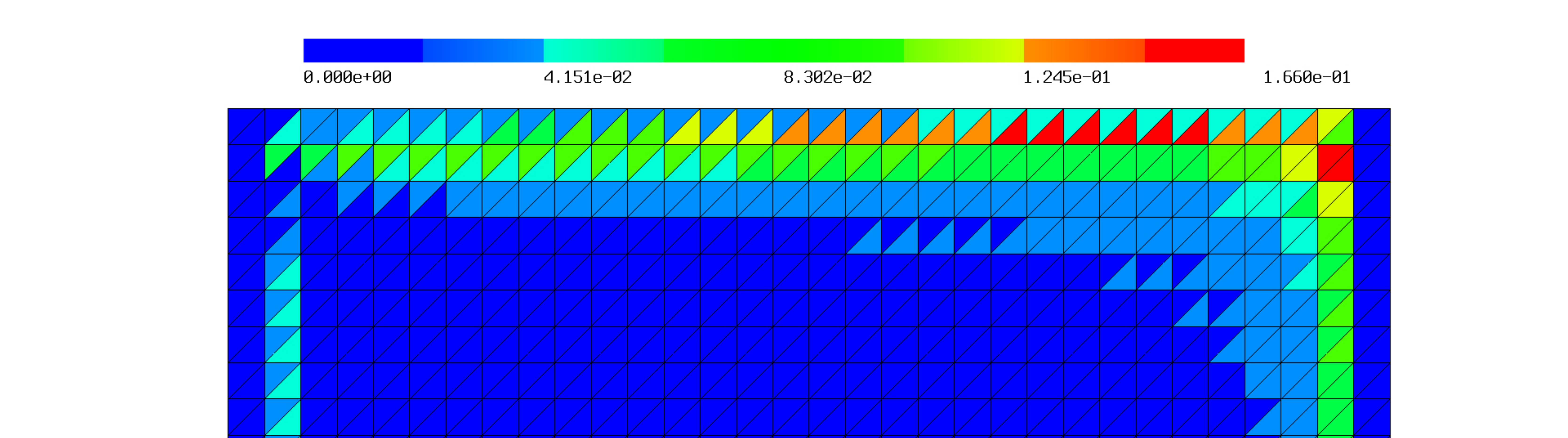} \\
        (a) & (b)
    \end{tabular}
    \caption{Comparison of model $\hat{\mathcal J}_{\text{SMWapprox}}$ in Neumann boundary region (a) without and (b) with special precomputation using truncated half spaces as depicted in Fig. \ref{fig_extProblem_bdy}. For comparison, the same color scale that is cut off at $16.6\%$ is used in (a) and (b). The maximum relative error in (a) is $77.89\%$ whereas it is only $16.6\%$ in (b).}
    \label{fig_improvementBdy}
\end{figure}

\subsubsection{Averaging of inhomogeneous material distribution}
We mention that we observed that the way material values are averaged over sectors has a strong impact on the obtained relative error. Figure \ref{fig_errorDomain_inhomo_different_alpha} shows the same experiments as discussed in Section \ref{sec_num_inhomo} for different values of $\alpha$ in the averaging process \eqref{eq_lambdaSj}. For $\alpha = 1$, the weighted H\"older mean \eqref{eq_lambdaSj} reduces to the weighted arithmetic mean which yields a large maximum error of about $1300\%$, see Fig. \ref{fig_errorDomain_inhomo_different_alpha}(a). Recall that our choice $\alpha = -0.5$ yielded the result in the right column of Fig. \ref{fig_errorDomain_inhomo} with maximal value of $\delta\hat{\mathcal J} \approx 315\%$. Further numerical studies for $\alpha = -0.2$ and $\alpha = -0.1$ are depicted in Fig. \ref{fig_errorDomain_inhomo_different_alpha}(b)-(c) showing that for $\alpha = -0.2$ the maximal error in the mesh is actually smaller than for the model $\hat{\mathcal J}_{\text{SMWdiag}}$.

Moreover, one might want to think of decomposing the domain $B_R(0)$ into more than three sectors in order to reduce the error made by the averaging of inhomogeneous material parameters. However, here one should keep in mind that the number of truncated exterior problems to be solved in the precomputation phase with $n_{\text{sec}}$ sectors and $N$ material points is of the order $N^{n_{\text{sec}}+1}$ and thus grows very fast with $n_{\text{sec}}$.

Finally, taking the average of material values on more than one layer of elements (cf. Fig. \ref{fig_illustrateAvg}) together with suitable distance-dependent weights could lead to a better representation of the local material configuration and thus to potentially higher accuracy of the models  $\hat{\mathcal J}_{\text{SMWdiag}}$ and $\hat{\mathcal J}_{\text{TDnum}}$ in an inhomogeneous setting.

\begin{figure}
    \begin{tabular}{ccc}
     \includegraphics[width=.33\textwidth, trim=100 0 50 0, clip]{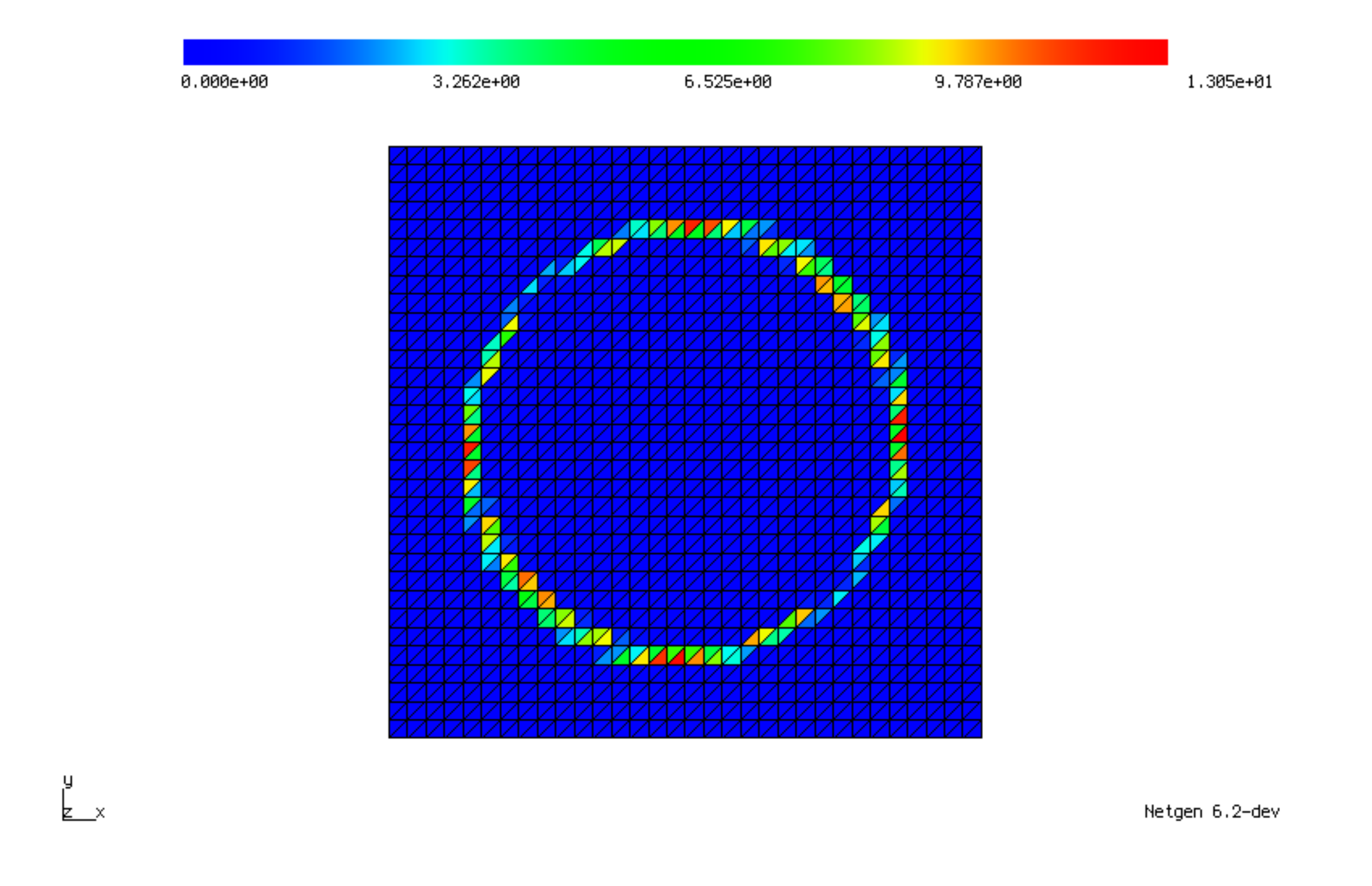} &
     \includegraphics[width=.33\textwidth, trim=100 0 50 0, clip]{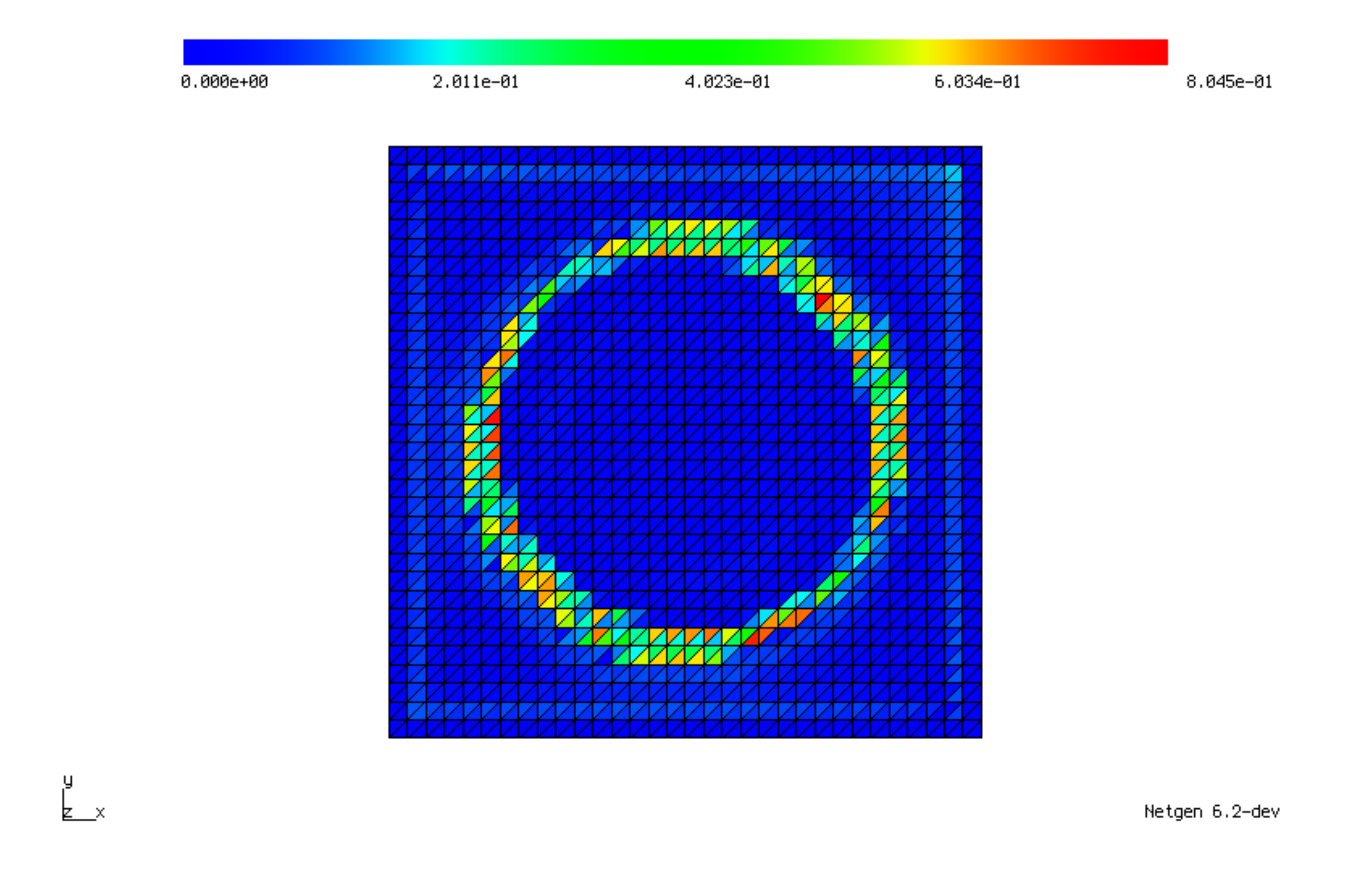} &
     \includegraphics[width=.33\textwidth, trim=100 0 50 0, clip]{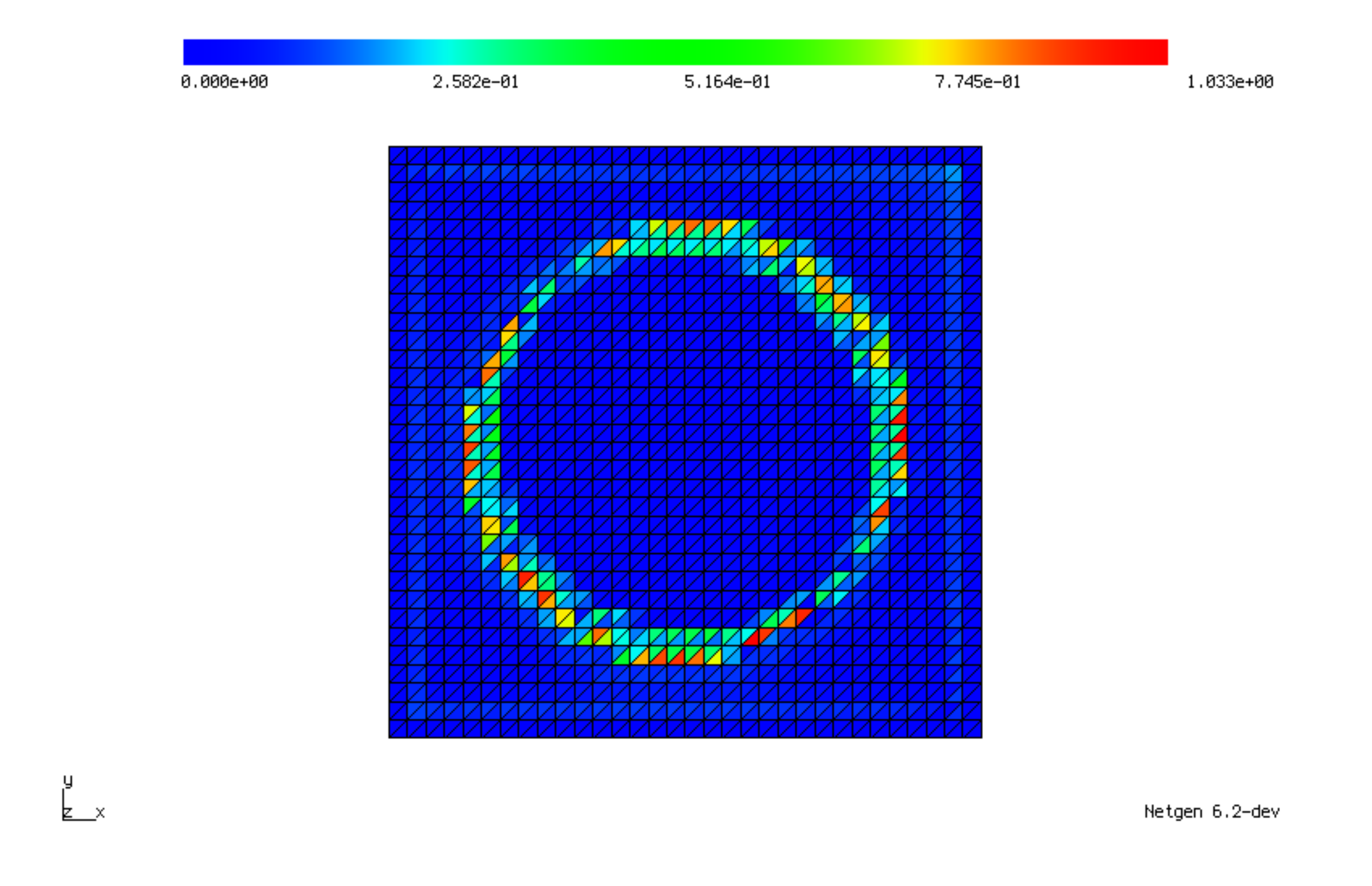} \\
     (a) $\alpha = 1$ & (b) $\alpha = -0.2$ & (c) $\alpha = -0.1$
    \end{tabular}
    \caption{Maximum relative error $\delta \hat{\mathcal J}_{\text{SMWapprox}}$ over computational domain for different averaging parameters $\alpha$ in \eqref{eq_lambdaSj}.}
    \label{fig_errorDomain_inhomo_different_alpha}
\end{figure}

\section*{Conclusion and Outlook}
In this paper, we introduced and examined different separable approximations to a discretized topology/material optimization problem. The Sherman-Morrison-Woodbury formula applied to the perturbed finite element stiffness matrix yielded a first separable exact model which, however, is prohibitively expensive to evaluate. A diagonal approximation of the stiffness matrix yielded a first tractable model. We introduced a model that is motivated by the continuous concept of topological derivatives for triangular inclusion shapes. Moreover, we also introduced a model that approximates the Sherman-Morrison-Woodbury model with high accuracy by performing similar rescaling steps as in the topological derivative model. Subsequently, we showed the somewhat surprising result that these latter two models coincide. Finally, we compared the performances of all models numerically. While the diagonal approximation of the Sherman-Morrison-Woodbury model can be evaluated very efficiently without any problems, the new models need data to be precomputed in an offline phase. In our model problem, however, we saw that the newly introduced models show significantly higher accuracy in most regions of the domain.

This work presented here can be extended and continued in several directions.
\begin{itemize}
 \item We illustrated our methods for the case of the compliance functional in a stationary heat equation. We emphasize that this model was chosen for compactness of presentation and that extensions to other cost functions and other linear PDE constraints (e.g., linear elasticity) can be obtained in a rather straight-forward way (possibly yielding slightly more technical formulas). An extension to non-selfadjoint problems (including also nonlinear cost functions) could be realized taking into account Remark \ref{rem_cost_nonselfadj}. An extension of models based on the Sherman-Morrison-Woodbury formula to other linear PDE constraints is straightforward since the structure of the discretized problem is the same as for our model problem. The topological derivative model can be extended to other PDE constraints following the general systematic procedure presented in \cite{GanglSturm_TDauto}.

 \item In this paper, we always assumed a structured mesh of a certain mesh topology to be given. While this is a common assumption made in many publications on topology optimization, an extension to general meshes with arbitrary element shapes and sizes would be an interesting topic of future research. In this setting, one might want to parametrize the shape of triangles. Then one could perform the precomputation for a (small) number of sample triangle shapes and interpolate their data in order to treat a family of element shapes.

 \item The extension of the proposed approaches to nonlinear PDE constraints such as nonlinear elasticity or nonlinear magnetostatics is another interesting yet challenging task. Also, here, the general procedure for obtaining topological derivatives \cite{GanglSturm_TDauto} could be used to establish a model similar to $\hat{\mathcal J}_{\text{TDnum}}$.

 \item Finally, the ultimate goal of this research is to obtain good approximate sub-problems in an iterative optimization algorithm. While solving the actual optimization problem was beyond the scope of this paper and subject of future research, we mention that this can be carried out in a similar way to \cite{NeesEtAl2022}. In particular, in \cite{NeesEtAl2022} it was shown that a sequential global programming approach with a diagonal approximation of a Sherman-Morrison-Woodbury model was superior to the well-established method of moving asymptotes (MMA) \cite{Svanberg1987} in terms of both number of optimization iterations and quality of obtained solutions. A similar or even better behavior is expected when replacing the diagonal approximation model to our model $\hat{\mathcal J}_{\text{SMWapprox}}$.
\end{itemize}

\paragraph{Acknowledgements:}
This work has been funded by the Deutsche Forschungsgemeinschaft (DFG, German Research Foundation) -- SFB 1411 (project ID 416229255) and  SFB 814 (project ID 61375930).
The work of P. Gangl is supported by the joint DFG/FWF Collaborative Research Centre CREATOR (CRC -- TRR361/F90) at TU Darmstadt, TU Graz and JKU Linz.

 \bibliographystyle{plain}
 \bibliography{sgptd}

\end{document}